\documentclass[10pt]{amsart}
\usepackage[all]{xy}
\usepackage{enumerate}
\usepackage{mathrsfs}
\usepackage{amssymb}
\usepackage{amsmath}
\usepackage{amsfonts}
\usepackage{graphicx}
\usepackage{color}
\usepackage{url}
\usepackage{hyperref}
\usepackage{environ}

\numberwithin{equation}{subsection}

\theoremstyle{plain}
\newtheorem{theorem}{Theorem}[section]
\newtheorem{lemma}[theorem]{Lemma}
\newtheorem{corollary}[theorem]{Corollary}
\newtheorem{proposition}[theorem]{Proposition}

\newtheorem{theoremX}{Theorem}

\theoremstyle{definition}
\newtheorem{definition}[theorem]{Definition}
\newtheorem{example}[theorem]{Example}

\theoremstyle{remark}
\newtheorem{remark}[theorem]{Remark}

\newtheorem{claim}{Claim}

\newcommand\Aut{\operatorname{Aut}}
\newcommand\Hom{\operatorname{Hom}}
\newcommand\End{\operatorname{End}}
\newcommand\GL{\operatorname{GL}}
\newcommand\id{\operatorname{id}}

\newcommand\op{{\operatorname{op}}}
\newcommand\sgn{\operatorname{sgn}}

\newcommand\Z{\mathbb{Z}}
\newcommand\C{\mathcal{C}}
\renewcommand\k{\mathbf{k}}   
\newcommand\Q{\mathbb{Q}}

\newcommand\gpS{\mathfrak{S}}

\newcommand\nc{\newcommand}
\nc\rnc{\renewcommand}

\nc\xto[1]{{\overset{#1}{\longrightarrow}}}
\nc\bfA{\mathbf{A}} \nc\bbA{\mathbb{A}} \nc\calA{\mathcal{A}} \nc\cA{\calA}
\nc\bfB{\mathbf{B}} \nc\bbB{\mathbb{B}} \nc\calB{\mathcal{B}} \nc\cB{\calB}
\nc\bfC{\mathbf{C}} \nc\bbC{\mathbb{C}} \nc\calC{\mathcal{C}} \nc\cC{\calC}
\nc\bfD{\mathbf{D}} \nc\bbD{\mathbb{D}} \nc\calD{\mathcal{D}} \nc\cD{\calD}
\nc\bfE{\mathbf{E}} \nc\bbE{\mathbb{E}} \nc\calE{\mathcal{E}} \nc\cE{\calE}
\nc\bfF{\mathbf{F}} \nc\bbF{\mathbb{F}} \nc\calF{\mathcal{F}} \nc\cF{\calF}
\nc\bfG{\mathbf{G}} \nc\bbG{\mathbb{G}} \nc\calG{\mathcal{G}} \nc\cG{\calG}
\nc\bfH{\mathbf{H}} \nc\bbH{\mathbb{H}} \nc\calH{\mathcal{H}} \nc\cH{\calH}
\nc\bfI{\mathbf{I}} \nc\bbI{\mathbb{I}} \nc\calI{\mathcal{I}} \nc\cI{\calI}
\nc\bfJ{\mathbf{J}} \nc\bbJ{\mathbb{J}} \nc\calJ{\mathcal{J}} \nc\cJ{\calJ}
\nc\bfK{\mathbf{K}} \nc\bbK{\mathbb{K}} \nc\calK{\mathcal{K}} \nc\cK{\calK}
\nc\bfL{\mathbf{L}} \nc\bbL{\mathbb{L}} \nc\calL{\mathcal{L}} \nc\cL{\calL}
\nc\bfM{\mathbf{M}} \nc\bbM{\mathbb{M}} \nc\calM{\mathcal{M}} \nc\cM{\calM}
\nc\bfN{\mathbf{N}} \nc\bbN{\mathbb{N}} \nc\calN{\mathcal{N}} \nc\cN{\calN}
\nc\bfO{\mathbf{O}} \nc\bbO{\mathbb{O}} \nc\calO{\mathcal{O}} \nc\cO{\calO}
\nc\bfP{\mathbf{P}} \nc\bbP{\mathbb{P}} \nc\calP{\mathcal{P}} \nc\cP{\calP}
\nc\bfQ{\mathbf{Q}} \nc\bbQ{\mathbb{Q}} \nc\calQ{\mathcal{Q}} \nc\cQ{\calQ}
\nc\bfR{\mathbf{R}} \nc\bbR{\mathbb{R}} \nc\calR{\mathcal{R}} \nc\cR{\calR}
\nc\bfS{\mathbf{S}} \nc\bbS{\mathbb{S}} \nc\calS{\mathcal{S}} \nc\cS{\calS}
\nc\bfT{\mathbf{T}} \nc\bbT{\mathbb{T}} \nc\calT{\mathcal{T}} \nc\cT{\calT}
\nc\bfU{\mathbf{U}} \nc\bbU{\mathbb{U}} \nc\calU{\mathcal{U}} \nc\cU{\calU}
\nc\bfV{\mathbf{V}} \nc\bbV{\mathbb{V}} \nc\calV{\mathcal{V}} \nc\cV{\calV}
\nc\bfW{\mathbf{W}} \nc\bbW{\mathbb{W}} \nc\calW{\mathcal{W}} \nc\cW{\calW}
\nc\bfX{\mathbf{X}} \nc\bbX{\mathbb{X}} \nc\calX{\mathcal{X}} \nc\cX{\calX}
\nc\bfY{\mathbf{Y}} \nc\bbY{\mathbb{Y}} \nc\calY{\mathcal{Y}} \nc\cY{\calY}
\nc\bfZ{\mathbf{Z}} \nc\bbZ{\mathbb{Z}} \nc\calZ{\mathcal{Z}} \nc\cZ{\calZ}

\nc\M{\calM}

\nc\simeqto{\overset{\simeq}{\longrightarrow }}
\nc\hr{\medskip\hrule\medskip}
\nc\trl{\triangleleft}
\nc\trr{\triangleright}

\nc\Ob{\operatorname{Ob}}

\nc\ot{\otimes}

\nc\ol{\overline}
\nc\congto{\overset{\cong}{\to}}

\nc\Sp{\operatorname{Sp}}
\rnc\O{\operatorname{O}}

\nc\Vect{\mathbf{Vect}}

\nc\mmod{\text{-}\mathrm{mod}}

\rnc\S{\mathbb{S}}

\nc\lara[1]{\langle #1 \rangle}

\nc\D{\calD}

\nc\Sbim{\S\text{-}\mathrm{bim}}

\nc\A{\calA}

\nc\idC[1]{\id_{\C^{#1}}}

\nc\Abim{\calA\text{-}\mathrm{bim}}
\nc\Cbim{\calC\text{-}\mathrm{bim}}

\nc\Kar{\operatorname{Kar}}
\nc\hC{\hat\C}

\nc\B{\mathbb{B}}
\nc\mB{mB}
\nc\qB{qB}
\nc\pmB{pB}

\nc\bRep{\underline{\mathrm{Rep}}}
\nc\bRepOd{\bRep(O_\delta)}

\nc\hB{\widehat B}
\nc\ee{{\mathbf{e}}}
\nc\adm{\mathit{a}}
\nc\sadm{\mathit{sa}}

\nc\eex{\ee_x}

\nc\Ataumod{A_\tau\mmod}
\nc\Abmod{A_b\mmod}
\nc\Imod{I\mmod}

\nc\Rbim{R\text{-bim}}
\nc\add{\operatorname{add}}
\nc\rank{\operatorname{rank}_\k}

\nc\Act{\operatorname{Act}}
\nc\Mon{\operatorname{Mon}}
\nc\SA{\mathcal{S}_\A}
\nc\TA{\mathcal{T}_\A}
\nc\bc{\overline{c}}
\nc\cprime{$'$}

\title[Quasi-cellular categories and the Brauer category]{Quasi-cellular categories and the structure of the Brauer category}

\author{Kazuo Habiro}
\author{Mai Katada}
\address{Graduate School of Mathematical Sciences, University of Tokyo, Tokyo 153-8914, Japan}
\email{habiro@ms.u-tokyo.ac.jp}
\email{mkatada@ms.u-tokyo.ac.jp}

\date{September 28, 2026 (First version)}
\subjclass[2020]{Primary 16G30, 18E05; Secondary 16D90, 18M05, 20G05}

\keywords{Brauer category, Brauer algebra, Quasi-cellular category, Stratified linear category, Morita equivalence, Orthogonal groups, Symplectic groups}

\begin{document}
\maketitle

\begin{abstract}
We introduce the notion of a \emph{quasi-cellular category}, which is a variant of cellular algebras and cellular categories.
Many diagram categories, such as the Brauer category and the partition category, are quasi-cellular.

We study the Brauer category $B$ over a commutative ring $\k$ with parameter $\delta$ through its quasi-cellular structure.
We construct a linear functor $L:B\to\mB$ to the \emph{matrix Brauer category} $\mB$, which has the same objects and hom-spaces as $B$ but a matrix-like composition.
For a non-singular parameter $\delta$, we show that $L$ is an isomorphism of linear categories, and that
$B$ is Morita equivalent to its subcategory $\S$ spanned by permutations, whose endomorphism algebras are the group algebras of symmetric groups. These results generalize those of Brown and K\"onig and Xi for the Brauer algebras. 
\end{abstract}

\setcounter{tocdepth}{1}
\tableofcontents

\setcounter{section}{-1}
\section{Introduction}
The \emph{Brauer algebra} was introduced by Brauer \cite{Brauer} in his study of the invariant theory of the orthogonal and symplectic groups.
In Schur--Weyl duality for the general linear group, the group algebra of the symmetric group acts on tensor powers, and its image is the centralizer algebra of the general linear group.  For the orthogonal and symplectic groups, the analogous role is played by the Brauer algebra.

More precisely, let $V$ be a finite-dimensional vector space equipped with a non-degenerate symmetric or skew-symmetric bilinear form. 
Then there are commuting actions on the tensor power $V^{\ot p}$ of the corresponding isometry group $\O(V)$ or $\Sp(V)$ and of the Brauer algebra $B(p)=B^\delta(p)$. Here $\delta=\dim V$  in the symmetric case and $\delta=-\dim V$ in the skew-symmetric case.

The representation theory of Brauer algebras has been extensively studied.
In characteristic $0$, the semisimplicity problem is completely understood:
Brown \cite{Brown1,Brown2} treated the case $\delta\in\Z_{>0}$
and showed that the Brauer algebra, when semisimple, is isomorphic to a direct sum of matrix algebras over the symmetric group algebras.
Hanlon and Wales \cite{Hanlon-Wales} formulated a semisimplicity conjecture for $\delta\notin\Z$, which was proved by Wenzl \cite{Wenzl}. Subsequent work of Doran--Wales--Hanlon \cite{Doran-Wales-Hanlon}, Rui \cite{Rui} and Rui--Si \cite{Rui-Si-II} gave a criterion for semisimplicity of the Brauer algebra.  The theory of cellular algebras of Graham--Lehrer \cite{Graham-Lehrer} provides a natural framework for the study of the Brauer algebra. K\"onig and Xi \cite{Koenig-Xi} showed that, in the non-singular case, the Brauer algebra is Morita equivalent to a direct sum of group algebras of symmetric groups.

The Brauer category $B=B^{\k,\delta}$ over a commutative ring $\k$ with parameter $\delta\in\k$ has as objects non-negative integers and as morphisms from $p$ to $q$ the $\k$-linear combinations of Brauer diagrams with $p$ top points and $q$ bottom points. The Brauer algebra is recovered as the endomorphism algebra $B(p)=B(p,p)$.
This category was studied systematically by Lehrer--Zhang \cite{Lehrer-Zhang}, and it is also the diagrammatic input in Deligne's construction of the interpolation category $\underline{\operatorname{Rep}}(O_\delta)$ \cite{Deligne, Deligne-Lie}.

The purpose of this paper is twofold.
First, we introduce the notion of a \emph{quasi-cellular category}, which is a categorical variant of the cellular structure on algebras, and develop its general theory, including a criterion for Morita equivalence with the base category.
Second, we apply this theory to the Brauer category and give a categorical explanation of some of the results on the Brauer algebras described above, in particular Brown's  matrix decomposition of semisimple Brauer algebras and the Morita equivalence of K\"onig--Xi.

\subsection*{Quasi-cellular categories}

Let $\k$ be a commutative ring.
By a \emph{base category}, we mean a $\k$-linear category $\A$ such that $\Ob(\A)$ is a well-founded poset and such that $\A(x,y)=0$ for $x\neq y$.
Thus, $\A$ amounts to a family of $\k$-algebras $\A_x=\A(x,x)$ indexed by the poset $\Ob(\A)$.
A \emph{stratified linear category} over $\A$ is a linear category $\C$ with $\Ob(\C)=\Ob(\A)$ equipped with two wide subcategories $\C^+$ and $\C^-$, called the subcategories of \emph{upward} and \emph{downward} morphisms, satisfying the following conditions. First, $\C^+(x,y)=0$ unless $x\le y$, and $\C^-(x,y)=0$ unless $x\ge y$. Second, $\C^+(x,x)=\C^-(x,x)=\A_x$ for $x\in\Ob(\A)$. Third, composition induces a linear isomorphism
\begin{gather}\label{intro-composition}
    \bigoplus_{y\in\Ob(\A)}\C^+(y,z)\ot_{\A_y}\C^-(x,y)\xrightarrow{\cong}\C(x,z)
\end{gather}
for $x,z\in\Ob(\A)$.
The term ``stratified'' is borrowed from the theory of stratified algebras
initiated by Cline--Parshall--Scott \cite{CPS96}. Related structures appear in many places;
see Remark~\ref{rem:notion}.

We regard $\C$ as a monoid in the monoidal category $\Abim$ of $\A$-bimodules.
The isomorphism \eqref{intro-composition} identifies the $\A$-bimodule $\C$ with $\C^+\ot_\A\C^-$. Under this identification, the composition $\mu$ of $\C$ is recovered as 
\[
\mu=\mu_\tau=(\mu^+\ot\mu^-)(\id_{\C^+}\ot\tau\ot\id_{\C^-})
\]
from the compositions $\mu^\pm$ of $\C^\pm$ and the \emph{twist}
\[
\tau:\C^-\ot_\A\C^+\to\C^+\ot_\A\C^-,
\]
which sends $f^-\ot f^+$ to the decomposition of $f^-f^+$ via \eqref{intro-composition}.
Conversely, for any twist $\tau$, the twisted tensor product $\C^+\ot_\tau\C^-=(\C^+\ot_\A\C^-,\mu_\tau)$ of an upward monoid $\C^+$ and a downward monoid $\C^-$ in $\Abim$, with $\A_x\to\C^\pm(x,x)$ an isomorphism for every $x\in\Ob(\A)$, is a stratified linear category over $\A$. This gives a one-to-one correspondence between stratified linear categories over $\A$ and twisted tensor products in $\Abim$ (Theorem~\ref{stratifiedlinearcategoryandtwistedtensorproduct}).
Section~\ref{sec-twisted-tensor-product} develops the theory of twisted tensor products of augmented monoids in a general monoidal category, and the results on stratified linear categories in Sections~\ref{sectionstrlincat}--\ref{sectionqc} are obtained by specializing it to $\Abim$.

The twist $\tau$ determines the \emph{pairing}
\[
b=b_\tau=(\varepsilon^+\ot\varepsilon^-)\tau:\C^-\ot_\A\C^+\to\A,
\]
where $\varepsilon^\pm:\C^\pm\to\A$ are the augmentations, which are the identity on $\C^\pm(x,x)=\A_x$ and $0$ on $\C^\pm(x,y)$ for $x\neq y$.
We say that $\C$ is \emph{perfect} if $b$ is perfect, i.e., if $b$ admits a \emph{copairing} $d:\A\to\C^+\ot_\A\C^-$ in $\Abim$ satisfying the zigzag identities with $b$.

A \emph{quasi-cellular category} over $\A$ is a stratified linear category $\C$ over $\A$ equipped with an auxiliary perfect pairing
\[
b':\C^-\ot_\A\C^+\to\A
\]
with copairing $d':\A\to\C^+\ot_\A\C^-$, compatible with the augmentations in the sense of \eqref{qc-eta-epsilon}.
The pair $(b',d')$ defines a second monoid $\C_{b'}=\C^+\ot_{b'}\C^-$ in $\Abim$ with underlying $\A$-bimodule $\C^+\ot_\A\C^-$, multiplication $\id_{\C^+}\ot b'\ot\id_{\C^-}$ and unit $d'$.
The linear category $\C_{b'}$ is of ``matrix type'': its composition is induced by the perfect pairing $b'$ alone. We have a canonical linear functor
\[
L:\C\to\C_{b'}.
\]

\begin{theoremX}[cf. Theorems~\ref{MoritaC} and~\ref{precontraction2} and Proposition~\ref{L-iso}]
\label{introthmqc}
Let $\C$ be a stratified linear category over $\A$.
Then conditions (1) and (2) below are equivalent.
\begin{enumerate}
\item $\C$ is perfect.
\item The pair $(\C^+,\C^-)$ gives a Morita equivalence between $\A$ and $\C$, i.e., we have $\C^+\ot_\A\C^-\cong\C$ in $\Cbim$ and $\C^-\ot_\C\C^+\cong\A$ in $\Abim$.
\end{enumerate}
If, moreover, $\C$ is a quasi-cellular category, then (1) and (2) imply the following condition.
\begin{enumerate}
\item[(3)] The functor $L:\C\to\C_{b'}$ is an isomorphism of linear categories.
\end{enumerate}
If, in addition, $\C$ admits an ``action'' (see Section~\ref{subsectionqcpsitilde}), then (3) is equivalent to 
(1) and (2).
\end{theoremX}

The term ``quasi-cellular'' is motivated by the following relation to cellular structures.
The cellular algebras of Graham--Lehrer \cite{Graham-Lehrer} were generalized to cellular categories by Westbury \cite{Westbury}, which were in turn rigidified by Elias--Lauda \cite{Elias-Lauda} to strictly object-adapted cellular categories.
In a cellular structure, there is an anti-involution of the algebra or the category. 
A quasi-cellular category has no such involution in general, but the auxiliary pairing $b'$ plays a similar role.

Examples of quasi-cellular categories include some diagram categories such as
the walled Brauer category and the partition category; see \cite{Comes-Ostrik,BCNR}.

In this paper, we focus on the Brauer category, and we plan to study other quasi-cellular categories in future work.

\subsection*{The Brauer category as a quasi-cellular category}

Let $\S$ denote the wide subcategory of $B$ whose morphisms are linear combinations of permutations.
Then $\S$ is a base category with $\S_p=\k[\gpS_p]$ for $p\ge0$ and with the partial order $\preceq$ defined by 
\[
r\preceq p \quad\Leftrightarrow\quad
0\le r\le p\text{ and }p-r\in2\Z.
\]
Let $B^+$ and $B^-$ be the wide subcategories of $B$ whose morphisms are linear combinations of upward and downward Brauer diagrams, respectively (Section~\ref{sectionBrauercategory}).
Since every Brauer diagram factors as a downward diagram followed by an upward diagram, uniquely up to a permutation in the middle, $B$ is a stratified linear category over $\S$.
This decomposition of $B$ was used by Sam and Snowden \cite{Sam-Snowden-BrauerI} in their study of the representation theory of the Brauer category.
The twist $\tau:B^-\ot_\S B^+\to B^+\ot_\S B^-$ encodes the composition of a downward diagram with an upward one, and the pairing $b=(\varepsilon^+\ot\varepsilon^-)\tau$ extracts from the composite $f^-f^+$ its permutation part, in which each closed loop contributes a factor $\delta$.

The Brauer category admits a quasi-cellular structure $(b',d')$ over $\S$.
The pairing $b'(f^-\ot f^+)$ 
for a downward Brauer diagram $f^-$ and an upward Brauer diagram $f^+$
is nonzero precisely when $f^-=\sigma\overline{f^+}$ for some permutation $\sigma$, where $\overline{f^+}$ is the morphism obtained from $f^+$ by interchanging the source and the target.
In that case, we have $b'(f^-\ot f^+)=\sigma$.
We call the linear category $\mB=B_{b'}$ the \emph{matrix Brauer category}.
Its hom-spaces decompose as direct sums of matrix modules over group algebras of symmetric groups: for $p,q\ge0$, we have
\[
\mB(p,q)\cong \bigoplus_{r\preceq p,q} M_{d_{q,r},d_{p,r}}(\S_r),
\]
where $d_{p,r}=\rank B^+(r,p)/r!$ (see Proposition~\ref{matrixBrauersemisimple}).
Consequently, the \emph{matrix Brauer algebra} $\mB(p):=\mB(p,p)$ is a direct sum of matrix algebras over group algebras of symmetric groups.
Unlike the Brauer category, the matrix Brauer category does not depend on the parameter $\delta$.

The linear functor $L:\C\to\C_{b'}$ defined for a quasi-cellular category $\C$ specializes to a linear functor 
\[
L: B\to \mB.
\]

We say that the parameter $\delta$ is \emph{non-singular} 
if for every $p,q\in\Ob(B)$ with $p\preceq q$, a certain matrix $\iota(A_{p,q})$ induced from the pairing $b:B^-(q,p)\ot_{\S_q} B^+(p,q)\to\S_p$ has an invertible determinant (see Section~\ref{sectionpairingofB}), or equivalently (Theorem~\ref{Brauerperfectnew}) if the Brauer category $B$ is perfect as a quasi-cellular category.
For $r\ge0$, let $B_{\preceq r}$ denote the full subcategory of $B$ on the objects $q$ with $q\preceq r$, and similarly for $\mB_{\preceq r}$ and $\S_{\preceq r}$.
Then $B_{\preceq r}$ is again a quasi-cellular category over $\S_{\preceq r}$. We say that $\delta$ is \emph{non-singular at ranks $\preceq r$} if the matrix $\iota(A_{p,q})$ has  invertible determinant for all $p\preceq q\preceq r$, or equivalently if $B_{\preceq r}$ is perfect.
Over a field, this notion agrees with the non-singularity condition of K\"onig--Xi \cite{Koenig-Xi} for the Brauer algebra $B(r)$; see Section~\ref{sectionpairingofB}.

The singular parameters are known explicitly.
Rui \cite{Rui} and Rui--Si \cite{Rui-Si-II} determined a finite set $X_r\subset\Z$ for each $r\ge0$ such that, over a field of characteristic $0$, $\delta$ is singular at ranks $\preceq r$ if and only if $\delta\in X_r$ (see Proposition~\ref{RuiSisingularset} for the explicit description of $X_r$).
We show that the same set governs non-singularity over an arbitrary commutative ring: $\delta\in\k$ is non-singular at ranks $\preceq r$ if and only if $\delta-i(m)$ is invertible in $\k$ for every $m\in X_r$, where $i:\Z\to \k$ is the unique ring homomorphism (Proposition~\ref{Rui-Si-singularset2}).

Theorem~\ref{introthmqc} applied to $B_{\preceq r}$ yields the following.

\begin{theoremX}[cf. Theorems~\ref{isomtruncatedBrauers} and~\ref{MoritaequivalencetruncatedBrauer} and Proposition~\ref{BMoritaequivalenttoS-fieldofchar0}]
\label{introthmB}
Let $r\ge0$.
Then the following conditions are equivalent.
\begin{enumerate}
\item $\delta$ is non-singular at ranks $\preceq r$.
\item The pair $(B^+_{\preceq r},B^-_{\preceq r})$ gives a Morita equivalence between $B_{\preceq r}$ and $\S_{\preceq r}$.
\item The restriction $L_{\preceq r}:B_{\preceq r}\to\mB_{\preceq r}$ of the linear functor $L:B\to\mB$ is an isomorphism.
\end{enumerate}
Moreover, if $\k$ is a field of characteristic $0$, then 
(1), (2) and (3) are equivalent to the following condition.
\begin{enumerate}
\item[(4)] $B_{\preceq r}$ and $\S_{\preceq r}$ are Morita equivalent.
\end{enumerate}
\end{theoremX}

K\"onig and Xi \cite{Koenig-Xi} proved that, over a field, if $\delta$ is non-singular, then the Brauer algebra $B(r)$ is Morita equivalent to a direct sum of group algebras of symmetric groups.
The Morita equivalence in Theorem~\ref{introthmB} is a categorical form of this result, and the isomorphism $L_{\preceq r}$ refines it to an explicit isomorphism $B(r)\cong\mB(r)$ of algebras.

\subsection*{Applications to tensor spaces}
We apply the above results to tensor spaces.
Let $\k$ be a field of characteristic $0$, and let $V$ be a finite-dimensional vector space over $\k$ equipped with a non-degenerate bilinear form $\omega$ that is either symmetric or skew-symmetric.
Set $e=+1$ if $\omega$ is symmetric and $e=-1$ if $\omega$ is skew-symmetric.
Let $G(V)$ be the subgroup of $\GL(V)$ preserving the form $\omega$.
Recall that the Brauer category $B=B^{\k,e\dim V}$ acts on the tensor spaces $V^{\ot p}$, $p\ge0$, in a canonical way \cite{Lehrer-Zhang}; see also Section~\ref{sec:tensor-spaces}.
For $p\ge0$, let $V^{\lara p}$ denote the traceless part of the tensor space $V^{\otimes p}$.

The following theorem gives a canonical, functorial version of Weyl's decomposition of tensor spaces in terms of traceless tensors.

\begin{theoremX}[Theorem~\ref{Phi-Psi}]
\label{Phi-Psi-intro}
For every $q\ge0$, we have a canonical injective $G(V)$-module map
\begin{gather*}
    \Phi^V_q: V^{\ot q}\to\bigoplus_{p\preceq q}B^+(p,q)\otimes_{\S_p}V^{\lara p}.
\end{gather*}
Moreover, the following hold.
\begin{enumerate}
    \item The maps $\Phi^V_q$, $q\ge0$, form a natural transformation between functors from $B$ to the category of $G(V)$-modules.
    \item If $\delta=e\dim V$ is non-singular at ranks $\preceq q$, then $\Phi^V_q$ is an isomorphism.
\end{enumerate}
\end{theoremX}

From Theorem~\ref{Phi-Psi-intro}, we obtain a functorial form of the decomposition of $V^{\ot q}$ into irreducibles, which may be regarded as a categorical enhancement of Schur--Weyl duality for the orthogonal and symplectic groups (Proposition~\ref{prop:SWB}).

\begin{remark}
The authors' interest in the Brauer category comes from the study of graph complexes.
Kontsevich \cite{Kontsevich93, Kontsevich94} and Conant--Vogtmann \cite{Conant-Vogtmann} identified the stable homology of the symplectic derivation Lie algebra associated to a cyclic operad with the homology of a graph complex, and the proof passes through the invariant theory of the symplectic group on tensor spaces.
In a forthcoming paper, we will apply Theorem~\ref{Phi-Psi-intro} to the study of these graph complexes, as well as of the hairy graph complexes of \cite{Conant-Kassabov-Vogtmann}.

There is an analogous picture for the general linear groups. The directed graph complexes of Merkulov \cite{Merkulov-wheels, Merkulov-exotic} (see also \cite{Willwacher-oriented}) compute the stable homology of the derivation Lie algebras of free algebras over operads \cite{Dotsenko}. In this setting, the role of the Brauer category is played by the walled Brauer category \cite{Powell-walled, Powell-relating}, which is also quasi-cellular.
\end{remark}

\subsection*{Organization}
We organize the rest of this paper as follows.
In Section~\ref{sec-twisted-tensor-product}, we develop the formalism of twisted tensor products of augmented monoids in a strict monoidal category.
In Section~\ref{sectionstrlincat}, we introduce stratified linear categories, and identify them as twisted tensor products in the category of bimodules over a base category.
In Section~\ref{sec:Perfect stratified linear categories}, we prove basic structural results on stratified linear categories, including results on stack-invertibility, Morita equivalence, centers, and canonical idempotents.
In Section~\ref{sectionqc}, 
we define quasi-cellular categories as stratified linear categories equipped with an auxiliary perfect pairing.
In Section~\ref{sectionBrauercategory}, we recall the definition of the Brauer category.
In Section~\ref{sectionBrauerasqc}, we study the quasi-cellular structure of the Brauer category and introduce its variants, the matrix Brauer category $\mB$ and the quasi-Brauer category $\qB$.
In Sections~\ref{sectionpairingofB} and~\ref{sectionPerfectnessBrauer}, we study perfectness of the Brauer category and its consequences.
In Section~\ref{sec:tensor-spaces}, we discuss applications to tensor representations of the orthogonal and symplectic groups.
In Section~\ref{sectionBasereplacement}, we introduce base replacements of quasi-cellular categories, and apply them to the Brauer category. This gives a quasi-cellular structure on a subcategory of Deligne's category $\bRepOd$.

\subsection*{Conventions}
Throughout this paper, rings and algebras are assumed to be associative and unital, and ring homomorphisms preserve the unit. Unless otherwise stated, $\k$ denotes a commutative ring, and ``linear'' means $\k$-linear.

\subsection*{Acknowledgments}
K. H. was supported in part by JSPS KAKENHI Grant Number 22K03311.
M. K. was supported by JSPS KAKENHI Grant Number 24K16916.

\section{Twisted tensor products in monoidal categories}
\label{sec-twisted-tensor-product}

We work in a strict monoidal category and study twisted tensor products of augmented monoids, their relation to perfect pairings, Morita equivalence with the unit, and centers. These results will later be applied to stratified and quasi-cellular categories.

Although we work here in the setting of a strict monoidal category, the arguments extend straightforwardly to strict monoidal $\k$-linear categories over a commutative ring $\k$.

\subsection{Preliminaries}
\label{Preliminaries}

Let $\calM=(\calM,\otimes,I)$ be a (strict) monoidal category.
A \emph{monoid} $A=(A,\mu,\eta)$ in $\calM$ consists of an object $A$ and morphisms $\mu:A\ot A\to A$ and $\eta:I\to A$ such that $\mu(\mu\ot A)=\mu( A\ot\mu)$ and $\mu(\eta\ot A)=\id_A=\mu( A\ot\eta)$.
An \emph{augmented monoid} $A=(A,\mu,\eta,\varepsilon)$ in $\calM$ consists of a monoid $(A,\mu,\eta)$ and a morphism $\varepsilon:A\to I$ such that $\varepsilon\eta=\id_I$ and $\varepsilon\mu=\varepsilon\ot\varepsilon$.

For a monoid $A=(A,\mu,\eta)$ in $\calM$, a \emph{left $A$-module} $X=(X,\alpha_X)$ in $\calM$ consists of an object $X$ and a morphism $\alpha_X:A\ot X\to X$ such that $\alpha_X (\mu\ot X)=\alpha_X(A\ot \alpha_X)$ and $\alpha_X(\eta\ot X)=\id_X$.
A \emph{right $A$-module} $X=(X,\alpha'_X)$ in $\calM$ is defined in a similar way.
An \emph{$A$-bimodule} $X=(X,\alpha_X,\alpha'_X)$ consists of a left $A$-module $(X,\alpha_X)$ and a right $A$-module $(X,\alpha'_X)$ such that $\alpha_X(A \ot \alpha'_X)=\alpha'_X(\alpha_X\ot A)$.
By a morphism $F:(X,\alpha_X)\to (Y,\alpha_Y)$ of left $A$-modules in $\calM$, we mean a morphism $F:X\to Y$ in $\calM$ such that $F\alpha_X=\alpha_Y(A\ot F)$.
Let $A\text{-mod}$ denote the category of left $A$-modules and $A$-module morphisms in $\calM$.
In a similar way, we define a morphism of $A$-bimodules, and let $A\text{-bim}$ denote the category of $A$-bimodules and $A$-bimodule morphisms in $\calM$.

For two objects $X$ and $Y$ in $\calM$, a morphism $b:X\ot Y\to I$ is said to be a  \emph{perfect pairing} if there exists a morphism $d:I\to Y\ot X$ such that
$$
(b\ot X)(X\ot d)=\id_X,\quad
(Y\ot b)(d\ot Y)=\id_Y.
$$
If such a $d$ exists, then it is unique.
We call $d$ the \emph{copairing} of $b$, and a pair $(b,d)$ a \emph{duality} in $\calM$.
In this situation, we have a monoid
$Y\ot_b X=(Y\ot X,\mu_b,d)$ in $\calM$, where
$\mu_b=Y\ot b\ot X$.

Let $A^+=(A^+,\mu^+,\eta^+)$ and $A^-=(A^-,\mu^-,\eta^-)$ be monoids in $\calM$.
We have a monoid structure on the set $\calM(I,A^+\ot A^-)$ with ``stack multiplication'' given by
\begin{gather}\label{stack}
    g \bullet f = (\mu^+\ot\mu^-)( A^+\ot g\ot A^-)f,
\end{gather}
with unit $\eta^+\ot\eta^-$.
A morphism $d:I\to A^+\ot A^-$ is said to be \emph{stack-invertible} if it is invertible in the monoid $\calM(I, A^+\ot A^-)$, i.e., there exists a morphism $d^-:I\to A^+\ot A^-$ such that $d\bullet d^-=d^-\bullet d=\eta^+\ot\eta^-$.
The morphism $d^-$ is called the \emph{stack-inverse} of $d$, and it is unique if it exists.

\subsection{Twisted tensor products of monoids}

Let $A^+=(A^+,\mu^+,\eta^+)$ and $A^-=(A^-,\mu^-,\eta^-)$ be monoids in $\calM$.
A \emph{twist} for $A^+$ and $A^-$ is a morphism
$$\tau:A^-\otimes A^+\to A^+\otimes A^-$$ in $\calM$ such that 
\begin{gather}
\label{twist}
    \begin{split}
        \tau (\mu^-\otimes  A^+)&= ( A^+\otimes \mu^-)(\tau\otimes  A^-)( A^-\otimes \tau),\quad \tau(\eta^-\otimes  A^+)= A^+\otimes \eta^-,\\
        \tau ( A^-\otimes \mu^+)&= (\mu^+\otimes  A^-)( A^+\otimes \tau)(\tau\otimes  A^+),\quad \tau( A^-\otimes \eta^+)=\eta^+\otimes  A^-.
    \end{split}
\end{gather}
Then we have a monoid $A_\tau=A^+\ot_{\tau}A^-=(A^+\ot A^-,\mu_\tau,\eta_\tau)$, where
\begin{gather}\label{mu-eta}
    \mu_\tau=(\mu^+\otimes \mu^-)( A^+\otimes \tau\otimes  A^-),\quad
    \eta_\tau=\eta^+\ot\eta^-.
\end{gather}
The monoid $A_\tau$ is called the \emph{twisted tensor product} of $A^+$ and $A^-$ with twist $\tau$ \cite{VDVK94,Cap-et-al}.

Let $B^+$ and $B^-$ be another pair of monoids in $\calM$, and let $\tau':B^-\ot B^+\to B^+\ot B^-$
be a twist.
A \emph{morphism of twisted tensor products} from $A_\tau$ to $B_{\tau'}$ is a pair $(f^+,f^-)$ consisting of monoid morphisms $f^+:A^+\to B^+$ and $f^-:A^-\to B^-$ such that the following diagram commutes:
\begin{gather*}
    \xymatrix{
    A^-\ot A^+\ar[r]^\tau\ar[d]_{f^-\ot f^+}&A^+\ot A^-\ar[d]^{f^+\ot f^-}\\
    B^-\ot B^+\ar[r]_{\tau'}&B^+\ot B^-.
    }
\end{gather*}
Clearly, such a morphism gives rise to a monoid morphism $f=f^+\ot f^-:A_\tau\to B_{\tau'}$.
An \emph{isomorphism} of twisted tensor products, $(f^+,f^-):A_\tau\to B_{\tau'}$, is such a morphism with $f^+$ and $f^-$ isomorphisms.

\subsection{Twisted tensor products of augmented monoids}
Let $A^+=(A^+,\mu^+,\eta^+,\varepsilon^+)$ and
$A^-=(A^-,\mu^-,\eta^-,\varepsilon^-)$ be augmented monoids in $\calM$.
Let $\tau:A^-\ot A^+\to A^+\ot A^-$ be a twist.
We associate to $\tau$ a pairing
$$b=b_\tau=(\varepsilon^+\otimes \varepsilon^-)\tau: A^-\otimes A^+\to I.$$

\begin{lemma}\label{relationbrho}
    We have
\begin{gather}
    \label{eq0}
    (\varepsilon^+\ot b)(\tau\ot A^+)=b( A^-\ot\mu^+),\\
    \label{eq00}
    (b\ot \varepsilon^-)( A^-\ot\tau)=b(\mu^-\ot A^+),\\
    \label{eq01}
    b(\eta^-\ot A^+)=\varepsilon^+,\\
    \label{eq02}
    b( A^-\ot\eta^+)=\varepsilon^-.
\end{gather}
\end{lemma}

\begin{proof}
    We have \eqref{eq0} as follows:
    \begin{gather*}
    \begin{split}
    \text{LHS}
    &=(\varepsilon^+\ot\varepsilon^+\ot\varepsilon^-)( A^+\ot\tau)(\tau\ot A^+)\\
    &=(\varepsilon^+\ot\varepsilon^-)(\mu^+\ot A^-)( A^+\ot\tau)(\tau\ot A^+)\\
    &=(\varepsilon^+\ot\varepsilon^-)\tau( A^-\ot\mu^+)\\
    &=\text{RHS}.
    \end{split}
    \end{gather*}
    The identity  \eqref{eq00} is similarly proved.
    We have \eqref{eq01} as follows:
    \begin{gather*}
    \begin{split}
    \text{LHS}
    =(\varepsilon^+\ot\varepsilon^-)\tau(\eta^-\ot A^+)
    =\text{RHS}.
    \end{split}
    \end{gather*}
    The identity  \eqref{eq02} is similarly proved.
\end{proof}

The twist $\tau$ is said to be \emph{perfect} if its associated pairing $b=b_\tau$ is perfect.
Let $d$ be the copairing of $b$.
Then $d$ satisfies the following relations.

\begin{lemma}\label{lemmadrho}
    Let $\tau:A^-\ot A^+\to A^+\ot A^-$ be a perfect twist with duality $(b,d)$. 
    We have
\begin{gather}
    \label{eq2}
    ( A^+\ot(\varepsilon^+\ot A^-)\tau)(d\ot A^+)
    =(\mu^+\ot A^-)( A^+\ot d),\\
    \label{eq20}
    (( A^+\ot\varepsilon^-)\tau\ot A^-)( A^-\ot d)
    =( A^+\ot\mu^-)(d\ot A^-),\\
    \label{eq21}
    (\varepsilon^+\ot A^-)d
    =\eta^-,\\
    \label{eq22}
    ( A^+\ot\varepsilon^-)d
    =\eta^+.
\end{gather}
\end{lemma}

\begin{proof}
    We have \eqref{eq2} as follows:
    \begin{gather*}
    \begin{split}
    \text{LHS}
    &=( A^+\ot(b\ot A^-)( A^-\ot d)  (\varepsilon^+\ot A^-)\tau)(d\ot A^+)\\
    &=( A^+\ot (\varepsilon^+\ot b)(\tau\ot A^+) \ot  A^-)(d\ot A^+\ot d)\\
    &=( A^+\ot b( A^-\ot\mu^+) \ot  A^-)(d\ot A^+\ot d) \quad\text{(by \eqref{eq0})}\\
    &=(( A^+\ot b)(d\ot A^-)\mu^+\ot A^-)( A^+\ot d)\\
    &=\text{RHS}.
    \end{split}
    \end{gather*}
    The identity \eqref{eq20} is similarly proved.
    We have \eqref{eq21} as follows:
    \begin{gather*}
    \begin{split}
    \text{LHS}
    &=(\varepsilon^+\ot\varepsilon^-\ot A^-)(\tau\ot A^-)(\eta^-\ot d)\\
    &=(b\ot  A^-)(\eta^-\ot d)\\
    &=\text{RHS}.
    \end{split}
    \end{gather*}
    The identity \eqref{eq22} is similarly proved.
\end{proof}

In this case, we have another monoid structure $A_{b}=A^+\ot_{b}A^-$
on the tensor product $A^+\otimes A^-$; see Section~\ref{Preliminaries}.

Define a morphism $F_d: A^+\ot A^-\to A^+\ot A^-$ by
\begin{gather}
\label{F-rho}
F_d=(\mu^+\otimes \mu^-)( A^+\otimes d\otimes  A^-).
\end{gather}
Note that if $d$ is stack-invertible with stack-inverse $d^-$, 
then 
$F_d$ is an isomorphism with inverse 
\begin{gather}\label{Frho-}
    F_d^{-1}=F_{d^-}=(\mu^+\otimes \mu^-)( A^+\otimes d^-\otimes  A^-).
\end{gather}

\begin{proposition}\label{propF-rho}
Let $\tau:A^-\otimes A^+\to A^+\otimes A^-$ be a perfect twist with duality $(b,d)$.
Then the morphism $F_d$ defined in \eqref{F-rho} gives rise to a monoid morphism $F_d: A_\tau\to A_{b}$.
If, moreover, $d$ is stack-invertible,
then $F_d$ is a monoid isomorphism.
\end{proposition}

\begin{proof}
For the first statement, we need to show that $F_d\mu_\tau=\mu_{b}(F_d\ot F_d)$ and $F_d\eta_\tau= d$, the latter being trivial.

First we prove
\begin{gather}
    \label{eq1}
    ( A^+\ot(\varepsilon^+\ot A^-)\tau)(F_d\ot A^+)=
    F_d(\mu^+\ot A^-)( A^+\ot\tau).
\end{gather}
Indeed, we have
\begin{gather*}
    \begin{split}
        \text{LHS} 
        &=(\mu^+\ot(\varepsilon^+\ot A^-)\tau(\mu^-\ot A^+))( A^+\ot d\ot A^-\ot A^+)\\
        &=(\mu^+\ot(\varepsilon^+\ot\mu^-)(\tau\ot A^-)( A^-\ot\tau))( A^+\ot d\ot A^-\ot A^+)\\
        &=(\mu^+\ot\mu^-)( A^+\ot(( A^+\ot(\varepsilon^+\ot A^-)\tau)(d\ot A^+))\ot A^-)( A^+\ot\tau)\\
        &=(\mu^+\ot\mu^-)( A^+\ot(\mu^+\ot A^-)( A^+\ot d)\ot A^-)( A^+\ot\tau)\quad\text{(by \eqref{eq2})}\\
        &=\text{RHS}.
    \end{split}
\end{gather*}

Using \eqref{eq1}, we have
\begin{gather*}
    \begin{split}
        &\mu_{b}(F_d\ot F_d)\\
        =&( A^+\ot b( A^-\ot\mu^+)\ot  A^-)(F_d\ot  A^+\ot ( A^+\ot\mu^-)(d\ot A^-))\\
        =&( A^+\ot  (\varepsilon^+\ot b)(\tau\ot A^+)  \ot  A^-)(F_d\ot  A^+\ot ( A^+\ot\mu^-)(d\ot A^-))\\
        =&( A^+\ot\mu^-)(( A^+\ot(\varepsilon^+\ot A^-)\tau)(F_d\ot A^+)\ot A^-)\\
        =&( A^+\ot\mu^-)( F_d(\mu^+\ot A^-)( A^+\ot\tau)  \ot A^-)\quad\text{(by \eqref{eq1})}\\
        =&F_d\mu_\tau.
    \end{split}
\end{gather*}

The second statement follows from the existence of the inverse $F_d^{-1}$ defined in \eqref{Frho-}.
\end{proof}

\subsection{Morita equivalence for twisted tensor products}\label{secMoritaequivformonoids}

Let $A^+$ and $A^-$ be augmented monoids in $\calM$.
Let $\tau:A^-\ot A^+\to A^+\ot A^-$ be a perfect twist with duality $(b,d)$ such that $d$ is stack-invertible.

We have a left $A_\tau$-module $A^+=(A^+,\trr_{\tau})$ consisting of the object $A^+$ and the morphism $\trr_{\tau}$ defined by 
\[
\trr_{\tau} = \mu^+( A^+\ot( A^+\ot \varepsilon^-)\tau):A_\tau\ot A^+\to A^+.
\]
Similarly, we have a right $A_\tau$-module $A^-=(A^-,\trl_{\tau})$ consisting of the object $A^-$ and the morphism $\trl_{\tau}$ defined by
\[
\trl_{\tau} = \mu^-((\varepsilon^+\ot A^-)\tau\ot A^-):A^-\ot A_\tau\to A^-.
\]
Then we have the $A_\tau$-bimodule $(A^+\ot A^-,\trr_{\tau}\ot A^-,A^+\ot \,\trl_{\tau})$.

We will prove Theorem~\ref{Morita1} below about ``Morita equivalence'' between the monoids $I$ and $A_\tau$ in $\calM$.
Note that if $\calM$ is the category of modules over a commutative ring, then the conditions in Theorem~\ref{Morita1} mean that $I$ and $A_\tau$ are Morita equivalent in the usual sense. (See Section~\ref{subsectionRbim}.)

We make the following definition of balanced tensor products.
Let $A$ be a monoid in $\calM$.  Let $M=(M,\trl)$ and $N=(N,\trr)$ be right and left $A$-modules, respectively.
Then the tensor product $M\otimes_A N$ of $M$ and $N$ over $A$ is the coequalizer of the pair of morphisms
\[
M\ot A\ot N 
\overset{\trl\ot N}{\underset{M\ot\,\trr}{\rightrightarrows}}M\ot N
\]
if it exists.  Thus, by $M\ot_A N\cong P$, we mean that there exists a coequalizer diagram
\[
M\ot A\ot N 
\overset{\trl\ot N}{\underset{M\ot\,\trr}{\rightrightarrows}}M\ot N\to P.
\]
If, moreover, $M$ is a $(B,A)$-bimodule and $N$ is an $(A,C)$-bimodule for some monoids $B$ and $C$, then $M\otimes_A N$, if it exists, has a $(B,C)$-bimodule structure in an obvious way.

Note that for the trivial monoid $I=(I,\id_I,\id_I)$, each object $X$ in $\M$ has an $I$-module structure with action $\id_X$.
This gives an isomorphism $\M\cong \Imod$ between $\M$ and the category $\Imod$ of $I$-modules in $\M$.
Similarly, we have an isomorphism $\M\cong I\text{-bim}$ between $\M$ and the category $I\text{-bim}$ of $I$-bimodules in $\M$. 
The tensor product $M\ot_I N$ over $I$ always exists, and is isomorphic to the tensor product $M\ot N$ in $\M$.

\begin{theorem}\label{Morita1}
Let $\calM$ be a strict monoidal category.
Let $A^+$ and $A^-$ be augmented monoids in $\calM$, and let $\tau:A^-\ot A^+\to A^+\ot A^-$ be a twist.
Then the following are equivalent:
\begin{enumerate}
    \item The twist $\tau$ is perfect with stack-invertible copairing.
    \item We have an $A_\tau$-bimodule isomorphism $A^+\ot_I A^-\cong A_\tau$.
    \item We have an $A_\tau$-bimodule isomorphism $A^+\ot_I A^-\cong A_\tau$ and an $I$-bimodule isomorphism $A^-\ot_{A_\tau}A^+\cong I$.
\end{enumerate}
If, moreover, $\calM$ is idempotent-complete, then the above conditions are equivalent to the following:
\begin{enumerate}
    \item[(4)] We have a well-defined functor $A^-\ot_{A_\tau} -: A_\tau\mmod \to I\mmod$, which gives an equivalence of categories
$$I\mmod \mathrel{\substack{\xrightarrow{\ A^+\ot_I -\ }\\[-0.4ex]\xleftarrow[A^-\ot_{A_\tau} -]{}}}A_\tau\mmod.$$
\end{enumerate}
\end{theorem}

Obviously, $(3)\Rightarrow (2)$.
We will prove $(2)\Rightarrow (1)$ in Section~\ref{subsub2to1}, $(1)\Rightarrow (3)$ in Section~\ref{subsub1to3}, $(4)\Rightarrow (2)$ in Section~\ref{subsub4to2} and $(1)\Rightarrow (4)$ in Section~\ref{subsub1to4}.

\subsubsection{Proof of Theorem~\ref{Morita1} $(2)\Rightarrow (1)$}\label{subsub2to1}

Let $f=f^+:A_\tau\congto A^+\ot A^-$ be an $A_\tau$-bimodule isomorphism with $f^-=f^{-1}$ its inverse. Set
\begin{gather*}
    d^+= \tilde{k}^-\bullet\hat{f}^+,\quad d^-=\hat{f}^-\bullet\tilde{k}^+: I\to A^+\ot A^-,
\end{gather*}
where
\begin{gather*}
    \hat{f}^{\pm}=f^\pm(\eta^+\ot \eta^-):I\to A^+\ot A^-,\\
    k^\pm=(\varepsilon^+\ot\varepsilon^-)\hat{f}^\pm:I\to I,\\
    \tilde{k}^\pm=\eta^+\ot k^\pm\ot \eta^-:I\to A^+\ot A^-.
\end{gather*}
It suffices to prove:
\begin{enumerate}
    \item[(i)] $b$ and $d^+$ satisfy the zigzag identities: 
    \begin{gather}
        \label{conv-zigzag}
        (A^+\ot b)(d^+\ot A^+)=\id_{A^+},\quad(b\ot A^-)(A^-\ot d^+)=\id_{A^-}.
    \end{gather}
    \item[(ii)] $d^+$ and $d^-$ are stack-inverses of each other.
\end{enumerate}

Since $f^+:A_\tau\congto A^+\ot A^-$ and its inverse are $A_\tau$-bimodule morphisms,
we have
\begin{gather}
\label{conv11}
f^\pm(\mu^+\ot A^-)=(\mu^+\ot A^-)(A^+\ot f^\pm),\\
\label{conv12}
    f^\pm(A^+\ot \mu^-)=(A^+\ot \mu^-)(f^\pm\ot A^-),\\
\label{conv13}
    f^-((A^+\ot\varepsilon^-)\tau\ot A^-)
    =(A^+\ot\mu^-)(\tau\ot A^-)(A^-\ot f^-),\\
\label{conv14}
    f^+(\mu^+\ot A^-)(A^+\ot\tau)
    =(A^+\ot(\varepsilon^+\ot A^-)\tau)(f^+\ot A^+).
\end{gather}

Set
\begin{gather*}
    g^\pm=(A^+\ot \varepsilon^-)f^\pm(A^+\ot \eta^-):A^+\to A^+.
\end{gather*}

\begin{lemma}\label{conv}
We have the following.
    \begin{gather}
        \label{conv1} b(A^-\ot g^-)=k^- b,\\
        \label{conv2} (A^+\ot b)(\hat{f}^+\ot A^+)=g^+,\\
        \label{conv0} g^\pm g^\mp=\id_{A^+},\\
        \label{conv4} \hat{f}^\mp\bullet\hat{f}^\pm=\eta^+\ot\eta^-,\\
        \label{conv3} k^\mp k^\pm=\id_I.
    \end{gather}
\end{lemma}

\begin{proof}
    We have \eqref{conv1} since
    \begin{gather*}
        \begin{split}
            b(A^-\ot g^-)
            &=(\varepsilon^+\ot \varepsilon^-\mu^-)(\tau\ot A^-)(A^-\ot f^-)(A^-\ot A^+\ot \eta^-)\\
            &=(\varepsilon^+\ot \varepsilon^-)f^-((A^+\ot \varepsilon^-)\tau\ot \eta^-)
            \quad(\text{by \eqref{conv13}})\\
            &=(\varepsilon^+\ot \varepsilon^-)f^-(\mu^+(A^+\ot\eta^+)\ot A^-)((A^+\ot\varepsilon^-)\tau\ot\eta^-)\\
            &=(\varepsilon^+\mu^+\ot \varepsilon^-)(A^+\ot f^-(\eta^+\ot\eta^-))(A^+\ot\varepsilon^-)\tau
            \quad(\text{by \eqref{conv11}})\\
            &=k^- b.
        \end{split}
    \end{gather*}

    We have \eqref{conv2} since
    \begin{gather*}
        \begin{split}
            (A^+\ot b)(\hat{f}^+\ot A^+)
            &=(A^+\ot (\varepsilon^+\ot\varepsilon^-)\tau)(f^+\ot A^+)(\eta^+\ot\eta^-\ot A^+)\\
            &=(A^+\ot \varepsilon^-)f^+(\mu^+\ot A^-)(A^+\ot\tau)(\eta^+\ot\eta^-\ot A^+)
            \quad(\text{by \eqref{conv14}})\\
            &=g^+.
        \end{split}
    \end{gather*}

    We have \eqref{conv0} since
    \begin{gather*}
    \begin{split}
        g^\pm g^\mp
        &=(A^+\ot \varepsilon^-)f^\pm(A^+\ot \eta^-)
        (A^+\ot \varepsilon^-)f^\mp(A^+\ot \eta^-)\\
        &=(A^+\ot \varepsilon^-\mu^-)(f^\pm\ot A^-)(A^+\ot \eta^-\ot A^-)
        f^\mp(A^+\ot \eta^-)\\
        &=(A^+\ot \varepsilon^-)f^\pm(A^+\ot \mu^-(\eta^-\ot A^-))
        f^\mp(A^+\ot \eta^-)
        \quad(\text{by \eqref{conv12}})\\
        &=(A^+\ot \varepsilon^-)f^\pm
        f^\mp(A^+\ot \eta^-)\\
        &=(A^+\ot \varepsilon^-)(A^+\ot \eta^-)\\
        &=\id_{A^+}.
    \end{split}
    \end{gather*}

    We have \eqref{conv4} since
\begin{gather*}
\begin{split}
    \hat{f}^\mp\bullet\hat{f}^\pm
    &=(\mu^+\ot\mu^-)(A^+\ot f^\mp\ot A^-)(A^+\ot\eta^+\ot\eta^-\ot A^-)\hat{f}^\pm\\
    &=f^\mp(\mu^+(A^+\ot\eta^+)\ot\mu^-(\eta^-\ot A^-))f^\pm(\eta^+\ot\eta^-)
    \quad(\text{by \eqref{conv11},\eqref{conv12}})\\
    &=f^\mp f^\pm(\eta^+\ot\eta^-)\\
    &=\eta^+\ot\eta^-.
\end{split}
\end{gather*}
    
    By post-composing $\varepsilon^+\ot\varepsilon^-$ with \eqref{conv4}, we easily obtain \eqref{conv3}.
\end{proof}

We first prove (ii).
By \eqref{conv3}, $\tilde{k}^+$ and $\tilde{k}^-$ are stack-inverses of each other.
So are $\hat{f}^+$ and $\hat{f}^-$ by \eqref{conv4}.
It is now easy to see that $d^+$ and $d^-$ are stack-inverses of each other.

We now prove (i).  It suffices to check the first identity of \eqref{conv-zigzag} since the second one follows similarly by symmetry.
We have
\begin{gather*}
    \begin{split}
        (A^+\ot b)(d^+\ot A^+)
        &=(A^+\ot k^-b)(\hat{f}^+\ot A^+)\\
        &=(A^+\ot b(A^-\ot g^-))(\hat{f}^+\ot A^+)\quad\text{(by \eqref{conv1})}\\
        &=g^+g^-\quad\text{(by \eqref{conv2})}\\
        &=\id_{A^+}\quad\text{(by \eqref{conv0})}.
    \end{split}
\end{gather*}

This completes the proof of the statement $(2)\Rightarrow (1)$ of Theorem~\ref{Morita1}.

\subsubsection{Proof of Theorem~\ref{Morita1} $(1)\Rightarrow (3)$}\label{subsub1to3}
We first record the corresponding Morita equivalence for the monoid $A_b.$
Let $A^+$, $A^-$ be objects in $\M$ and $\eta^+:I\to A^+$, $\eta^-:I\to A^-$ morphisms in $\M$.
(Here we do not assume that $A^+$ or $A^-$ is a monoid.)
Let $(b,d)$ be a duality for the objects $A^+$ and $A^-$.
Then we have a monoid $A_b=(A^+\ot_b A^-, \mu_b=A^+\ot b\ot A^-,d)$ as in Section~\ref{Preliminaries}.

We have a left $A_b$-module $A^+=(A^+,\trr_{b})$ consisting of the object $A^+$ and the morphism $\trr_{b}$ defined by 
\[
\trr_{b} = A^+\ot b:A_b\ot A^+\to A^+.
\]
Similarly, we have a right $A_b$-module $A^-=(A^-,\trl_{b})$ consisting of the object $A^-$ and the morphism $\trl_{b}$ defined by
\[
\trl_{b} = b\ot A^-:A^-\ot A_b\to A^-.
\]
Then the $A_b$-bimodule $A^+\ot A^-=(A^+\ot A^-, \trr_{b}\ot A^-,A^+\ot\,\trl_{b})$ is isomorphic to the $A_b$-bimodule $A_b$.

\begin{lemma}
\label{coequalizer}
If $b(\eta^-\ot \eta^+)=\id_I$, then we have a coequalizer diagram
    \[
    A^-\ot A_b\ot A^+ \overset{\trl_b\ot A^+}{\underset{A^-\ot\,\trr_b}{\rightrightarrows}}A^-\ot A^+\xrightarrow{b}I.
    \]
    In other words, we have $A^-\ot_{A_b} A^+\cong I$.
\end{lemma}

\begin{proof}
    We have $b (\trl_b\ot A^+)=b\ot b=b (A^-\ot\,\trr_b)$.
    Suppose that a morphism $u:A^-\ot A^+\to X$ in $\calM$ satisfies $u(\trl_b\ot A^+)=u(A^-\ot\,\trr_b)$.
    Precomposing $ A^-\ot A^+\ot\eta^-\ot\eta^+$ with $u(\trl_b\ot A^+)=u(A^-\ot\,\trr_b)$, we obtain 
    \begin{gather*}
        u=u(\eta^-\ot \eta^+)b.
    \end{gather*}
    Thus, $u$ factors through $b$.
    This factorization is unique, since if $u=v b$ for a morphism $v:I\to X$, then we have
    $$v=v b(\eta^-\ot\eta^+)=u(\eta^-\ot\eta^+).$$
    This completes the proof.
\end{proof}

\begin{proposition}
\label{MoritaAb1}
Let $\calM$ be a strict monoidal category.
Let $A^+, A^-$ be objects in $\calM$. Let $\eta^+:I\to A^+,\,\eta^-:I\to A^-$ be morphisms in $\M$. Let $b:A^-\ot A^+\to I$ be a perfect pairing such that $b(\eta^-\ot \eta^+)=\id_I$.
Then we have an $I$-bimodule isomorphism $A^-\ot_{A_b}A^+\cong I$ and an $A_b$-bimodule isomorphism $A^+\ot_I A^-\cong A_b$.
\end{proposition}

\begin{proof}
By Lemma~\ref{coequalizer}, we have the tensor product $A^-\ot_{A_b}A^+\cong I$ as an object in $\M$. Via the isomorphism $\M\cong I\text{-bim}$, we obtain an $I$-bimodule isomorphism $A^-\ot_{A_b}A^+\cong I$.
By the definition of the $A_b$-bimodule structure on $A^+\otimes A^-$, we have an $A_b$-bimodule isomorphism $A^+\ot_I A^-\cong A^+\ot A^-\cong A_b$.
This completes the proof.
\end{proof}

Now we prove Theorem~\ref{Morita1} $(1)\Rightarrow (3)$.

\begin{proof}[Proof of Theorem~\ref{Morita1} $(1)\Rightarrow (3)$]
Since $b=b_\tau$ satisfies $b(\eta^-\ot A^+)=\varepsilon^+$ by \eqref{eq01} and thus $b(\eta^-\ot \eta^+)=\id_I$, Proposition~\ref{MoritaAb1} implies an $I$-bimodule isomorphism $A^-\ot_{A_b}A^+\cong I$ and an $A_b$-bimodule isomorphism $A^+\ot_I A^-\cong A_b$.
By Proposition~\ref{propF-rho}, we have a monoid isomorphism $F_{d}: A_\tau\xrightarrow{\cong}A_{b}$.
Since we have 
$$\trl_{\tau}\otimes A^+=(\trl_{b}\otimes A^+)(A^-\otimes F_d\otimes A^+), \quad A^-\otimes \trr_{\tau}=(A^-\otimes \trr_{b})(A^-\otimes F_d\otimes A^+),$$ we have a coequalizer diagram 
\begin{gather}\label{coequalizertau}
    A^-\ot A_\tau\ot A^+ \overset{\trl_{\tau}\ot A^+}{\underset{A^-\ot\,\trr_{\tau}}{\rightrightarrows}}A^-\ot A^+\xrightarrow{b}I,
\end{gather}
that is, we have an $I$-bimodule isomorphism $A^-\ot_{A_{\tau}}A^+\cong I$.
Since we have 
$$A^+\otimes \trl_{\tau}=(A^+\otimes \trl_{b})(A^+\otimes A^-\otimes F_d), \quad \trr_{\tau}\otimes A^-=(\trr_{b}\otimes A^-)(F_d\otimes A^+\otimes A^-),$$
we have an $A_{\tau}$-bimodule isomorphism $A_{\tau}\xrightarrow[\cong]{F_d}A_{b}\cong A^+\otimes_I A^-$, which completes the proof of Theorem~\ref{Morita1} $(1)\Rightarrow (3)$.
\end{proof}

\begin{remark}
  For a perfect twist $\tau$, even if the copairing is not stack-invertible, 
  we have the coequalizer diagram \eqref{coequalizertau}, that is,
  we have $A^-\ot_{A_\tau} A^+\cong I$.
\end{remark}

\subsubsection{Proof of Theorem~\ref{Morita1} $(4)\Rightarrow (2)$}\label{subsub4to2}

By assumption, we have a natural isomorphism
\begin{gather*}
   \varphi=( \varphi_M: M\xrightarrow{\cong}A^+\ot (A^-\ot_{A_\tau}M) )_{M\in A_\tau\mmod}.
\end{gather*}
Since the natural right action $\trl:A_\tau\ot A_\tau\to A_\tau$ given by multiplication is a left $A_\tau$-module morphism,
we have the following commutative diagram:
\begin{gather*}
    \xymatrix{
    A_\tau\ot A_\tau 
    \ar[rr]^-{\varphi_{(A_\tau\ot A_\tau)}}_-{\cong} \ar[d]^{\trl}&&
    A^+\ot(A^-\ot_{A_\tau}(A_\tau\ot A_\tau))
    \ar[d]^{A^+\ot (A^-\ot_{A_\tau} \trl)}
    \\
    A_\tau \ar[rr]^-{\varphi_{A_\tau}}_-{\cong}&&
    A^+\ot (A^-\ot_{A_\tau} A_\tau) 
    }
\end{gather*}
Similarly to Lemma~\ref{coequalizer}, we obtain $A_\tau$-module isomorphisms 
\begin{gather*}
    A^-\ot_{A_\tau}(A_\tau\ot A_\tau)\cong A^-\ot A_\tau,\quad A^-\ot_{A_\tau} A_\tau  \cong A^-.
\end{gather*}
Via these isomorphisms, the $A_\tau$-module morphism $A^+\ot (A^-\ot_{A_\tau} \trl)$ corresponds to the map $A^+\otimes \trl_{\tau}: A^+\ot (A^-\ot A_\tau)\to A^+\ot A^-$.
Therefore, the isomorphism 
$\varphi_{A_\tau}: A_\tau\xrightarrow{\cong} A^+\ot A^-$
of $A_\tau$-modules is, in fact, an isomorphism of $A_\tau$-bimodules, which completes the proof of Theorem~\ref{Morita1} $(4)\Rightarrow (2)$.

\subsubsection{Proof of Theorem~\ref{Morita1} $(1)\Rightarrow (4)$}\label{subsub1to4}
As in Section~\ref{subsub1to3}, in order to prove Theorem~\ref{Morita1} $(1)\Rightarrow (4)$, we consider the monoid $A_b$ in a general setting.

\begin{proposition}\label{MoritaAb2}
Let $\calM$ be a strict monoidal category.
Let $A^+,A^-$ be objects in $\calM$. Let $\eta^\pm:I\to A^\pm$, $\varepsilon^+:A^+\to I$ be morphisms in $\M$ such that $\varepsilon^+\eta^+=\id_I$. Let $b:A^-\ot A^+\to I$ be a perfect pairing such that $b(\eta^-\ot A^+)=\varepsilon^+$.
Assume that $\calM$ is idempotent-complete.
Then we have a well-defined functor $A^-\ot_{A_b} -: A_b\mmod \to I\mmod$, which gives an equivalence of categories
$$(\M\simeq)I\mmod \mathrel{\substack{\xrightarrow{\ A^+\ot_I -\ }\\[-0.4ex]\xleftarrow[A^-\ot_{A_b} -]{}}}A_b\mmod.$$
\end{proposition}

\begin{proof}
    For each $A_b$-module $N=(N,\alpha_N:A_b\ot N\to N)$,
    define a morphism in $\M$
    \[
    \theta_N = \alpha_N(\eta^+\ot \eta^-\ot N): N\to N.
    \]
    Then $\theta_N$ is an idempotent since we have
    $$\theta_N^2=\alpha_N(A_b\ot \alpha_N)(\eta^+\ot\eta^-\ot\eta^+\ot\eta^-\ot N)=\alpha_N(\mu_b\ot N)(\eta^+\ot\eta^-\ot\eta^+\ot\eta^-\ot N)=\theta_N.$$
    By assumption, for each $N$, we can choose a splitting 
    \begin{gather}\label{splitting-theta}
    N\overset{p_N}{\underset{i_N}{\rightleftarrows}} T(N)    
    \end{gather}
    of $\theta_N$ in $\M$,
    i.e., $i_Np_N=\theta_N$ and $p_Ni_N=\id_{T(N)}$.
    Then we have a functor $T:\Abmod\to\M$ such that for a morphism $f:N\to N'$ in $\Abmod$ we have $T(f)=p_{N'} f\, i_N$. Note that \eqref{splitting-theta} is a functorial splitting, meaning that for a morphism $f:N\to N'$ in $\Abmod$, we have $p_{N'}f=T(f)p_N$ and $f i_N=i_{N'}T(f)$.

    We prove that $T\cong A^-\otimes_{A_b}-$.
    For each $A_b$-module $N$, we have the following coequalizer diagram 
    \begin{gather*}
        A^-\otimes A_b\otimes N
        \overset{\trl_b \ot N}{\underset{A^-\ot\,\alpha_N}{\rightrightarrows}}A^-\ot N\xrightarrow{p_N \alpha_N(\eta^+\otimes A^-\otimes N)}T(N),
    \end{gather*}
    which can be checked in a way similar to the proof of Lemma~\ref{coequalizer}; for any morphism $u:A^-\otimes N\to X$ in $\M$ satisfying $u(\trl_b \ot N)=u(A^-\ot\,\alpha_N)$, we have 
    \begin{gather*}
        u=u(\eta^-\otimes i_N)p_N \alpha_N(\eta^+\otimes A^-\otimes N)
    \end{gather*}
    and thus $u$ factors through $p_N \alpha_N(\eta^+\otimes A^-\otimes N)$, and the factorization is unique.

    We prove that $T$ and the functor $A^+\ot-:\M\to\Abmod$ are inverses to each other up to natural isomorphism.
    First let us check $A^+\ot T(N)\cong N$ for $N\in \Abmod$.
    Define morphisms
    \[ 
    A^+\ot T(N)\;\overset{\beta_N}{\underset{\gamma_N}{\rightleftarrows}}
    \; N\]
    by
    \begin{gather*}
    \beta_N=\alpha_N(A^+\ot \eta^-\ot i_N),\\
    \gamma_N = ( A^+\ot (p_N\alpha_N(\eta^+\ot A^-\ot N)))(d\ot N).
    \end{gather*}
    Then one can check that $\beta_N$ and $\gamma_N$ are inverses to each other as follows.
    We have
    \begin{gather*}
        \begin{split}
            \beta_N\gamma_N
            &=\alpha_N(A^+\ot \eta^-\ot i_N p_N)(A^+\ot \alpha_N)(((A^+\ot \eta^+\ot A^-)d)\ot N)\\
            &=\alpha_N(((A^+\ot (b(\eta^-\ot \eta^+)) \ot (b(\eta^-\ot \eta^+))\ot A^-)d)\ot N)\\
            &=\alpha_N(d\ot N)\\
            &=\id_N.
        \end{split}
    \end{gather*}
    We have
    \begin{gather*}
        \begin{split}
            \gamma_N\beta_N
            &= (A^+\ot p_N\alpha_N)(((A^+\ot \eta^+\ot A^-)d)\ot N)\alpha_N( A^+\ot \eta^-\ot i_N)\\
            &= (A^+\ot b)(d\ot A^+)\ot (p_N\alpha_N(\eta^+\ot \eta^-\ot i_N))\\
            &=A^+\ot p_N \theta_N i_N\\
            &=\id_{A^+\ot T(N)}.
        \end{split}
    \end{gather*}
    
    We now check $T(A^+\ot X)\cong X$ for $X\in\M$. 
    Since $\theta_{A^+\ot X}=(\eta^+\ot X)(\varepsilon^+\ot X)$ and $(\varepsilon^+\ot X)(\eta^+\ot X)=\id_X$,
    \[ 
    A^+\ot X\;\overset{\varepsilon^+\ot X}{\underset{\eta^+\ot X}{\rightleftarrows}}
    \; X\]
    is a splitting of $\theta_{A^+\ot X}$.  Hence, $T(A^+\ot X)\cong X$.
    This completes the proof.
\end{proof}

Finally, we prove the rest of Theorem~\ref{Morita1} $(1)\Rightarrow (4)$.

\begin{proof}[Proof of Theorem~\ref{Morita1} $(1)\Rightarrow (4)$]
Since $b=b_\tau$ satisfies $b(\eta^-\ot A^+)=\varepsilon^+$ by \eqref{eq01},
Propositions~\ref{propF-rho} and~\ref{MoritaAb2} imply condition $(4)$ of Theorem~\ref{Morita1}.
\end{proof}

\subsection{The center $Z^{\calM}(A_\tau)$ of $A_\tau$}\label{sectioncenterofmonoid}
The Morita equivalence between two algebras over a commutative ring implies the isomorphism of their centers.
Here we study the center of $A_\tau$.

By a \emph{central element} of a monoid $A=(A,\mu,\eta)$ in $\calM$, we mean a morphism $f:I\to A$ such that
\begin{gather}
    \mu(f\ot A)=\mu( A\ot f).
\end{gather}
The \emph{center} $Z^{\calM}(A)$ of $A$ is defined to be the subset of $\calM(I,A)$ consisting of central elements of $A$, which
has a commutative monoid structure with multiplication $*$ given by
\[f*g=\mu(f\ot g)\]
for $f,g\in Z^{\calM}(A)$, and unit $\eta:I\to A$.  If $\calM$ is $\k$-linear monoidal for a commutative ring $\k$, then $Z^{\calM}(A)$ has a $\k$-algebra structure.
If $\calM=\k\mmod$, where a monoid in $\k\mmod$ is a $\k$-algebra, then the center $Z^{\calM}(A)$ of a $\k$-algebra $A$ in the above sense coincides with the usual center of $A$.

Note that the endomorphism monoid $\End_{\calM}(I)=\calM(I,I)$ of the unit object $I$ is commutative, and therefore, we have 
$$Z^{\calM}(I)=\End_{\calM}(I),$$
where the multiplication coincides with the tensor product.

\begin{theorem}\label{centerofAtau}
Let $A^+$ and $A^-$ be augmented monoids in $\calM$, and let $\tau:A^-\ot A^+\to A^+\ot A^-$ be a perfect twist with duality $(b,d)$ such that $d$ is stack-invertible.
Then we have a monoid isomorphism
\begin{gather*}
    \zeta_\tau:Z^{\calM}(I)\congto Z^{\calM}(A_\tau)
\end{gather*}
defined by $\zeta_\tau(f)=F_{d}^{-1}(A^+\ot f\ot A^-)d$ for $f\in Z^{\calM}(I)$.
\end{theorem}

In order to prove Theorem~\ref{centerofAtau}, we prove the following proposition, which studies the center of the monoid $A_b$ in a general setting. 

\begin{proposition}\label{centerofAb}
    Let $A^+,A^-$ be objects in $\M$.
    Let $\eta^+:I\to A^+$ and $\varepsilon^{\pm}:A^{\pm}\to I$ be morphisms in $\M$ such that $\varepsilon^+\eta^+=\id_I$.
    Let $(b,d)$ be a duality for the objects $A^+$ and $A^-$
    such that $b(A^-\ot \eta^+)=\varepsilon^-$.
    Then we have a monoid isomorphism 
    \begin{gather*}
        \zeta_d: Z^{\calM}(I)\congto Z^{\calM}(A_b)
    \end{gather*}
    defined by
    $\zeta_d(f)=(A^+\ot f\ot A^-)d$ for $f\in Z^{\calM}(I)$.
\end{proposition}

\begin{proof}
    The map $\zeta_d$ is well defined since we have for $f\in Z^{\calM}(I)$,
    \begin{gather*}
        \mu_b(\zeta_d(f)\ot A^+\ot A^-)=A^+\ot f\ot A^-=\mu_b(A^+\ot A^-\ot \zeta_d(f)).
    \end{gather*}
    It is easy to check that $\zeta_d$ is a monoid morphism.
    Therefore, it suffices to construct the inverse of $\zeta_d$. Let 
    $$\varepsilon_*=(\varepsilon^+\ot \varepsilon^-)\circ -: Z^{\calM}(A_b)\to Z^{\calM}(I).$$
    We will check that $\varepsilon_*$ is the inverse of $\zeta_d$. 
    We have for $f\in Z^{\calM}(I)$
    \begin{gather*}
        \varepsilon_*\zeta_d(f)=(\varepsilon^+\ot \varepsilon^-)(A^+\ot f\ot A^-)d=f,
    \end{gather*}
    where $(\varepsilon^+\ot \varepsilon^-)d=\id_I$ follows from $b(A^-\ot \eta^+)=\varepsilon^-$ and $\varepsilon^+\eta^+=\id_I$.
    By applying the map 
    $$g:\calM(A^+\ot A^-,A^+\ot A^-)\to \calM(I,A^+\ot A^-),\quad v\mapsto (A^+\ot (\varepsilon^+\ot A^-)v(\eta^+\ot A^-))d$$
    to both sides of $\mu_b(u\ot A_b)=\mu_b(A_b\ot u)$  with $u\in Z^{\calM}(A_b)$, we obtain 
    \[
    \zeta_d\varepsilon_*(u)=g(\mu_b(u\ot A_b))=g(\mu_b(A_b\ot u))=u
    \]
    using $b(A^-\ot \eta^+)=\varepsilon^-$ and $\varepsilon^+\eta^+=\id_I$.
    This completes the proof.
\end{proof}

\begin{proof}[Proof of Theorem~\ref{centerofAtau}]
    Since the augmented monoid $A^+$ satisfies $\varepsilon^+\eta^+=\id_I$, and since $b=b_\tau$ satisfies $b(A^-\ot \eta^+)=\varepsilon^-$ by \eqref{eq02}, we have the monoid isomorphism
\begin{gather*}
    \zeta_{d}: Z^{\calM}(I)\congto Z^{\calM}(A_b)
\end{gather*}
by Proposition~\ref{centerofAb}.
Therefore, by Proposition~\ref{propF-rho}, the composition
\begin{gather*}
    \zeta_\tau=F_{d}^{-1}\zeta_d
\end{gather*}
is a monoid isomorphism, which completes the proof.
\end{proof}

\subsection{Canonical idempotents}
\label{sec:canonical-idempotents}

Let $A^+,A^-$ be augmented monoids in $\calM$.
Let $\tau: A^-\otimes A^+\to A^+\otimes A^-$ be a twist.
We introduce canonical idempotents for $A_\tau$ and relate them to perfectness of the twist.

Let $M(A_\tau)$ denote the set $\calM(I,A^+\ot A^-)$.
Then the stack multiplication~$\bullet$ defined in \eqref{stack}
with unit $\eta=\eta^+\ot\eta^-$ forms a monoid structure on $M(A_\tau)$, which is denoted by $(M(A_\tau),\bullet)$.
(Note that stack multiplication does not depend on the twist $\tau$.)
Besides the stack multiplication, the set $M(A_\tau)$ admits another multiplication defined by $f * g := \mu_\tau(f\ot g)$, which forms a monoid with unit $\eta=\eta^+\ot\eta^-$, which is denoted by $(M(A_\tau),*)$.

We say that an element $f \in M(A_\tau)$ is \emph{admissible}
if we have
\begin{gather*}
    (\varepsilon^+\ot A^-)f = \eta^- (\varepsilon^+\ot\varepsilon^-)f,\quad
    ( A^+\ot\varepsilon^-)f = \eta^+ (\varepsilon^+\ot\varepsilon^-)f.
\end{gather*}
We say that $f$ is \emph{strongly admissible} if it is admissible and we have $(\varepsilon^+\ot\varepsilon^-)f=\id_I$, or more simply if
\begin{gather*}
    (\varepsilon^+\ot A^-)f = \eta^-,\quad
    ( A^+\ot\varepsilon^-)f = \eta^+.
\end{gather*}

\begin{example}\label{dadm}
    If the twist $\tau$ is perfect with duality $(b,d)$, then $d$ is strongly admissible by \eqref{eq21} and \eqref{eq22}.
\end{example}

Let $M^\adm(A_\tau)\subset M(A_{\tau})$ denote the subset of admissible morphisms and $M^\sadm(A_\tau)\subset  M(A_{\tau})$ that of strongly admissible morphisms.
Then it is straightforward to check that $M^\adm(A_\tau)$ and $M^\sadm(A_\tau)$ are submonoids of $M(A_\tau)$ with respect to both multiplications $\bullet$ and $*$.

\begin{remark}\label{sastackinversesa}
    If $f\in M^{\sadm}(A_\tau)$ is stack-invertible, then the stack-inverse $f^-$ of $f$ is strongly admissible.
\end{remark}

By a \emph{canonical idempotent} for the twisted tensor product $A_\tau$, we mean a morphism $\ee:I\to A^+\ot A^-$ satisfying the following identities:
    \begin{gather}
    \label{eq3ee}
    (\mu^+\ot A^-)( A^+\ot \tau)(\ee\ot  A^+)
    =\ee\varepsilon^+,\\
    \label{eq30ee}
   ( A^+\ot \mu^-)(\tau\ot A^-)( A^-\ot \ee)=\ee\varepsilon^-,\\
    \label{eq31ee}
    (\varepsilon^+\ot A^-)\ee
    =\eta^-,\\
    \label{eq32ee}
    ( A^+\ot\varepsilon^-)\ee
    =\eta^+.
\end{gather}
It is easy to check using \eqref{eq3ee} and \eqref{eq31ee} that $\ee$ is an idempotent for the monoid $A_\tau$ in $\calM$, i.e.,
    \begin{gather}\label{ee-idempotent}
        \ee*\ee=\mu_{\tau}(\ee\ot \ee)=\ee.
    \end{gather}
By \eqref{eq31ee} and \eqref{eq32ee}, we have $\ee\in M^\sadm(A_\tau)$.

\begin{lemma}
    A canonical idempotent $\ee$ for $A_\tau$ is central in the submonoid $M^\adm(A_\tau)$ of $(M(A_\tau),*)$.
\end{lemma}

\begin{proof}
    For any $f\in M^\adm(A_\tau)$, we have by \eqref{eq3ee} 
    \[\ee*f
    =(A^+\ot\mu^-)(\ee\ot (\varepsilon^+\ot A^-)f)
    =\ee(\varepsilon^+\ot \varepsilon^-)f.
    \]
    Similarly, we have $f*\ee=\ee(\varepsilon^+\ot \varepsilon^-)f$ by \eqref{eq30ee}.  Hence we have $\ee*f=f*\ee$,
    which implies that $\ee$ is central in $M^\adm(A_\tau)$.
\end{proof}

\begin{lemma}\label{d-rho-minus}
Let $\tau$ be a perfect twist with duality $(b,d)$.
If $d$ is stack-invertible, then the stack-inverse $d^-$ is a canonical idempotent for $A_\tau$.
\end{lemma}

\begin{proof} 
We will prove \eqref{eq3ee}--\eqref{eq32ee} with $\ee=d^-$.

For \eqref{eq3ee}, we have 
    \begin{gather*}
    \begin{split}
    F_d\circ \text{LHS}
    &=(\mu^+\ot\mu^-)( A^+\ot d\ot  A^-)(\mu^+\ot A^-)( A^+\ot \tau)(d^-\ot  A^+)\\
    &=(\mu^+\ot\mu^-)
    ( A^+\ot ( A^+\ot(\varepsilon^+\ot A^-)\tau)(d\ot A^+) \ot  A^-)\\
    &\qquad( A^+\ot \tau)(d^-\ot  A^+)
    \quad\text{(by \eqref{eq2})}\\
    &=( A^+\ot (\varepsilon^+\ot A^-) \tau)((\mu^+\ot\mu^-)( A^+\ot d\ot A^-)d^-\ot  A^+)\\
    &=F_d \circ\text{RHS}.
    \end{split}
    \end{gather*}
    Post-composing $F_d^{-1}$ yields \eqref{eq3ee}.
    The identity \eqref{eq30ee} is similarly proved.
    We have \eqref{eq31ee} as follows:
    \begin{gather*}
    \begin{split}
    \text{LHS}
    &=\mu^-((\varepsilon^+\ot A^-)d \ot  A^-)(\varepsilon^+\ot A^-)d^-\\
    &=(\varepsilon^+\ot A^-)F_d d^-\\
    &=\text{RHS}.
    \end{split}
    \end{gather*}
    The identity \eqref{eq32ee} is similarly proved.
\end{proof}

The following lemma is an ``inverse'' to Lemma~\ref{d-rho-minus}.
\begin{lemma}
    \label{lem-good-idempotent}
Let $\ee:I\to A^+\ot A^-$ be a canonical idempotent for $A_\tau$.
If $\ee$ is stack-invertible with stack-inverse $\ee^-$, then $\ee^-$ is the copairing of the pairing $b$.
\end{lemma}

\begin{proof}
We have
\begin{gather*}
    \begin{split}
         A^+\ot\eta^-
        &=(\mu^+\ot A^-)( A^+\ot\tau)(\eta^+\ot\eta^-\ot A^+)\\
        &=(\mu^+\ot A^-)( A^+\ot\tau)(
        (\ee\bullet \ee^-)
        \ot A^+)\\
        &=(\mu^+\ot\mu^-)
        ( A^+\ot (\mu^+\ot A^-)( A^+\ot\tau)(\ee\ot A^+)\ot A^-)\\
    &\qquad        ( A^+\ot\tau)(\ee^-\ot A^+)
        \\
        &=(\mu^+\ot\mu^-)
        ( A^+\ot \ee\varepsilon^+\ot A^-)( A^+\ot\tau)(\ee^-\ot A^+)
        \\
        &=(\mu^+\ot\mu^-)( A^+\ot\ee\ot A^-)( A^+\ot(\varepsilon^+\ot A^-)\tau)(\ee^-\ot A^+).
    \end{split}
\end{gather*}
By post-composing $(\mu^+\ot\mu^-)( A^+\ot\ee^-\ot A^-)$, we obtain
\begin{gather}\label{e234}
    (\mu^+\ot A^-)( A^+\ot\ee^-)=( A^+\ot(\varepsilon^+\ot A^-)\tau)(\ee^-\ot A^+).
\end{gather}
Then we have
\begin{gather*}
    \begin{split}
        ( A^+\ot b)(\ee^-\ot A^+)
        &=( A^+\ot (\varepsilon^+\ot\varepsilon^-)\tau)(\ee^-\ot A^+)\\
        &=( A^+\ot \varepsilon^-)(\mu^+\ot A^-)( A^+\ot\ee^-)\quad(\text{by \eqref{e234}})\\
        &= A^+.
    \end{split}
\end{gather*}
Similarly, we have $(b\ot A^-)( A^-\ot\ee^-)= A^-$.
Thus, $\ee^-$ is the copairing of $b$.
\end{proof}

Even if the twist $\tau$ is perfect and if $A_\tau$ admits a canonical idempotent, it seems that the copairing $d$ and the canonical idempotent $\ee$ are not necessarily stack-inverses of each other. In the following proposition, we will observe a sufficient condition for $d$ and $\ee$ to be stack-inverses.

\begin{proposition}\label{stack-invertible-perfect-idempotent}
    Suppose that the monoid $M^\sadm(A_\tau)$ with stack multiplication is a group.
    Then $\tau$ is perfect if and only if there is a canonical idempotent for $A_\tau$.
    Moreover, a canonical idempotent for $A_\tau$, if any, is unique.
\end{proposition}

\begin{proof}
    Suppose $\tau$ is perfect with duality $(b,d)$.  Since $d$ is strongly admissible, it follows from the assumption that $d$ admits a stack-inverse $d^-$. By Lemma~\ref{d-rho-minus}, $d^-$ is a canonical idempotent.

    Conversely, let $\ee$ be a canonical idempotent, which is strongly admissible, and hence admits a stack-inverse $\ee^-$ by the assumption.
    By Lemma~\ref{lem-good-idempotent}, $\ee^-$ is the copairing of $b$. Thus, $\tau$ is perfect.

    Finally, if there are two canonical idempotents $\ee$ and $\ee'$, then they are equal since both of them are equal to $d^-$, the stack-inverse of the copairing $d$ of $b$.  Thus, canonical idempotents are unique.
\end{proof}

\subsection{Monoid morphisms between $A_\tau$ and $A_{b'}$}

Let $A^+$ and $A^-$ be augmented monoids in $\calM$.
We compare the monoids $A_\tau$ and $A_{b'}$, where $b'$ is a perfect pairing not necessarily equal to $b_\tau$.

\begin{proposition}\label{morphismL}
Let $\tau$ be a twist for $A^+$ and $A^-$.
Let $b':A^-\ot A^+\to I$ be a perfect pairing with copairing $d':I\to A^+\ot A^-$.
Then we have the following.
\begin{enumerate}
    \item We have a monoid morphism
\begin{gather*}
    L_{\tau,b'}
    : A_\tau\to A_{b'}
\end{gather*}
given by
\begin{gather}\label{eq-L}
    L_{\tau,b'}=( A^+\ot \varepsilon^+\ot \mu^-)( A^+\ot \tau\ot  A^-) (d'\ot  A^+\ot  A^-).
\end{gather}
\item If $\tau$ is perfect with duality $(b,d)$, then we have a monoid isomorphism 
\begin{gather*}
   K_{b,b'}:A_{b}\xrightarrow{\cong} A_{b'}
\end{gather*}
given by
\begin{gather*}
    K_{b,b'}=K^+_{b,b'}\ot A^-,
\end{gather*}
where 
\begin{gather*}
    K^+_{b,b'}=( A^+\ot b)(d'\ot A^+),
\end{gather*}
and $L_{\tau,b'}$ is equal to the composition of monoid morphisms
\begin{gather*}
    L_{\tau,b'}:
    A_\tau
    \xrightarrow{F_{d}}
    A_{b}
    \xrightarrow[\cong]{K_{b,b'}}
    A_{b'}.
\end{gather*}
\end{enumerate}
\end{proposition}

\begin{proof}
    (1) We need to show that $L_{\tau,b'}\mu_\tau=\mu_{b'}(L_{\tau,b'}\ot L_{\tau,b'})$ and $L_{\tau,b'}\eta_\tau= d'$, the latter being trivial.
    We have
    \begin{gather*}
        \begin{split}
            L_{\tau,b'}\mu_\tau&=( A^+\ot \varepsilon^+\ot \mu^-)( A^+\ot \tau\ot  A^-) (d'\ot  A^+\ot  A^-)(\mu^+\ot\mu^-)( A^+\ot \tau \ot A^-)\\
            &=( A^+\ot \varepsilon^+\ot \mu^-)(( A^+\ot\tau)( A^+\ot A^-\ot\mu^+)(d'\ot A^+\ot A^+)\ot\mu^-)( A^+\ot \tau \ot A^-)\\
            &=( A^+\ot\mu^-)(( A^+\ot \varepsilon^+\ot A^-)( A^+\ot \varepsilon^+\ot\tau)( A^+\ot \tau \ot A^+)(d'\ot A^+\ot A^+)\ot\mu^-)\\&( A^+\ot \tau \ot A^-)\\
            &=( A^+\ot\varepsilon^+\ot\mu^-)( A^+\ot\varepsilon^+\ot ( A^+\ot\mu^-)(\tau\ot A^-)( A^-\ot\tau)\ot A^-)\\&(( A^+\ot\tau)(d'\ot A^+)\ot A^-\ot A^+\ot A^-)\\
            &=( A^+\ot\varepsilon^+\ot\mu^-)( A^+\ot\tau\ot A^-)( A^+\ot\varepsilon^+\ot\mu^-\ot A^+\ot A^-)\\
            &(( A^+\ot \tau)(d'\ot A^+)\ot A^-\ot A^+\ot A^-)\\
            &=( A^+\ot\varepsilon^+\ot\mu^-)( A^+\ot\tau\ot A^-)(L_{\tau,b'}\ot  A^+\ot A^-)\\
            &=\mu_{b'}(L_{\tau,b'}\ot L_{\tau,b'}).
        \end{split}
    \end{gather*}

    (2) is easy.

\end{proof}

\begin{remark}\label{Lprime}
In the setting of Proposition~\ref{morphismL}, by symmetry, we also have the following monoid morphisms:
\begin{gather*}
     L'_{\tau,b'}=(\mu^+\ot \varepsilon^-\ot  A^-)( A^+\ot \tau\ot  A^-) ( A^+\ot  A^-\ot d'): A_\tau\to A_{b'},\\
     K'_{b,b'}= A^+\ot(b\ot A^-)( A^-\ot d'):A_{b}\congto A_{b'}.
\end{gather*}
\end{remark}

\subsection{Precontractions}\label{section1precontraction}
We introduce precontractions, morphisms satisfying contraction identities modeled on the Brauer-category case (see Section~\ref{subsectionpsiforB}).
Let $A^+,A^-$ be augmented monoids in $\calM$.

By an \emph{action} of $A^-$ on $A^+$, we mean a morphism $\partial: A^-\ot A^+\to A^+$ satisfying 
\begin{gather*}
    \partial(\mu^-\ot  A^+)=\partial( A^-\ot \partial), \quad \partial(\eta^-\ot  A^+)= A^+,\quad \partial( A^-\ot \eta^+)=\eta^+\varepsilon^-.
\end{gather*}
(Here we do \emph{not} require any identity involving the multiplication $\mu^+$.)
Note that any twist $\tau':A^-\otimes A^+\to A^+\otimes A^-$ yields an action $\partial_{\tau'}=( A^+\ot \varepsilon^-)\tau'$.

Let $\partial$ be an action of $A^-$ on $A^+$.
We associate to $\partial$ a pairing 
$$b'=b_{\partial}:=\varepsilon^+\partial: A^-\otimes A^+\to I.$$
The action $\partial$ is said to be \emph{perfect} if the associated pairing is perfect.
In this case, we have a monoid structure $A_{b'}=A^+\otimes_{b'} A^-$ on the tensor product $A^+\otimes A^-$; see Section~\ref{Preliminaries}.
If the action $\partial$ is induced by a twist $\tau'$, then we have $b'=b_{\partial_{\tau'}}=b_{\tau'}$.

As is the case with Lemmas~\ref{relationbrho} and~\ref{lemmadrho}, we obtain the following relations.

\begin{lemma}
   Let $\partial$ be an action of $A^-$ on $A^+$, and let $b'=b_\partial$. Then we have 
    \begin{gather}
        \label{eq4}
    b'( A^-\ot\partial)=b'(\mu^-\ot A^+),\\
    \label{eq41}
    b'(\eta^-\ot A^+)=\varepsilon^+,\\
    \label{eq42}
    b'( A^-\ot\eta^+)=\varepsilon^-.
    \end{gather}
    If $\partial$ is perfect with duality $(b',d')$, then we have 
    \begin{gather}
    \label{eq43}
    (\varepsilon^+\ot A^-)d'
    =\eta^-,\\
    \label{eq44}
    ( A^+\ot\varepsilon^-)d'
    =\eta^+.
    \end{gather}
\end{lemma}

Let $\tau$ be a twist for $A^+$ and $A^-$, and $\partial$ a perfect action of $A^-$ on $A^+$ with duality $(b',d')$.
A \emph{precontraction} (for $A^+,A^-,\tau$ and $\partial$) is a morphism
$$\psi : A^+\to A^+\ot A^-$$
satisfying the following relations:
\begin{enumerate}
    \item[(P1)] $( A^+ \ot \mu^-)(\tau \ot  A^-)( A^-\ot \psi )=\psi \partial$,
    \item[(P2)] $(\mu^+\ot  A^-)( A^+\ot \tau)(\psi \ot  A^+)=\psi \ot \varepsilon^+$,
    \item[(P3)] $(\varepsilon^+\ot  A^-)\psi =\eta^-\varepsilon^+$,
    \item[(P4)] $( A^+\ot \mu^-)(\psi \ot A^-)d'=\eta^+\ot\eta^-$.
\end{enumerate}

\begin{remark}\label{A-A+bimodulestructures}
Let $\tau$ be a twist, and $\partial$ a perfect action with duality $(b',d')$ as above.
Define an $(A^-,A^+)$-bimodule structure on $A^+$ by
\begin{gather*}
    \partial\otimes \varepsilon^+:A^-\ot A^+\ot A^+\to A^+.
\end{gather*}
Define an $(A^-,A^+)$-bimodule structure on $A_{\tau}=A^+\ot A^-$ by
\begin{gather*}
    (\mu^+\ot \mu^-)( A^+\ot\tau\ot A^-)(\tau\ot\tau):A^-\ot A_\tau
    \ot A^+\to A_\tau.
\end{gather*}
Then the relations (P1) and (P2) mean that $\psi:A^+\to A_\tau$ is a morphism of $(A^-,A^+)$-bimodules.
Define an $(A^-,A^+)$-bimodule structure on $A_{b'}=A^+\ot A^-$ by
\begin{gather*}
    ( A^+\ot\varepsilon^+\ot A^-)(\partial\ot\tau):A^-\ot A_{b'}\ot A^+\to A_{b'}.
\end{gather*}
Then the monoid morphism $L_{\tau,b'}: A_{\tau}\to A_{b'}$ is a morphism of $(A^-,A^+)$-bimodules.
\end{remark}

\begin{theorem}
\label{precontraction}
Let $\tau$ be a twist for $A^+$ and $A^-$, and $\partial$ a perfect action of $A^-$ on $A^+$ with duality $(b',d')$.
Then the following conditions are equivalent.
\begin{enumerate}
\item The twist $\tau$ is perfect with duality $(b,d)$ such that $d$ is stack-invertible.
\item There exists a precontraction $\psi $.
\item The morphism $L_{\tau,b'}$ defined in Proposition~\ref{morphismL} (1) is an isomorphism.
\end{enumerate}
If these three conditions hold, then we have the following identities:
\begin{gather}
\label{eq-dtau}
    d=( A^+\ot \varepsilon^-\ot  A^-)(\psi \ot  A^-)d',\\
\label{eq-dtau-}
    d^-=\psi \eta^+
    ,\\
\label{eq-psitilde}
    \psi=(\mu^+\ot  A^-)((K^+_{b,b'})^{-1}\ot d^-) =(\mu^+\ot  A^-)( A^+\ot d^-)( A^+\ot b')(d\ot  A^+),\\ 
\label{eq-pstilde2}
    \psi =L_{\tau,b'}^{-1}( A^+\ot\eta^-),\\
\label{eq-Linverse}
    L_{\tau, b'}^{-1}=( A^+\ot \mu^-)(\psi \ot  A^-),
\end{gather}
where $d^-$ is the stack-inverse of $d$.
(It follows from \eqref{eq-psitilde} or \eqref{eq-pstilde2} that a precontraction is unique if it exists.)
\end{theorem}

\begin{proof}
(2)$\Rightarrow$(3).
It suffices to prove that for a precontraction $\psi $, the morphism $L^-:=( A^+\ot \mu^-)(\psi \ot  A^-):A^+\ot A^-\to A^+\ot A^-$ is inverse to $L_{\tau,b'}$.
We have 
    \begin{gather*}
        \begin{split}
            &L_{\tau, b'}L^-\\
            &=( A^+\ot \varepsilon^+\ot \mu^-)( A^+\ot ( A^+ \ot \mu^-)(\tau \ot  A^-)( A^-\ot \psi ) \ot A^-)(d'\ot A^+\ot A^-)\\
            &=( A^+\ot \varepsilon^+\ot \mu^-)( A^+\ot \psi \partial \ot A^-)(d'\ot A^+\ot A^-)\\
            &=( A^+\ot \mu^-)( A^+\ot \eta^-\ot  A^-)(( A^+\ot b')(d'\ot  A^+)\ot  A^-)\\
            &= ( A^+\ot b')(d'\ot A^+)\ot A^- \\
            &= \id_{A^+\ot A^-}
        \end{split}
    \end{gather*}
    by (P1) and (P3).
    We also have $L^-L_{\tau, b'}=\id_{A^+\ot A^-}$ by (P2) and (P4).
    Therefore, we obtain (3).

(3)$\Rightarrow$(2).
We will check that $\psi $ defined by \eqref{eq-pstilde2} is a precontraction, i.e. it satisfies (P1)--(P4).

(P1) and (P2):
The morphism $$ A^+\ot\eta^-:A^+\to A_{b'}$$
is a morphism of $(A^-,A^+)$-bimodules, where the $(A^-,A^+)$-bimodule structures are defined in Remark~\ref{A-A+bimodulestructures}.  
Since $L_{\tau,b'}:A_\tau\to A_{b'}$ is a morphism of $(A^-,A^+)$-bimodules, so is $\psi =L_{\tau,b'}^{-1}( A^+\ot\eta^-)$.
Therefore, we obtain (P1) and (P2) by Remark~\ref{A-A+bimodulestructures}.

(P3): We have $(\varepsilon^+\ot A^-)L_{\tau,b'}=\varepsilon^+\ot A^-$ by \eqref{eq43}.
By pre-composing $L_{\tau,b'}^{-1}$ we obtain
$\varepsilon^+\ot A^-=(\varepsilon^+\ot A^-)L_{\tau,b'}^{-1}$.
Hence 
\begin{gather*}
    (\varepsilon^+\ot A^-)\psi =(\varepsilon^+\ot A^-)L_{\tau,b'}^{-1}( A^+\ot\eta^-)
=(\varepsilon^+\ot A^-)( A^+\ot\eta^-)
=\eta^-\varepsilon^+.
\end{gather*}

(P4): We have $L_{\tau,b'}(\eta^+\ot\eta^-)=d'$.
Post-composing $L_{\tau,b'}^{-1}$ to this identity yields $L_{\tau,b'}^{-1} d'=\eta^+\ot \eta^-$.
Hence 
\begin{gather*}
\begin{split}
     ( A^+\ot \mu^-)(\psi \ot A^-)d'
     &=L_{\tau,b'}^{-1}L_{\tau,b'}( A^+\ot \mu^-)(\psi \ot A^-)d'\\
     &=L_{\tau,b'}^{-1}( A^+\ot\mu^-)(L_{\tau,b'}L_{\tau,b'}^{-1}\ot  A^-)( A^+\ot\eta^-\ot A^-)d'\\
     &=L_{\tau,b'}^{-1} d'\\
     &=\eta^+\ot\eta^-.
\end{split}
\end{gather*}

(1)$\Rightarrow$(2).
We will check that $\psi $ defined by \eqref{eq-psitilde} satisfies (P1)--(P4).
    For (P1), we have
    \begin{gather*}
        \begin{split}
            &( A^+ \ot \mu^-)(\tau \ot  A^-)( A^-\ot \psi )\\
            &=(\mu^+\ot  A^-)( A^+\ot ( A^+\ot \mu^-)(\tau\ot A^-)( A^-\ot d^-))\tau( A^-\ot ( A^+\ot b')(d\ot  A^+))\\
            &=(\mu^+\ot  A^-)( A^+\ot d^-\varepsilon^-)\tau( A^-\ot ( A^+\ot b')(d\ot  A^+)) \quad (\text{by }\eqref{eq30ee})\\
            &=(\mu^+\ot  A^-)( A^+\ot d^-)( A^+\ot b')( A^+\ot\mu^-\ot A^+)(d\ot A^-\ot A^+) \quad (\text{by }\eqref{eq20})\\
            &=\psi \partial.
        \end{split}
    \end{gather*}
    For (P2), we have
    \begin{gather*}
        \begin{split}
            &(\mu^+\ot  A^-)( A^+\ot \tau)(\psi \ot  A^+)\\
            &=(\mu^+\ot A^-)( A^+\ot (\mu^+\ot A^-)( A^+\ot \tau)(d^-\ot  A^+))(( A^+\ot b')(d\ot A^+)\ot A^+)\\
            &=(\mu^+\ot A^-)( A^+\ot d^-\varepsilon^+)(( A^+\ot b')(d\ot A^+)\ot A^+)\quad (\text{by } \eqref{eq3ee})\\
            &=\psi \ot \varepsilon^+.
        \end{split}
    \end{gather*}
    (P3) can be easily verified using \eqref{eq21}, \eqref{eq31ee} and \eqref{eq41}. 
    (P4) follows from
    \begin{gather*}
        \begin{split}
            ( A^+\ot \mu^-)(\psi \ot A^-)d'&=F_{d^-}( A^+\ot (b'\ot A^-)( A^-\ot d'))d\\
            &=F_{d^-} d\\
            &=\eta^+\ot\eta^-.
        \end{split}
    \end{gather*}

(2)$\Rightarrow$(1).
    Let $\psi $ be a precontraction.
    Then the morphism $d$ defined by \eqref{eq-dtau} is a copairing of $b$ since we have
    \begin{gather*}
        \begin{split}
            ( A^+\ot b)(d\ot A^+)
            &=( A^+\ot \varepsilon^-\ot\varepsilon^-)(\psi \ot\varepsilon^+\ot A^-)( A^+\ot\tau)(d'\ot A^+)\\
            &=(\mu^+\ot \varepsilon^-)( A^+\ot \tau)(( A^+\ot \mu^-)(\psi \ot A^-)d'\ot  A^+)\quad (\text{by (P2)})\\
            &= A^+ \quad (\text{by (P4)})
        \end{split}
    \end{gather*}
    and 
    \begin{gather*}
        \begin{split}
            (b\ot  A^-)( A^-\ot d)
            &=(\varepsilon^+\ot \varepsilon^-\ot  A^-)(( A^+ \ot \mu^-)(\tau \ot  A^-)( A^-\ot \psi )\ot  A^-)( A^-\ot d')\\
            &=(((\varepsilon^+\ot\varepsilon^-)\psi \partial)\ot A^-)( A^-\ot d') \quad (\text{by (P1)})\\
            &=(b'\ot  A^-)( A^-\ot d')\quad (\text{by (P3)})\\
            &= A^-.
        \end{split}
    \end{gather*}
    We can also check that  $d^-$ defined by \eqref{eq-dtau-} is a stack-inverse of $d$
    since
    \begin{gather*}
        \begin{split}
            F_{d}d^-&=( A^+\ot\varepsilon^-\ot A^-)((\mu^+\ot  A^-)( A^+\ot \tau)(\psi \ot  A^+) \ot  A^-)(\eta^+\ot d)\quad (\text{by }\eqref{eq20})\\
            &=( A^+\ot\varepsilon^-)\psi \eta^+\ot (\varepsilon^+\ot  A^-)d \quad(\text{by (P2)})\\
            &=( A^+\ot\varepsilon^-)( A^+\ot \mu^-)(\psi \ot A^-)d'\ot \eta^-\\
            &= \eta^+\ot \eta^-\quad(\text{by (P4)})
        \end{split}
    \end{gather*}
    and 
    \begin{gather*}
        \begin{split}
             F_{d}^-d
             &=(\mu^+\ot\mu^-)( A^+\ot  \psi \partial\ot A^-)(\psi \ot \eta^+\ot  A^-)d'\\
             &=( A^+\ot \mu^-)((\mu^+\ot\mu^-)( A^+\ot \tau\ot  A^-)(\psi \ot \psi \eta^+) \ot  A^-)d' \quad(\text{by (P1)})\\
             &=( A^+\ot \mu^-)(( A^+\ot \mu^-)(\psi \ot (\varepsilon^+\ot A^-)\psi \eta^+) \ot  A^-)d' \quad(\text{by (P2)})\\
             &=( A^+\ot \mu^-)(\psi \ot A^-)d'\quad(\text{by (P3)})\\
             &=\eta^+\ot\eta^- \quad(\text{by (P4)}).
        \end{split}
    \end{gather*}

Suppose that the three conditions in Theorem~\ref{precontraction} hold. Then, by the above arguments, we have \eqref{eq-dtau}--\eqref{eq-Linverse}.
\end{proof}

\begin{remark}\label{L'inverseL1}
The three conditions of Theorem~\ref{precontraction} are equivalent to:
\begin{itemize}
    \item[(3')] The monoid morphism $L'_{\tau,b'}$ defined in Remark~\ref{Lprime} is an isomorphism.
\end{itemize}
Indeed, (1)$\Leftrightarrow$(3) implies (1)$\Leftrightarrow$(3')  by symmetry.
If one, hence all, of (1), (2), (3) and (3') holds, then we have
\begin{gather*}
    (L'_{\tau,b'})^{-1}=F_{d}^{-1}(A^+ \ot (b'\ot A^+)(A^-\ot d)),
\end{gather*}
and, by composing it with $L_{\tau,b'}$, we obtain a canonical monoid automorphism
\begin{gather*}
    (L'_{\tau,b'})^{-1}L_{\tau,b'}:A_\tau\xrightarrow{\cong} A_\tau.
\end{gather*}
This automorphism is not the identity in general; we will see in Remark~\ref{L'inverseL} that it is non-trivial in the case of the Brauer category.
\end{remark}

\begin{remark}\label{actionendo}

Let $\Act(A_{b'},A^+)$ denote the set of left actions of the monoid $A_{b'}$ on $A^+$, and $\End_{\Mon}(A_{b'})$ the set of monoid endomorphisms of $A_{b'}$. Similarly for $\Act(A_\tau,A^+)$ and $\End_{\Mon}(A_\tau)$.
Let $\Aut_{\M}(A^+)$ denote the set of automorphisms of the object $A^+$ in $\M$.
If $L_{\tau,b'}$ is an isomorphism, then we have the following diagram:
\[
\xymatrix{
\Act(A_{b'},A^+)\ar[r]^{\gamma_{b'}}_\cong\ar[d]_{-\circ (L_{\tau,b'}\ot A^+)}^\cong&\End_{\Mon}(A_{b'})\ar[d]^{L_{\tau,b'}^{-1}\circ -\circ L_{\tau,b'}}_\cong&\ar[l]_-{\delta}\Aut_\M(A^+)
\\
\Act(A_\tau,A^+)\ar[r]^{\gamma_{\tau}}_\cong&\End_{\Mon}(A_\tau)\ar@{}[lu]|{\circlearrowleft}&
}
\]
where $\gamma_{b'}(\alpha)=(\alpha\ot A^-) (A_{b'}\ot d')$
for $\alpha\in \Act(A_{b'},A^+)$, and 
$\delta(f)=f\ot (f^{-1})^*$ for $f\in \Aut_\M(A^+)$ with $(f^{-1})^*=(b'\ot A^-)(A^-\ot f^{-1}\ot A^-)(A^-\ot d')$.
Here, 
the automorphism $(L'_{\tau,b'})^{-1}L_{\tau,b'}$ of $A_{\tau}$ considered in Remark~\ref{L'inverseL1} corresponds to the action $(A^+\ot b')(L_{\tau,b'}(L'_{\tau,b'})^{-1}L_{\tau,b'}\ot A^+)$ of $A_{\tau}$ on $A^+$.
\end{remark}

\subsection{The case where $\calM$ is a bimodule category}\label{subsectionRbim}

We specialize Theorems~\ref{Morita1} and~\ref{centerofAtau} to the case where $\M$ is the monoidal category $\Rbim$ of $R$-bimodules over a fixed $\k$-algebra $R$. Although $\Rbim$ is not strict since the tensor product $\otimes_R$ is associative only up to canonical isomorphisms, Mac~Lane's coherence theorem ensures that arguments for strict monoidal categories apply equally well in the non-strict setting. Consequently, we suppress the relevant coherence maps and treat $\Rbim$ as strict in what follows.

A monoid in $\Rbim$ is known as an \emph{$R$-ring}, which consists of a $\k$-algebra $A$ and a $\k$-algebra homomorphism $i:R\to A$.
An augmented monoid in $\Rbim$ is an \emph{augmented $R$-ring}, which is an $R$-ring $(A,i)$ equipped with an $R$-bimodule map
$\varepsilon:A\to R$
that is also a $\k$-algebra homomorphism.

Let $A^\pm=(A^\pm,\mu^\pm,\eta^\pm,i^\pm,\varepsilon^\pm)$
be augmented $R$-rings.
A twist for $A^+$ and $A^-$ is an $R$-bimodule map
\[
\tau: A^-\ot_R A^+ \to A^+\ot_R A^-
\]
satisfying the identities \eqref{twist}.
The twisted tensor product $A_\tau$ of $A^+$ and $A^-$ with twist $\tau$ is the $R$-ring with underlying $R$-bimodule $A^+\ot_R A^-$ equipped with the multiplication
\[
\mu_\tau: A^+\ot_R A^-\ot_R A^+\ot_R A^- 
\xrightarrow{\id_{A^+}\ot \tau\ot \id_{A^-}}
A^+\ot_R A^+\ot_R A^- \ot_R A^-
\xrightarrow{\mu^+\ot \mu^-}
A^+\ot_R A^-
\]
and unit $\eta_\tau=\eta^+\ot\eta^-:R\to A^+\ot_R A^-$.
We have a pairing $b=(\varepsilon^+\ot\varepsilon^-)\tau:A^-\ot_R A^+\to R$ in $\Rbim$.
The twist $\tau$ is \emph{perfect} if there is a copairing
\[d:R\to A^+\ot_R A^-\]
in $\Rbim$ such that $(A^+\ot b)(d \ot A^+)=\id_{A^+}$ and $(b\ot A^-)(A^-\ot d)=\id_{A^-}$.

The center $Z^{\Rbim}(R)$ coincides with the usual center $Z(R)$ of $R$, and $Z^{\Rbim}(A_{\tau})$ coincides with the center $Z(A_{\tau})$ of $A_{\tau}$ defined by
\begin{gather*}
    Z(A_{\tau})=\{a\in A_{\tau}\mid a*b=b*a \text{ for all }b\in A_{\tau}\},
\end{gather*}
where $a*b=\mu_{\tau}(a\otimes b)$.

Theorems~\ref{Morita1} and~\ref{centerofAtau} imply the following.

\begin{theorem}\label{Morita3}
Let $R$ be an algebra over a commutative ring $\k$.
Let $A^+$ and $A^-$ be augmented $R$-rings, and $\tau:A^-\ot_R A^+\to A^+\ot_R A^-$ a perfect twist with copairing $d$ which has stack-inverse $d^-$. Then we have the following.
\begin{enumerate}
    \item The $\k$-algebras $R$ and $A_\tau$ are Morita equivalent.
    \item We have a $\k$-algebra isomorphism
    \begin{gather*}
        \zeta_\tau : Z(R)\congto Z(A_\tau) 
    \end{gather*}
    defined by 
    \begin{gather*}
        \zeta_\tau(r)= \sum_i d^+_i\,r\,d^-(1)\,d^-_i
    \end{gather*}
    for $r\in R$, 
    where $d(1)=\sum_i d^+_i\ot d^-_i\in A^+\ot_R A^-$.
\end{enumerate}
\end{theorem}

\begin{example}
Let $\k$ be an algebraically closed field of characteristic $0$, and let $q\in\k$ be a primitive $N$th root of unity with $N\ge2$.
Consider the truncated $q$-Weyl algebra $A=\k\lara{x,y}/(yx-qxy-1,x^N,y^N)$.
Set $A^+=\k[x]/(x^N)$ and $A^-=\k[y]/(y^N)$.
We have a linear isomorphism $A^+\otimes_\k A^-\congto A$ given by multiplication. Hence $A$ is isomorphic to the twisted tensor product $A_\tau=A^+\otimes_\tau A^-$ for some twist $\tau:A^-\otimes_\k A^+\to A^+\otimes_\k A^-$.
The twist $\tau$ is perfect and stack-invertible.
Therefore, by Theorem~\ref{Morita3}, $\k$ and $A_\tau\cong A$ are Morita equivalent.  In fact, it is known that $A$ is isomorphic to the matrix algebra $M_N(\k)$,
    see e.g. \cite{Heider-Wang}.
Note also that the stack multiplication gives $\k\text{-bim}(\k,A_\tau)\cong A^+\ot_\k A^-$ a commutative algebra structure $\k[x,y]/(x^N,y^N)$.
\end{example}

\begin{example}  \label{heisenberg}
Let $A^+$ and $A^-$ be finite-dimensional Hopf algebras over a field $\k$.  
Let $b:A^-\ot\ A^+ \to\k$ be a \emph{Hopf pairing}.
Then we have a twist $\tau:A^-\ot A^+\to A^+\ot A^-$ defined by
\begin{gather*}
     \tau(x\ot y) = \sum \langle x_{(1)},y_{(2)}\rangle  y_{(1)}\ot x_{(2)},
\end{gather*}
where 
$\Delta^-(x)=\sum x_{(1)}\ot x_{(2)}$ and
$\Delta^+(y)=\sum y_{(1)}\ot y_{(2)}$.
The associated pairing $b_\tau$ is equal to $b$.
The \emph{Heisenberg double} $H$ of $A^+$ and $A^-$ with respect to the Hopf pairing $b$ can be defined as the twisted tensor product $A_\tau$.
If the pairing $b$ is perfect, then
it is well known that $H$
is isomorphic to the matrix algebra $\End_\Vect(A^+)$,
see e.g. \cite[Proposition 1.1]{Militaru}, and hence is Morita equivalent to $\k$.
This fact can be recovered also from Theorem~\ref{Morita3}.
\end{example}

Theorem~\ref{Morita3} is naturally generalized to the case of linear categories.
We study the case of stratified linear categories in Section~\ref{sec:Perfect stratified linear categories}.

\section{Stratified linear categories}\label{sectionstrlincat}
Over a commutative ring $\k$, we define stratified linear categories over a base category and identify them with twisted tensor products in the bimodule category. The notion is a reformulation of existing structures; see Remark~\ref{rem:notion}.

\subsection{Base categories}

We fix a linear category $\A$, which we call the \emph{base category}, with the following properties:
\begin{enumerate}
    \item The set $\Ob(\A)$ is equipped with a partial order $\le$ with descending chain condition.
    \item If $x,y\in\Ob(\A)$ with $x\neq y$, then we have $\A(x,y)=0$.
\end{enumerate}
We use the notation $\A_x:=\A(x,x)$, which is an algebra.
Note that a base category $\A$ is essentially a family of algebras $\A_x$ for $x\in \Ob(\A)$.

By a \emph{linear category over the base category $\A$}, we mean a linear category $\C$ with $\Ob(\C)=\Ob(\A)$, equipped with a linear functor $i:\A\to\C$ which is the identity on objects.
Note that for $x,y\in\Ob(\A)$ the hom-set $\C(x,y)$ has an $(\A_y,\A_x)$-bimodule structure.

By a linear functor $F:(\C,i)\to(\C',i')$ between linear categories $(\C,i)$ and $(\C',i')$ over $\A$, we mean a linear functor $F:\C\to\C'$ such that $F\circ i=i'$.

\subsection{Stratified linear categories}

A \emph{stratified linear category} over the base category $\A$ is a linear category $(\C,i)$ over $\A$
equipped with 
    wide linear subcategories $\C^+$ and $\C^-$ of $\C$
    \footnote{Recall that a (linear) subcategory $\C'$ of $\C$ is \emph{wide} if $\Ob(\C')=\Ob(\C)$.}
satisfying the following conditions:
\begin{enumerate}
    \item The linear categories $\C^+$ and $\C^-$ are upward and downward, respectively,
    in the sense that for each $x,y\in \Ob(\C)$, we have
    \begin{itemize}
        \item $\C^+(x,y)=0$ unless $x\le y$, 
        \item $\C^-(x,y)=0$ unless $x\ge y$. 
    \end{itemize}
    \item For each $x\in \Ob(\C)$, we have $\C^+(x,x)=\C^-(x,x)$.
    (At this point, we set $\C^0(x,x)=\C^\pm(x,x)$ and $\C^0(x,y)=0$ for $x\neq y$. Then the $\C^0(x,y)$ form a wide linear subcategory of both $\C^+$ and $\C^-$ and hence of $\C$. We also set $\C^0_x:=\C^0(x,x)$.)
    \item For each $x\in\Ob(\A)$, the functor $i:\A\to\C$ maps $\A_x=\A(x,x)$ isomorphically onto $\C^0(x,x)$.
    (Thus, $i$ induces an isomorphism of linear categories $i:\A\congto\C^0$.)
    \item For each $x,z\in\Ob(\C)$, the composition map
    \begin{gather}\label{composition-map}
        \circ : \bigoplus_{y\in\Ob(\C)} \C^+(y,z)\ot_{\C^0_y}\C^-(x,y)\to \C(x,z),
        \quad g\ot f\mapsto gf
    \end{gather}
    is a linear isomorphism.
\end{enumerate}

We will identify the categories $\A$ and $\C^0$ via the isomorphism $i:\A\xrightarrow{\cong}\C^0$.
Note that the map $\circ$ in \eqref{composition-map} is an $(\A_z,\A_x)$-bimodule map, and hence it gives a bimodule isomorphism.

For a stratified linear category $\C$, we have projections 
\begin{gather}\label{projectiontoC0}
    \varepsilon^{\pm}:\C^{\pm}\twoheadrightarrow\C^0, 
\end{gather}
which are linear functors defined by $\varepsilon^{\pm}=\id:\C^{\pm}(x,y)\to \C^0(x,y)$ if $x=y$ and $\varepsilon^{\pm}=0$ otherwise.

Let $\C=(\C,\C^+,\C^-,\C^0,i_\C)$ and $\D=(\D,\D^+,\D^-,\D^0,i_\D)$ be stratified linear categories over $\A$.
A \emph{morphism} $F:\C\to\D$ of stratified linear categories over $\A$
is a linear functor $F:(\C,i_\C)\to(\D,i_\D)$ of linear categories over $\A$  such that
we have $F(\C^{\pm}(x,y))\subset \D^{\pm}(x,y)$ for $x,y\in\Ob(\C)$.

    For a stratified linear category $(\C,\C^+,\C^-,\C^0,i)$ over a base category $\A$, the linear category structure of $\C$ and the isomorphism \eqref{composition-map} induce a linear category structure on the tensor product $\C^+\ot_{\C^0}\C^-$.
    Then the tuple $(\C^+\ot_{\C^0}\C^-,\C^+,\C^-,\C^0,i:\A\xrightarrow{\cong}\C^0\hookrightarrow\C^+\ot_{\C^0}\C^-)$ forms a stratified linear category over $\A$.
    The composition map \eqref{composition-map} induces an isomorphism of stratified linear categories $\circ:\C^+\ot_{\C^0}\C^-\xrightarrow{\cong}\C$, which we call the \emph{canonical isomorphism}.

\begin{remark}\label{rem:notion}
A stratified linear category may be loosely viewed as a linear-categorical analogue of a (standardly) stratified algebra in the sense of Cline--Parshall--Scott \cite{CPS96}.
Related notions---categories equipped with a well-founded filtration on objects and classes of upward, downward, and level morphisms---appear in several contexts.
Reedy categories and their generalizations have been introduced and studied in the context of model categories \cite{Reedy1974,Cisinski2006,BergerMoerdijk2011,Shulman2015}. Strictly object-adapted cellular categories in the sense of Elias--Lauda \cite{Elias-Lauda} whose cellular bases are closed under composition, such as the Temperley--Lieb category, can be regarded as stratified linear categories. Triangular categories in the sense of Sam--Snowden \cite{Sam-Snowden-BrauerI} are closely related but distinct from stratified linear categories.
The term ``stratified linear category'' was used in talks by the first author \cite{Htalk18,Htalk19} concerning the Hochschild--Mitchell homology of such categories.
\end{remark}

\subsection{The category $\Abim$ of $\A$-bimodules}

Let $\k\textrm{-mod}$ denote the category of $\k$-modules. 
An \emph{$\A$-bimodule} $M$ is a $\k$-linear functor $M:\A^{\op}\otimes_{\k} \A\to \k\textrm{-mod}$.
By the property (2) of the base category $\A$, an $\A$-bimodule is a family $M(x,y)$ of $(\A_y,\A_x)$-bimodules for $x,y\in\Ob(\A)$.
A morphism $f:M\to N$ of $\A$-bimodules
is a family of $(\A_y,\A_x)$-bimodule maps $f_{x,y}:M(x,y)\to N(x,y)$.
Let $\Abim$ denote the category of $\A$-bimodules and morphisms, where the composition and identity morphisms are defined in the obvious way.
We define a monoidal structure of $\Abim$ as follows.
For $\A$-bimodules $M$ and $N$, the tensor product $M\ot N$ in $\Abim$, denoted by $M\ot_\A N$, is defined by
$$(M\ot_{\A} N)(x,y)=\bigoplus_{z\in\Ob(\A)}M(z,y)\ot_{\A_z} N(x,z)$$
for $x,y\in\Ob(\A)$.
The monoidal unit is $\A$ considered as an $\A$-bimodule, i.e., the family $\A(x,y)$ for $x,y\in\Ob(\A)$.
It is straightforward to verify that there are associativity and unitality natural isomorphisms satisfying the coherence conditions.
Thus, we have a monoidal category $\Abim$.
Mac Lane's coherence theorem allows us to think of $\Abim$ as a strict monoidal category so that we can apply the results in Section~\ref{sec-twisted-tensor-product} to $\Abim$.

\subsection{Stratified linear categories as twisted tensor products}

An $\A$-bimodule $M$ is said to be \emph{upward} (resp. \emph{downward}) if we have $M(x,y)=0$ unless $x\le y$ (resp. $x\ge y$) for $x,y\in \Ob(\A)$. 

Let $(\C,\mu,\eta)$ be a monoid in $\Abim$.
We say that $\C$ is \emph{strongly upward} (resp.\ \emph{strongly downward}) if
\begin{enumerate}
\item $\C$ is upward (resp.\ downward), and
\item for every $x\in\Ob(\A)$, the component
\[
\eta_x:\ \A_x\longrightarrow \C(x,x)
\]
is an isomorphism of $\A_x$–bimodules.
\end{enumerate}
A strongly upward or downward monoid is augmented since we have an augmentation $\varepsilon:\C\to\A$ in $\Abim$ defined by
$\varepsilon_x=\eta_x^{-1}:\C(x,x)\to\A_x$.

In the rest of this section, we will prove the following result, which allows us to study stratified linear categories within the framework of twisted tensor products.

\begin{theorem}\label{stratifiedlinearcategoryandtwistedtensorproduct}
Stratified linear categories over $\A$ are in one-to-one correspondence, up to isomorphisms, with the twisted tensor products of a strongly upward monoid and a strongly downward monoid in $\Abim$.
\end{theorem}

\subsection{Monoids in $\Abim$}

We describe monoids in the monoidal category $\Abim$ in terms of linear categories and linear functors.

Recall that a monoid $\C$ in $\Abim$ consists of an $\A$-bimodule $\C$, a multiplication $\mu: \C\ot_{\A}\C\to \C$ and a unit $\eta:\A\to \C$ satisfying associativity and unitality.
The $\A$-bimodule $\C$ consists of $(\A_y,\A_x)$-bimodules $\C(x,y)$.
The multiplication $\mu$ consists of $(\A_y,\A_x)$-bimodule maps
$$\circ=\circ_{x,y}:\bigoplus_{z\in\Ob(\A)}\C(z,y)\ot_{\A_z}\C(x,z)\to\C(x,y)$$
for $x,y\in\Ob(\A)$,
which are equivalent to 
$(\A_y,\A_x)$-bimodule maps
$$\circ=\circ_{x,z,y}: \C(z,y)\ot_{\A_z}\C(x,z)\to\C(x,y)$$
for $x,y,z\in\Ob(\A)$.
The unit $\eta$ consists of $(\A_x,\A_x)$-bimodule maps
$$\eta_x: \A_x\to \C(x,x)$$
for $x\in\Ob(\A)$, which are determined by the element $\id^\C_x:=\eta_x(\id^\A_x)\in \C(x,x)$, where $\id^\A_x:x\to x$ is the identity morphism of $x$ in the category $\A$.

\begin{lemma}\label{monoidsandlinearcategories}We have the following.
    \begin{enumerate}
        \item Monoids in $\Abim$ are in one-to-one correspondence with linear categories over $\A$.
    \item Let $\C$ and $\C'$ be monoids in $\Abim$,
    and let $(\C,\eta)$ and $(\C',\eta')$ be the corresponding linear categories over $\A$.
    Then morphisms $F:\C\to \C'$ of monoids in $\Abim$ are in one-to-one correspondence with linear functors $F:(\C,\eta)\to (\C',\eta')$ over $\A$.
    \end{enumerate}
\end{lemma}

The above lemma holds for any linear category $\A$.
We here give a concrete proof that relies on the special structure of the base category $\A$.

\begin{proof}[Proof of Lemma~\ref{monoidsandlinearcategories}]
(1) Each monoid $\C=(\C,\mu,\eta)$ in $\Abim$ gives rise to a linear category $\C$ over $\A$ as follows.
The object set is given by $\Ob(\C)=\Ob(\A)$.
The $(\A_y,\A_x)$-bimodules $\C(x,y)$ for $x,y\in\Ob(\A)$ serve as the hom-spaces of $\C$.
The composition in the category $\C$ is given by
$$\circ: \C(z,y)\ot\C(x,z)\twoheadrightarrow\C(z,y)\ot_{\A_z}\C(x,z)\xrightarrow{\circ_{x,z,y}}\C(x,y),$$
and the identity morphism is $\id^\C_x\in \C(x,x)$,
where $\circ_{x,z,y}$ and $\id^\C_x$ are given above Lemma~\ref{monoidsandlinearcategories}.
Moreover, by considering $\C$ as a linear category, the unit $\eta:\A\to \C$ gives rise to a linear functor $\eta:\A\to \C$, which is the identity on objects.

Conversely, a linear category $(\C,\eta)$ over $\A$ gives rise to a monoid $\C$ in $\Abim$ as follows. 
The hom-space $\C(x,y)$ has the structure of an $(\A_y,\A_x)$-bimodule defined by
\begin{gather*}
   \A_y\otimes \C(x,y)\otimes \A_x\xrightarrow{\eta_y\otimes\id\otimes\eta_x}\C(y,y)\otimes \C(x,y)\otimes \C(x,x)\xrightarrow{\circ(\id\otimes\circ)} \C(x,y).
\end{gather*}
The composition map $\circ:\C(z,y)\ot \C(x,z)\to \C(x,y)$ is an $(\A_y,\A_x)$-bimodule map which factors through $\C(z,y)\ot_{\A_z} \C(x,z)$, and the linear functor $\eta$ gives a family $\eta_x$ of $(\A_x,\A_x)$-bimodule maps. It is easy to see that the associativity and unitality hold.

(2) A morphism $F:(\C,\mu,\eta)\to (\C',\mu',\eta')$ of monoids in $\Abim$ consists of a family of linear maps $F_{x,y}:\C(x,y)\to\C'(x,y)$ for $x,y\in\Ob(\A)$ satisfying certain properties.
A linear functor $F:(\C,\eta)\to(\C',\eta')$ of linear categories over $\A$ also consists of such a family of linear maps.  The defining properties of these two cases are the same: $F(\id_x^\C)=\id_x^{\C'}$ for $x\in\Ob(\A)$ and $F(g\circ f)=F(g)\circ F(f)$ for composable pairs $(g,f)$ of morphisms in $\C$.
Therefore there is a one-to-one correspondence between these two notions.
\end{proof}

\subsection{Augmented monoids in $\Abim$}

Now let $\C$ be an augmented monoid in $\Abim$. The counit $\varepsilon:\C\to\A$ consists of $\A_x$-bimodule maps
$$\varepsilon_x=\varepsilon_{x,x}:\C(x,x)\to\A_x$$
for $x\in \Ob(\A)$.

As an extension of Lemma~\ref{monoidsandlinearcategories}(1), we easily obtain the following description of augmented monoids in $\Abim$.

\begin{lemma}\label{augmonoidsandlinearcategories}
    Augmented monoids in $\Abim$ are in one-to-one correspondence with triples $(\C,\eta,\varepsilon)$ consisting of a linear category $(\C,\eta)$ over $\A$ and a linear functor $\varepsilon: \C\to\A$ which is the identity on objects such that $\varepsilon\eta=\id_\A$.
\end{lemma}

\subsection{Twisted tensor products of strongly upward and downward monoids in $\Abim$}

Let $\C^+=(\C^+,\mu^+,\eta^+,\varepsilon^+)$ and $\C^-=(\C^-,\mu^-,\eta^-,\varepsilon^-)$ be augmented monoids in $\Abim$ with $\C^+$ strongly upward and $\C^-$ strongly downward.

A \emph{twist} for $\C^+$ and $\C^-$ is a morphism
$$\tau:\C^-\otimes_\A \C^+\to \C^+\otimes_\A \C^-$$ in $\Abim$ satisfying \eqref{twist}.
For a twist $\tau$ for $\C^+$ and $\C^-$, we have a monoid $\C_\tau=\C^+\ot_{\tau}\C^-=(\C^+\ot_{\A} \C^-,\mu_\tau,\eta_\tau)$ in $\Abim$, where $\mu_\tau$ and $\eta_\tau$
are given by \eqref{mu-eta}.

\subsection{Proof of Theorem~\ref{stratifiedlinearcategoryandtwistedtensorproduct}}
Let $\mathcal{S}_\A$ denote the class of stratified linear categories over $\A$.
Let $\mathcal{T}_\A$ denote the class of twisted tensor products in $\Abim$, or more precisely, the triples $(\C^+,\C^-,\tau)$ consisting of a strongly upward monoid $\C^+$, a strongly downward monoid $\C^-$ and a twist $\tau:\C^-\ot_{\A}\C^+\to\C^+\ot_{\A}\C^-$ in $\Abim$.
Note that the notion of isomorphisms is defined on $\SA$ and $\TA$.
We are going to define a map $F:\SA\to\TA$ and a map $G:\TA\to \SA$ and
show that they are inverses up to isomorphisms.

First, we define a map $F:\SA\to \TA$.
Let $\C\in\SA$ be a stratified linear category over $\A$.
Then $\C$ can be considered as a monoid in $\Abim$ by Lemma~\ref{monoidsandlinearcategories} (1).
Moreover, the wide subcategories $\C^{\pm}$ of $\C$ with $i: \A\xrightarrow{\cong} \C^0\hookrightarrow \C^{\pm}$ and $\varepsilon^{\pm}$ defined in \eqref{projectiontoC0} can be considered as augmented monoids in $\Abim$ by Lemma~\ref{augmonoidsandlinearcategories}.
By definition, $\C^+$ is strongly upward and $\C^-$ is strongly downward.
Define an $\A$-bimodule map 
\begin{gather}\label{tauforstr}
    \tau:\C^-\otimes_{\A} \C^+\to \C^+\otimes_{\A} \C^-
\end{gather}
as the composite of the composition map $\circ: \C^-\ot_{\A}\C^+ \to \C$ and the inverse $\circ^{-1}: \C\xrightarrow{\cong} \C^+\ot_{\A}\C^-$ to the composition map \eqref{composition-map}, which is an isomorphism.

\begin{lemma}\label{strtotau}
    The morphism $\tau$ is a twist in $\Abim$.
\end{lemma}

\begin{proof}
    We only check the identity 
    $\tau (\mu^-\otimes \C^+)= (\C^+\otimes \mu^-)(\tau\otimes \C^-)(\C^-\otimes \tau)$; the other identities similarly follow.
    It follows from $\circ \tau=\circ:\C^-\ot_{\A}\C^+\to \C$ that we have
    $$
    \circ \tau (\mu^-\otimes {\C^+})=\circ (\circ \ot {\C^+})=\circ({\C^-}\ot \circ):\C^-\ot_{\A}\C^-\ot_{\A}\C^+\to \C
    $$
    and 
    \begin{gather*}
    \begin{split}
         \circ ({\C^+}\otimes \mu^-)(\tau\otimes {\C^-})({\C^-}\otimes \tau)
         &=\circ (\circ\ot{\C^-})(\tau\otimes {\C^-})({\C^-}\otimes \tau)\\
         &=\circ (\circ\ot{\C^-})({\C^-}\otimes \tau)\\
         &=\circ ({\C^-}\ot \circ)({\C^-}\otimes \tau)\\
         &=\circ ({\C^-}\ot \circ).
    \end{split}
    \end{gather*}
    By composing $\circ^{-1}:\C\to \C^+\ot_{\A}\C^-$, we obtain the identity.
\end{proof}

Now we obtain a map $F:\SA\to \TA$ defined by $F(\C)=(\C^+,\C^-,\tau)$.

Next, we define a map $G:\TA\to \SA$.
For $(\C^+,\C^-,\tau)\in \TA$, let  $\C_\tau$ be its twisted tensor product.

\begin{lemma}
    The twisted tensor product $\C_\tau$ is a stratified linear category over $\A$.
\end{lemma}

\begin{proof}
    Since $\C_\tau$ is a monoid in $\Abim$, by Lemma~\ref{monoidsandlinearcategories} (1), $\C_\tau$ has a structure of a linear category over $\A$.
    The same holds when we replace $\C_\tau$ with $\C^\pm$.
    
    We see that $\C^+$ embeds into $\C_\tau$ as a wide subcategory via injective maps $i^+=i^+_{x,y}:\C^+(x,y)\to \C_\tau(x,y)$ defined by $i^+(f)=f\otimes \id_x$ for $f\in\C^+(x,y)$.
    Similarly $\C^-$ embeds into $\C_\tau$ as a wide subcategory.

    We will check the four conditions (1)--(4) in the definition of stratified linear categories.
    Conditions (1) and (2) are satisfied since the monoids $\C^+$ and $\C^-$ in $\Abim$ are strongly upward and downward, respectively.
    Let $\C^0$ be the image of $\eta:\A\to\C$, which is the wide linear subcategory of $\C$ such that $\C^0(x,y)=\eta(\A(x,y))$ for $x,y\in\Ob(\C)$; then we have (3).
    Condition (4) follows since the restriction
    \begin{gather*}
        \circ : \bigoplus_{y\in\Ob(\A)} \C^+(y,z)\ot_{\A_y}\C^-(x,y)\to \C_\tau(x,z)
    \end{gather*}
    of the composition map is the identity.
\end{proof}

Lastly, we show that $F$ and $G$ are inverses up to isomorphisms.
It is obvious that we have $F\circ G=\id_{\TA}.$
For any $\C\in \SA$, we have the canonical isomorphism $G\circ F(\C)=\C_\tau\xto{\cong} \C$.
This completes the proof of Theorem~\ref{stratifiedlinearcategoryandtwistedtensorproduct}.

\section{Properties of stratified linear categories}
\label{sec:Perfect stratified linear categories}
We now apply the results for monoidal categories in Section~\ref{sec-twisted-tensor-product} to stratified linear categories via Theorem~\ref{stratifiedlinearcategoryandtwistedtensorproduct}.

\subsection{Stack-invertibility in stratified linear categories}\label{sectionstackinvinstrlincat}

Let $\C=\C_\tau$ be a stratified linear category over $\A$.
We apply the definition of two monoid structures on $M(A_\tau)$ in Section~\ref{sec:canonical-idempotents} to $\calM=\Abim$ and $A_\tau=\C_\tau$.
The set $M(\C)=\Abim(\A,\C)$ consists of $\A$-bimodule maps $u:\A\to\C$.
Here, an $\A$-bimodule map $u:\A\to\C$ is a family of $\A_x$-bimodule maps $u:\A_x\to\C(x,x)$, which are determined by the elements $u(1_x)$ of the centralizers $Z_{\C(x,x)}(\A_x)$ of $\A_x$ in $\C(x,x)$.

The stack multiplication
\begin{gather*}
    \bullet: M(\C)\times M(\C) \to M(\C)
\end{gather*}
is given by
\begin{gather}
\label{u-star-v}
    (v\bullet u)(1_x)=\sum_{y\le x} u^+_{x,y} v(1_y) u^-_{x,y},
\end{gather}
where we express $u(1_x)\in \C(x,x)\cong \bigoplus_{y\le x}\C^+(y,x)\ot_{\A_y}\C^-(x,y)$ as a finite sum
\begin{gather}\label{u1x}
u(1_x)=\sum_{y\le x}u^+_{x,y}\ot u^-_{x,y}    
\end{gather}
with $u^+_{x,y}\in\C^+(y,x)$, $u^-_{x,y}\in\C^-(x,y)$.
Although the expression for $u(1_x)$ in \eqref{u1x} is not unique,
the definition of $v\bullet u$ in \eqref{u-star-v} is well defined since $v(1_y)\in Z_{\C(y,y)}(\A_y)$.

The other multiplication in $M(\C)$
$$*:M(\C)\times M(\C)\to M(\C)$$ is 
given by
$(u* v)(1_x)=u(1_x)\circ v(1_x)$ for $x\in\Ob(\A)$, $u,v\in M(\C)$,
where $\circ$ denotes the composition of morphisms in $\C(x,x)$.
We have a $\k$-algebra isomorphism
$$(M(\C),*)\cong \prod_{x\in\Ob(\A)}Z_{\C(x,x)}(\A_x).$$

In Section~\ref{sec:canonical-idempotents}, we defined the subset $M^\adm(A_\tau)$ of admissible elements of $M(A_\tau)$.
For stratified linear categories, we have the following.

\begin{lemma}\label{admissibleforstr}
For a stratified linear category $\C$ over $\A$, we have $M(\C)=M^\adm(\C)$. In other words, any $\A$-bimodule map $u:\A\to\C$ is admissible in the sense of Section~\ref{sec:canonical-idempotents}.
\end{lemma}

\begin{proof}
For $x\in\Ob(\A)$, by the expression \eqref{u1x} of $u(1_x)$, we have 
$$(\varepsilon^+\ot\idC-)u(1_x)=u_{x,x}^+ u_{x,x}^-=\eta^-(\varepsilon^+\ot\varepsilon^-)u(1_x)\in \C^-(x,x).$$
Since an element of $M(\C)$ is determined by the images of $1_x$, we have $(\varepsilon^+\ot\idC-)u=\eta^-(\varepsilon^+\ot\varepsilon^-)u$.
Similarly, we have $(\idC+\ot\varepsilon^-)u=\eta^+(\varepsilon^+\ot\varepsilon^-)u$.
Therefore, $u$ is admissible.
\end{proof}

By Lemma~\ref{admissibleforstr}, a strongly admissible element of $M(\C)$ is an $\A$-bimodule map $u:\A\to\C$ such that $(\varepsilon^+\ot\varepsilon^-)u=\id_\A$, which is equivalent to $(\varepsilon^+\ot\varepsilon^-)u(1_x)=1_x$ for each $x\in\Ob(\A)$.

\begin{proposition}\label{stack-invertibility-stratified}
For a stratified linear category $\C$ over $\A$, the monoid $M^\sadm(\C)$ with stack multiplication is a group.
In other words, any strongly admissible map $u: \A\to \C$ has a strongly admissible stack-inverse.
\end{proposition}

\begin{proof}
By letting $u_{x,y}:=u_{x,y}^+\ot u_{x,y}^-\in \C^+(y,x)\ot_{\A_y}\C^-(x,y)$ in \eqref{u1x}, we express $u(1_x)$ as a (unique) finite sum
$$u(1_x)=\sum_{y\le x}u_{x,y}.$$
Since $u$ is strongly admissible, we have $(\varepsilon^+\ot \varepsilon^-)u(1_x)=1_x$, and thus $u_{x,x}=1_x$.

Set $u'=u-\eta\in M(\C)$.
Then for each $x\in\Ob(\A)$, we have 
$$u'(1_x)=u(1_x)-\eta(1_x)=(1_x+\sum_{y<x}u_{x,y})-1_x=\sum_{y<x} u_{x,y}.$$
For $k\ge0$, define $(u')^{\bullet k}$ to be the $k$-th power of $u'$ in the 
stack multiplication, i.e., it is defined inductively by $(u')^{\bullet 0}=\eta$, $(u')^{\bullet k+1}=u'\bullet (u')^{\bullet k}$ for $k\ge0$.
We claim that for each $x\in\Ob(\A)$, the sum
\begin{gather}
\label{uxminus}
    u^-_x := \sum_{k\ge0} (-1)^k (u')^{\bullet k}(1_x)
\end{gather}
is well defined, i.e., $(u')^{\bullet k}(1_x)=0$ for sufficiently large $k$, which depends on both $u$ and $x$.
By computation, for each $x\in\Ob(\A)$, we have
\begin{gather}\label{uprimestark}
    (u')^{\bullet k}(1_x)=\sum_{y_k<y_{k-1}<\dots<y_1<x} 
     u_{y_{k-1},y_k} \bullet\dots\bullet u_{y_1,y_2} \bullet u_{x,y_1}
    \in\C(x,x).
\end{gather}

Let $\prec$ denote the binary relation on $\Ob(\A)$ such that 
$y\prec x$ if and only if $x\neq y$ and $u_{x,y}\neq0$.
Since $y\prec x$ implies $y<x$ and since the partial order $\le$ satisfies the descending chain condition, it follows that $\prec$ is well-founded.
Since for each $x\in\Ob(\A)$, there are only finitely many $y$ with $y\prec x$, it follows that the sum \eqref{uprimestark} vanishes if $k$ is sufficiently large.
Thus, the sum \eqref{uxminus} is well defined.

Since $u'\in M(\C)$, by the definition \eqref{uxminus}, the elements $u^-_x$ define $u^-\in M(\C)$.
We will check that $u^-$ is a stack-inverse of $u$.
In fact, we have
\begin{gather*}
    \begin{split}
        (u^-\bullet u)(1_x)&=\sum_{k\ge 0}(-1)^k \sum_{y_k<\cdots <y_1<y\le x}
        u_{y_{k-1},y_k} \bullet\dots\bullet u_{y_1,y_2} \bullet u_{y,y_1}\bullet u_{x,y}\\
        &=(-1)^0 u_{x,x}
        +(-1)^0\sum_{y< x}u_{x,y}
        +(-1)^1\sum_{y_1<y= x}
         u_{y,y_1}\bullet u_{x,y}\\
        &+\sum_{k\ge 2}(-1)^k \sum_{y_k<\cdots <y_1<y= x}
        u_{y_{k-1},y_k} \bullet\dots\bullet u_{y_1,y_2} \bullet u_{y,y_1}\bullet u_{x,y}\\
        &+\sum_{k\ge 1}(-1)^k \sum_{y_k<\cdots <y_1<y< x}
        u_{y_{k-1},y_k} \bullet\dots\bullet u_{y_1,y_2} \bullet u_{y,y_1}\bullet u_{x,y}\\
        &=1_x,
    \end{split}
\end{gather*}
where the second and third terms cancel each other out, as do the fourth and fifth.
In a similar way, we obtain $(u\bullet u^-)(1_x)=1_x$.

By Remark~\ref{sastackinversesa}, the stack-inverse $u^-$ of $u$ is strongly admissible, which completes the proof.
\end{proof}

\subsection{Morita equivalence for stratified linear categories}\label{sectionMoritaC}

Let $\C$ be a stratified linear category over $\A$.
We say that $\C$ is \emph{perfect} if the twist $\tau$ is perfect.

\begin{proposition}\label{perfectstack-invertible}
    Let $\C$ be a stratified linear category over $\A$. If $\C$ is perfect, then the copairing $d$ of the pairing $b=b_\tau$ is stack-invertible.
\end{proposition}

\begin{proof}
    Since the copairing $d\in M(\C)$ of $b$ is strongly admissible by Example~\ref{dadm}, it follows from Proposition~\ref{stack-invertibility-stratified} that $d$ is stack-invertible. 
\end{proof}

By applying Theorem~\ref{Morita1} to $\C=\C_\tau$,
we obtain a Morita equivalence between $\A$ and $\C$.

\begin{theorem}\label{MoritaC}
    Let $\C$ be a stratified linear category over $\A$.
    Then the following are equivalent.
    \begin{enumerate}
        \item $\C$ is perfect.
        \item We have a $\C$-bimodule isomorphism $\C^+\ot_\A\C^-\cong \C$.
        \item The pair $(\C^+,\C^-)$
    gives a Morita equivalence between $\A$ and $\C$, i.e., we have
    \begin{gather*}
        \C^+\ot_\A\C^-\cong \C\quad\text{in $\Cbim$},\quad \text{and}\quad
        \C^-\ot_\C\C^+\cong \A\quad \text{in $\Abim$}.
    \end{gather*}
    \item We have a well-defined functor $\C^-\ot_{\C} -: \C\mmod \to \A\mmod$, which gives an equivalence of linear categories
$$\A\mmod \mathrel{\substack{\xrightarrow{\ \C^+\ot_\A -\ }\\[-0.4ex]\xleftarrow[\C^-\ot_{\C} -]{}}}\C\mmod.$$
    \end{enumerate}
\end{theorem}

\begin{proof}
If $\C$ is perfect, then by Proposition~\ref{perfectstack-invertible}, the copairing $d$ of $b_{\tau}$ is stack-invertible.
Since the linear category $\A\mmod$ is idempotent-complete, the statement follows from 
Theorem~\ref{Morita1}.
\end{proof}

\subsection{The center $Z(\C)$ of a stratified linear category $\C$}\label{sectioncenterC}

The \emph{center}  of a $\k$-linear category $\C$ is the commutative $\k$-algebra 
$Z(\C)=\operatorname{Nat}(\id_\C,\id_\C)$
consisting of the natural transformations from $\id_\C$ to itself.
Thus, an element $z\in Z(\C)$ is a family of endomorphisms
$z=\{z_x\in \C(x,x)\}_{x\in\Ob(\C)}$
such that for every morphism $f:x\to y$ in $\C$ we have
$f z_x =z_y f$.
The multiplication in $Z(\C)$ is the pointwise composition
$(zz')_x := z_x z'_x$ with unit $1=\{\id_x\}_x$.

For a base category $\A$, we see easily that $$Z(\A)=\Abim(\A,\A)=\prod_{x\in\Ob(\A)}Z(\A_x).$$
For a stratified linear category, we have the following.

\begin{lemma}
\label{lem-center}
    Let $\C$ be a stratified linear category over $\A$.
    Then we have a canonical isomorphism of $\k$-algebras
    \begin{gather*}
        Z(\C)\cong Z^{\Abim}(\C),
    \end{gather*}
    where $Z^{\Abim}(\C)$ is the center of the monoid $\C$ in the monoidal category $\Abim$ of $\A$-bimodules (see Section~\ref{sectioncenterofmonoid}).
\end{lemma}

\begin{proof}
    For $z=\{z_x\}_{x\in\Ob(\C)}\in Z(\C)$, define a morphism
    $z:\A\to \C$ in $\Abim$ by
    \[z(f):=f z_x=z_x f\]
    for $f\in\A_x$, $x\in\Ob(\A)=\Ob(\C)$.
    One can check that this correspondence is a bijection.
\end{proof}

Then Theorem~\ref{centerofAtau}, Proposition~\ref{perfectstack-invertible} and Lemma~\ref{lem-center} imply the following.

\begin{theorem}\label{centerofC}
    Let $\C$ be a perfect stratified linear category over $\A$.
    Then we have a $\k$-algebra isomorphism
    \begin{gather*}
        \zeta_{\tau}: Z(\A)\to Z(\C)
    \end{gather*}
    defined for $f=(f_y)_{y\in \Ob(\A)}\in Z(\A)$, $x\in \Ob(\C)$ by 
    \begin{gather*}
        (\zeta_\tau(f))_x=\sum_{y\le x}d^+_{x,y}\,f_y\,d^-(1_y)\,d^-_{x,y}\in \C(x,x),
    \end{gather*}
    where we write $d(1_x)=\sum_{y\le x}d^+_{x,y}\ot d^-_{x,y}$.
\end{theorem}

\subsection{Canonical central idempotents in stratified linear categories}\label{Sectioncancenidemp}

Let $\C$ be a stratified linear category over $\A$, and let $x\in\Ob(\C)$.
By a \emph{canonical central idempotent} of $\C(x,x)$, we mean a morphism $\ee_x\in\C(x,x)$ satisfying
\begin{enumerate}
    \item[(i)] $(\varepsilon^+\ot\varepsilon^-)(\ee_x)=1_x$, and
    \item[(ii)] $\ee_x \C^+(y,x)=0$ and $\C^-(x,y)\ee_x=0$ for all $y<x$.
\end{enumerate}

Let $J_x\subset\C(x,x)$ be the ideal defined by $J_x=\sum_{y<x}\C^+(y,x)\ot_{\A_y} \C^-(x,y)$.
Then condition (i) is equivalent to 
\begin{enumerate}
    \item[(i')] $\ee_x\equiv 1_x\pmod{J_x}$,
\end{enumerate}
and condition (ii) implies
\begin{enumerate}
    \item[(ii')] $\eex J_x=J_x\eex=0$.
\end{enumerate}
Since the existence of an element $\ee_x\in \C(x,x)$ satisfying conditions (i') and (ii') is equivalent to the splitting of the short exact sequence 
$$0\to J_x\to \C(x,x)\to \C(x,x)/J_x\to 0,$$
the element $\ee_x$ is called the \emph{splitting idempotent} (see \cite{Martin-Woodcock, King-Martin-Parker} for details). 
The following lemma is an observation in \cite[Section 1]{Martin-Woodcock}, but we give a proof of it for completeness.

\begin{lemma}\label{ex-central-idempotent}
    A canonical central idempotent $\ee_x$ is a central idempotent in the algebra $\C(x,x)$.
    Moreover, $\ee_x$ is unique if it exists.
\end{lemma}

\begin{proof}
By conditions (i') and (ii'), we have $\ee_x^2-\ee_x=\eex(\eex-1_x)\in\eex J_x=0$, which means that $\eex$ is an idempotent.
To see that $\eex$ is central, we first note that for any $f\in\C(x,x)$ we have $\eex f-f\eex\in J_x$ by (i'). 
Then we have
\begin{gather*}
        \eex f-f \eex=\eex^2 f-f \eex^2
        =\eex (\eex f-f\eex)+(\eex f-f\eex)\eex
        \in \eex J_x+ J_x\eex=0.
\end{gather*}
Thus, $\eex$ is a central idempotent in $\C(x,x)$.

To see that $\ee_x$ is unique, let $\ee_x'$ be another canonical central idempotent of $\C(x,x)$.
Since $\ee_x-\ee_x\ee'_x= \ee_x^2-\ee_x\ee'_x=\ee_x(\ee_x-\ee_x')\in\ee_x J_x=0$, we have $\ee_x=\ee_x\ee_x'$.
Similarly, we have $\ee_x'=\ee_x\ee_x'$.
Hence $\ee_x=\ee_x'$.
\end{proof}

For a canonical central idempotent $\eex$ of $\C(x,x)$, we have a two-sided ideal $\eex\C(x,x)\eex$ of the algebra $\C(x,x)$, which forms an algebra with unit $\eex$.
By condition (ii), we have 
\begin{gather}\label{excxxex}
    \eex\C(x,x)\eex=\C(x,x)\eex=\C^0(x,x)\eex\subset \C(x,x).
\end{gather}

\begin{proposition}
For a stratified linear category $\C$ over $\A$ and an object $x$ in $\C$, let $\eex\in \C(x,x)$ be a canonical central idempotent.
Then we have an algebra isomorphism
\begin{gather*}
    \varphi_x:\A_x\xrightarrow{\ \cong\ }
    \eex\;\C(x,x)\;\eex,\quad
    \varphi_x(f)=\eex i(f) \eex.
\end{gather*}   
\end{proposition}

\begin{proof}
Since $\eex$ is a central idempotent of the algebra $\C(x,x)$, $\varphi_x$ is an algebra homomorphism. It follows from \eqref{excxxex} that $\varphi_x$ is surjective since $i:\A_x\to \C^0(x,x)$ is an isomorphism. Moreover, by condition (i), the composition of $\varphi_x$ and the restriction $\eex\C(x,x)\eex\to\C^0(x,x)$ of the projection $\C(x,x)\twoheadrightarrow\C^0(x,x)$
is equal to the isomorphism $i:\A_x\xrightarrow{\cong}\C^0(x,x)$.
Therefore, $\varphi_x$ is injective, which completes the proof.
\end{proof}

By a \emph{family of canonical central idempotents} in $\C$, we mean a family $\ee=\{\ee_x\}_{x\in\Ob(\C)}$ of canonical central idempotents in $\C(x,x)$.

\begin{lemma}
\label{lem-canonical-idempotent}
    A family $\ee$ of canonical central idempotents in $\C$ is a canonical idempotent $\ee$ for the twisted tensor product $\C=\C_\tau$ in the monoidal category $\Abim$, i.e.,
    an $\A$-bimodule map $\ee:\A\to\C$ satisfying 
    \begin{gather}
    \label{eq3eeC}
    (\mu^+\ot\idC-)(\idC+\ot \tau)(\ee\ot \idC+)
    =\ee\varepsilon^+,\\
    \label{eq30eeC}
   (\idC+\ot \mu^-)(\tau\ot\idC-)(\idC-\ot \ee)=\ee\varepsilon^-,\\
    \label{eq31eeC}
    (\varepsilon^+\ot\idC-)\ee
    =\eta^-,\\
    \label{eq32eeC}
    (\idC+\ot\varepsilon^-)\ee
    =\eta^+.
\end{gather}
\end{lemma}

\begin{proof}
Let $\ee$ be a family of canonical central idempotents.
Then $\ee$ forms an $\A$-bimodule map $\ee:\A\to \C$ which maps $\ee(1_x)=\ee_x$ for $x\in \Ob(\A)$.
We have for $f\in\C^+(y,x)$ with $y\le x$,
\begin{gather*}
    (\mu^+\ot\idC-)(\idC+\ot \tau)(\ee\ot \idC+)(f)=\eex f.
\end{gather*}
If $y=x$, then $f\in\C^0(x,x)$ and $\eex f=\eex \varepsilon^+(f)$.
If $y<x$, then $\eex f=0 = \eex \varepsilon^+(f)$.
Thus, we have \eqref{eq3eeC}.
Similarly, \eqref{eq30eeC} follows.
    \eqref{eq31eeC} follows since
    $\eex\in 1_x+ J_x$ and $(\varepsilon^+\ot\idC-)(J_x)=0$.  Similarly, \eqref{eq32eeC} follows.

Conversely, suppose that an $\A$-bimodule map $\ee:\A\to\C$ satisfies \eqref{eq3eeC}--\eqref{eq32eeC}.  Then \eqref{eq31eeC}, \eqref{eq32eeC} imply condition (i) of a canonical central idempotent, and \eqref{eq3eeC}, \eqref{eq30eeC} imply condition (ii) of a canonical central idempotent, which completes the proof.
\end{proof}

\begin{theorem}\label{perfectcci}
Let $\C$ be a stratified linear category over $\A$. Then $\C$ is perfect if and only if $\C$ admits a family of canonical central idempotents.
If this is the case, then this family of canonical central idempotents is unique. 
\end{theorem}

\begin{proof}
By Propositions~\ref{stack-invertibility-stratified} and~\ref{stack-invertible-perfect-idempotent}, $\tau$ is perfect if and only if there is a canonical idempotent for $\C$, which is necessarily unique. 
Therefore, the statement follows from Lemma~\ref{lem-canonical-idempotent}.
\end{proof}

\section{Quasi-cellular categories}\label{sectionqc}
We define quasi-cellular categories as stratified linear categories with a pair of extra pairing and copairing, and then apply the results of Section~\ref{sec-twisted-tensor-product} through the identification of a stratified linear category over $\A$ with the twisted tensor product in $\Abim$
in Section~\ref{sectionstrlincat}.

\subsection{Quasi-cellular categories}

By a \emph{quasi-cellular category} over $\A$, we mean a stratified linear category $\C$ over $\A$ equipped with a pair $(b',d')$ of a perfect pairing $b': \C^-\ot_{\A}\C^+\to \A$ and its copairing $d':\A\to\C^+\ot_{\A}\C^-\cong \C$ in $\Abim$
satisfying
\begin{gather}\label{qc-eta-epsilon}
    b'(\eta^-\ot{\C^+})=\varepsilon^+,\quad
    b'({\C^-}\ot\eta^+)=\varepsilon^-.
\end{gather}
As is easily seen, the copairing $d'$ is strongly admissible.

By the observation of Section~\ref{sectionstackinvinstrlincat}, giving the pair $(b',d')$ as above is equivalent to giving for each $x\in \Ob(\A)$ an $\A_x$-bimodule map
\begin{gather*}
    b':\bigoplus_{y\ge x}\C^-(y,x)\ot_{\A_y}\C^+(x,y)\to \A_x
\end{gather*}
and an element 
\begin{gather*}
    d'(1_x)=\sum_{y\le x}(d')^+_{x,y}\otimes (d')^-_{x,y}
    \in Z_{\C(x,x)}(\A_x)\subset \bigoplus_{y\le x}\C^+(y,x)\ot_{\A_y}\C^-(x,y)
\end{gather*}
which satisfy for $x,z\in\Ob(\A)$,
\begin{gather*}
   \sum_{y\le x} (d')_{x,y}^+ b'( (d')_{x,y}^-\ot f)=f\quad \text{for any $f\in\C^+(z,x)$},\\
   \sum_{y\le x} b'(g\ot (d')_{x,y}^+)(d')_{x,y}^-= g \quad \text{for any $g\in\C^-(x,z)$},\\
   b'(1_x\ot f)=\begin{cases}
       f&(x=z)\\0&(x\neq z)
   \end{cases}\quad\text{for any $f\in\C^+(z,x)$},\\
   b'(g\ot 1_x)=\begin{cases}
       g&(x=z)\\0&(x\neq z)
   \end{cases}\quad\text{for any $g\in\C^-(x,z)$}.
\end{gather*}

Proposition~\ref{stack-invertibility-stratified} and the fact that $d'$ is strongly admissible imply the following result.

\begin{proposition}
\label{stack-invertible-dtau}
Let $\C=(\C,b',d')$ be a quasi-cellular category over $\A$. Then the $\A$-bimodule map $d'$ is stack-invertible.
\end{proposition}

Many interesting examples of quasi-cellular categories admit natural anti-involutions. Therefore, it is natural to make the following definition.

\begin{definition}
A quasi-cellular category $\C=(\C,b',d')$ over $\A$ is \emph{involutive} if it is equipped with an anti-involution $\iota$ on $\C$ (i.e., a linear functor $\iota:\C^\op\to\C$ such that $\iota\circ\iota^\op=\id_\C$)
satisfying
\begin{enumerate}
\item $\iota:\Ob(\C)\to\Ob(\C)$ is order-preserving,
\item $\iota(\C^\pm(x,y))=\C^\mp(\iota(y),\iota(x))$ for any $x,y\in\Ob(\C)$,
\item $\iota(b'(g\ot f))=b'(\iota(f)\ot\iota(g))$ for $f\in\C^+(x,y)$, $g\in\C^-(y,x)$, $x,y\in\Ob(\C)$.
\end{enumerate}
\end{definition}

Involutive quasi-cellular categories are closely related to the strictly object-adapted cellular categories of Elias and Lauda \cite{Elias-Lauda}.  The latter are a rigidified form of Westbury's cellular categories \cite{Westbury}: the cellular basis is required to factor through objects representing the cells.  Westbury's cellular categories, in turn, are a natural many-object generalization of the cellular algebras of Graham and Lehrer \cite{Graham-Lehrer}.

\subsection{The functor $L:\C\to\C_{b'}$ for a quasi-cellular category}
Let $\C=(\C,b',d')$ be a quasi-cellular category over $\A$.
Then we have a monoid $\C_{b'}=\C^+\ot_{b'}\C^-=(\C^+\ot_{\A}\C^-,\mu_{b'},d')$ in $\Abim$.
The monoid $\C_{b'}$ corresponds to 
a linear category $\C_{b'}$ equipped with the linear functor $d':\A\to\C_{b'}$ via Lemma~\ref{monoidsandlinearcategories}.
Then Proposition~\ref{morphismL}(1) implies the following.

\begin{proposition}
We have a linear functor
\begin{gather*}
        L=L_{\tau,b'}: \C\to \C_{b'}.
\end{gather*}
\end{proposition}

By \eqref{eq-L}, the functor $L$ is described as follows:
for $f^+\in\C^+(z,y)$, $f^-\in\C^-(x,z)$, we have
\begin{gather}\label{eq-L-C}
    L(f^+\ot f^-)=\sum_{w\le y} (d')_{y,w}^+ \ot((\varepsilon^+\ot{\C^-})\tau((d')_{y,w}^- \ot f^+))f^-
    \\
    \notag
    \in \bigoplus_{w\in \Ob(\A)}\C^+(w,y)\ot_{\A_w} \C^-(x,w).
\end{gather}
Thus,
for $f\in\C(x,y)$, we have
\begin{gather}\label{eq-L-C-2}
    L(f)=\sum_{w\le y} (d')_{y,w}^+ \ot (\varepsilon^+\ot {\C^-})((d')_{y,w}^- f)\in \C_{b'}(x,y).
\end{gather}

A quasi-cellular category over $\A$ is \emph{perfect} if it is perfect as a stratified linear category.

\begin{proposition}\label{L-iso}
Let $\C$ be a perfect quasi-cellular category over $\A$.
Then the functor $L:\C\to\C_{b'}$ is an isomorphism of linear categories.
\end{proposition}

\begin{proof}
By Propositions~\ref{propF-rho} and~\ref{morphismL}, if a twist $\tau$ is perfect with duality $(b,d)$ such that $d$ is stack-invertible, then the monoid morphism $L_{\tau,b'}$ is an isomorphism.
Therefore, the statement follows from Proposition~\ref{perfectstack-invertible} via Lemma~\ref{monoidsandlinearcategories}.
\end{proof}

\subsection{Precontractions for quasi-cellular categories}\label{subsectionqcpsitilde}
We here translate the definition of a precontraction in Section~\ref{section1precontraction} into the setting of quasi-cellular categories.

By a quasi-cellular category over $\A$ \emph{with an action}, we mean a quasi-cellular category $\C=(\C,b',d')$ such that the pairing $b':\C^-\ot_{\A}\C^+\to\A$ is induced from an action $\partial:\C^-\ot_{\A}\C^+\to\C^+$ of $\C^-$ on $\C^+$ by $b'=b_\partial=\varepsilon^+\partial$.

\begin{proposition}\label{precontractionforQC}
A precontraction for a quasi-cellular category $\C=(\C,\partial,b'=b_{\partial},d')$ over $\A$ with an action is an $\A$-bimodule morphism $\psi :\C^+\to \C$
satisfying for any $x,y,z,w\in \Ob(\A)$
\begin{enumerate}
    \item[(P1)] $f^- \psi (f^+)=\psi \partial(f^-\ot f^+)$
     for $f^+\in \C^+(x,y)$, $f^-\in \C^-(y,z)$,
    \item[(P2)] $\psi (f^+)g^+=0$ for $f^+\in \C^+(x,y)$, $g^+\in \C^+(w,x)$ with $w<x$,
    \item[(P4)] $\sum_{y\le x}\psi ((d')_{x,y}^+)(d')_{x,y}^-=1_x$.
\end{enumerate}    
\end{proposition}

\begin{proof}
    Conditions (P1), (P2) and (P4) above are translations of the corresponding conditions in Section~\ref{section1precontraction}.
In the present setting, (P3) reads:
\begin{itemize}
    \item[(P3)] 
    $(\varepsilon^+\ot  \C^-)\psi (f)=\varepsilon^+(f)$ for $f\in \C^+(x,y)$.
\end{itemize}  
If $x<y$, then both sides are $0$. If $x=y$, then (P3) is equivalent to 
$$(\varepsilon^+\ot  \C^-)\psi(1_x)=1_x \text{ for } x\in \Ob(\C),$$
which follows from (P4) since $d'$ is strongly admissible and thus we have $(d')^+_{x,x}=(d')^-_{x,x}=1_x$.
\end{proof}

The following result immediately follows from Theorem~\ref{precontraction} and Proposition~\ref{perfectstack-invertible} via Lemma~\ref{monoidsandlinearcategories}.

\begin{theorem}
\label{precontraction2}
Let $\C$ be a quasi-cellular category over $\A$ with an action.
Then the following conditions are equivalent.
\begin{enumerate}
\item $\C$ is perfect as a quasi-cellular category.
\item There exists a (unique) precontraction $\psi :\C^+\to\C$.
\item The linear functor $L:\C\to\C_{b'}$ is an isomorphism.
\end{enumerate}
\end{theorem}

If, furthermore, the pairing $b'$ for a quasi-cellular category $\C=(\C,b',d')$ over $\A$ is induced from a twist $\tau': \C^-\ot_{\A}\C^+\to \C^+\ot_{\A}\C^-$ by $b'=b_{\tau'}$,
then one can strengthen Proposition~\ref{L-iso} as follows.

\begin{theorem}\label{isomofqC}
 Let $\C=(\C,\; b'=b_{\tau'},\;d')$ be a quasi-cellular category equipped with a twist $\tau': \C^-\ot_{\A} \C^+\to \C^+\ot_{\A} \C^-$.
 Then we have linear functors
 \begin{gather*}
      \C\xrightarrow[]{L_{\tau,b'}} \C_{b'}\xleftarrow[\cong]{F_{d'}}\C_{\tau'}.
 \end{gather*}
 Moreover, $\C$ is perfect if and only if $L_{\tau,b'}$ is an isomorphism; thus
 we have isomorphisms of linear categories
    \begin{gather*}
        \C\xrightarrow[\cong]{F_{d}} \C_{b} \xrightarrow[\cong]{K_{b,b'}}\C_{b'}\xrightarrow[\cong]{F_{d'}^{-1}}\C_{\tau'}.
    \end{gather*}
\end{theorem}

\begin{proof}
    This follows from Propositions
   ~\ref{propF-rho},~\ref{morphismL}, ~\ref{perfectstack-invertible} and~\ref{stack-invertible-dtau} 
    and Theorem~\ref{precontraction}
    via Lemma~\ref{monoidsandlinearcategories}.
\end{proof}

\section{The Brauer category}
\label{sectionBrauercategory}

We recall the Brauer category $B^{\k,\delta}$ over a commutative ring $\k$ with parameter $\delta\in\k$.

\subsection{Brauer diagrams}

For integers $p,q\ge0$ with $p+q$ even, a \emph{$(p,q)$-Brauer diagram} is a partition of the set $[p]^+\sqcup [q]^-$ into $(p+q)/2$ unordered pairs, where $[p]^+=\{1^+,\dots,p^+\}$ is a copy of $[p]=\{1,\dots,p\}$ and $[q]^-=\{1^-,\dots,q^-\}$ is a copy of $[q]=\{1,\dots,q\}$.
For example, 
$$\{\{1^+,3^+\},\{2^+,1^-\},\{4^+,5^-\},\{2^-,4^-\},\{3^-,6^-\}\}$$ is a $(4,6)$-Brauer diagram.
A $(p,q)$-Brauer diagram can be depicted as $(p+q)/2$ arcs joining $p$ points $1^+,\dots,p^+$ on the top and $q$ points $1^-,\dots,q^-$ on the bottom, where we consider the upside-down version of the Brauer diagrams drawn by Lehrer and Zhang (see \cite[Fig.1]{Lehrer-Zhang}).

Let $\B(p,q)$ denote the set of $(p,q)$-Brauer diagrams.

A Brauer diagram is \emph{upward} (resp. \emph{downward}) if 
the partition contains no pair of two positive (resp. negative) elements.
A Brauer diagram is \emph{level} if it is upward and downward.
Let $\B^+(p,q)$ (resp. $\B^-(p,q)$, $\B^0(p,q)$) denote the subset of $\B(p,q)$ that consists of upward (resp. downward, level) Brauer diagrams.
Then we have $\B^0(p,q)\cong\gpS_p$ if $p=q$ and $\B^0(p,q)=\emptyset$ otherwise.

We define a partial order $\preceq$ on non-negative integers by
\[p\preceq q\quad\text{iff}\quad p\le q\text{ with $p+q$ even}.\]
By $p\prec q$, we mean $p\preceq q$ and $p\neq q$.
Note that $\B^+(p,q)\neq\emptyset$ if and only if $p\preceq q$. 

We have a bijection
$$\overline{(\ )}:\B(p,q)\to \B(q,p), \quad f\mapsto \overline{f},$$
where $\overline{f}$ is obtained from $f$ by applying the bijection $[p]^+\sqcup[q]^-\simeqto [q]^+\sqcup [p]^-$ that sends $i^\pm$ to $i^\mp$.
Note that $\overline{(\ )}$ maps upward Brauer diagrams to downward Brauer diagrams, and vice versa, and level Brauer diagrams to level Brauer diagrams.

For two Brauer diagrams $f\in \B(q,r)$ and $g\in \B(p,q)$, 
we consider the following operation needed for defining composition: place $g$ above $f$ and identify the lower $q$ vertices of $g$ with the
upper $q$ vertices of $f$.
Let $l(f,g)$ be the number of closed
loops obtained in the middle, and let $f\hat{\circ}g\in \B(p,r)$ be
the Brauer diagram obtained after deleting these loops. 

A Brauer diagram $b\in\B(p,q)$ is said to be \emph{monotone} if we have $(i-k)(j-l)>0$ for any pairs in the partition of the form $\{i^+,j^-\}$, $\{k^+,l^-\}$.
Let $\B^m(p,q)$ denote the subset of $\B(p,q)$ consisting of the monotone Brauer diagrams.
Let $\B^{\pm,m}(p,q)=\B^\pm(p,q)\cap\B^m(p,q)$.
Note that we have
$\B^{0}(p,p)\cap \B^m(p,p)=\{1_p\}$, where $1_p:=
\{\{i^+,i^-\}\mid i\in[p]\}$.
We have a bijection
\begin{gather}\label{decompositionBrauermorphism}
    \coprod_r \B^{+,m}(r,q)\times\B^0(r,r)\times\B^{-,m}(p,r)\to \B(p,q),
\end{gather}
which is given by $(f^+,f^0,f^-)\mapsto f^+\hat{\circ} f^0\hat{\circ} f^-$.

Let $p\ge2$.
For $i,j\in[p]$, $i\neq j$, set
\begin{gather*}
c_{i,j}=\{\{i^+,j^+\}, \{1^-,f_{i,j}(1)^+\}, \dots, \{(p-2)^-,f_{i,j}(p-2)^+\}\}\in \B^{-,m}(p,p-2),
\end{gather*}
where $f_{i,j}:[p-2]\simeqto [p]\setminus\{i,j\}$ is the order-preserving bijection.
Then we have
$\overline{c_{i,j}}\in \B^{+,m}(p-2,p)$.

\subsection{The Brauer category $B$}

We here define the Brauer category $B=B^{\k,\delta}$ over $\k$ with parameter $\delta\in\k$.
The objects of the category $B$ are non-negative integers and the hom space $B(p,q)$ is the free $\k$-module with basis $\B(p,q)$.

The composition of two Brauer diagrams $f\in \B(q,r)$ and $g\in \B(p,q)$ is given by 
\[f\circ g=\delta^{l(f,g)}f\hat{\circ}g\in B(p,r).\]
At the graphical level, $f\circ g$ is obtained by pasting the diagrams of $f$ and $g$ at the middle $q$ points, where each of the resulting loops contributes a factor $\delta$. 
The identity morphism $\id_p:p\to p$ is $1_p\in \B^0(p,p)$.
The endomorphism algebra $B(p)=B(p,p)$ of the object $p$ in the Brauer category $B$ is the Brauer algebra.

The Brauer category $B$ has a $\k$-linear PROP structure, i.e., a $\k$-linear symmetric strict monoidal category with objects non-negative integers.
The tensor functor $\otimes: B\times B\to B$ is addition on objects and horizontal juxtaposition on morphisms.
The symmetry $P_{p,q}:p\otimes q\to q\otimes p$ is given by
\begin{gather}\label{symmetryBrauer}
    P_{p,q}=\{\{1^+,(q+1)^-\},\dots,\{p^+,(q+p)^-\},\{(p+1)^+,1^-\},\dots,\{(p+q)^+,q^-\}\}.
\end{gather}

We have an isomorphism $\overline{(\ )}:B^{\op}\to B$ of $\k$-linear categories consisting of linear isomorphisms
$\overline{(\ )}:B(p,q)\simeqto B(q,p)$, which is the identity on objects and maps a Brauer diagram $f\in \B(p,q)$ to $\overline{f}\in \B(q,p)$.

The following presentation of the Brauer category $B$ is well known (see e.g. \cite{Lehrer-Zhang}).
As a $\k$-linear symmetric monoidal category, $B$ is generated by the object $1$ and the morphisms
\begin{gather*}
    c=\{\{1^+,2^+\}\}:2\to0,\quad
    \overline{c}=\{\{1^-,2^-\}\}:0\to2
\end{gather*}
with relations
\begin{gather}\label{relations-for-B}
    (\id_1\otimes c)(\overline{c}\otimes \id_1)=\id_1=(c\otimes \id_1)(\id_1\otimes \overline{c}),\quad
    cP_{1,1}=c,\quad P_{1,1}\overline{c}=\overline{c},\quad
    c\overline{c}=\delta\id_0.
\end{gather}
Note that $\id_1=(c\otimes \id_1)(\id_1\otimes \overline{c})$ and $P_{1,1}\overline{c}=\overline{c}$ follow from $(\id_1\otimes c)(\overline{c}\otimes \id_1)=\id_1$ and $cP_{1,1}=c$.
If $\k=\k'[\delta]$ for some commutative ring $\k'$, then the Brauer category $B$ is a free symmetric monoidal $\k'$-linear category with duals on a single self-dual object.

\section{The Brauer category as a quasi-cellular category}
\label{sectionBrauerasqc}

We equip the Brauer category $B$ with a quasi-cellular structure.

\subsection{The stratified linear structure on $B$}

Let $\S$ denote the base category whose objects are non-negative integers with the partial order $\preceq$, and whose hom-spaces are given by  $\S(p,q)=0$ for $p\neq q$ and $\S(p,p)=\k[\gpS_p]$ for $p\ge 0$.
Let $\S_p=\S(p,p)$ for $p\ge 0$.

We will see that the Brauer category $B$ is a stratified linear category over $\S$.
We have a linear functor $i:\S\to B$ which is the identity on objects and sends $\sigma\in \gpS_p$ to the Brauer diagram $\{\{1^+,\sigma(1)^-\}, \dots, \{p^+,\sigma(p)^-\}\}\in \B^0(p,p)\subset B(p,p)$.
Let $B^{\pm}$ (resp. $B^0$) denote the wide subcategories of $B$ whose hom-space $B^{\pm}(p,q)$ (resp. $B^0(p,q)$) is the free $\k$-module with basis $\B^{\pm}(p,q)$ (resp. $\B^0(p,q)$).
It follows from the definitions of $\B^{\pm}(p,q)$ and $\B^0(p,q)$ that
\begin{itemize}
    \item we have $B^+(p,q)=0$ unless $p\preceq q$ and $B^-(p,q)=0$ unless $q\preceq p$, 
    \item $B^+(p,p)=B^-(p,p)=B^0(p,p)$ for each $p\ge 0$,
    \item the functor $i$ maps $\S_p$ isomorphically onto $B^0(p,p)$.
\end{itemize}
By \eqref{decompositionBrauermorphism}, one can check that the composition map in $B$
\begin{gather*}
    \circ: \bigoplus_{r\ge 0}B^+(r,q)\otimes_{\S_r}B^-(p,r)\to B(p,q), \quad f\otimes g\mapsto f\circ g 
\end{gather*}
is an $(\S_q,\S_p)$-bimodule isomorphism.
Therefore, the Brauer category $B$ is a stratified linear category over $\S$.

Similar stratifications of the Brauer categories have been observed e.g. in \cite{Sam-Snowden-BrauerI}.

\subsection{The quasi-cellular structure on $B$}

By Theorem~\ref{stratifiedlinearcategoryandtwistedtensorproduct}, the stratified linear category structure over $\S$ of the Brauer category $B$ induces a strongly upward monoid $B^+$, a strongly downward monoid $B^-$, and a twist $\tau: B^-\otimes_{\S} B^+\to B^+\otimes_{\S} B^-$ in the category $\Sbim$ such that $B_{\tau}=B^+\otimes_{\tau}B^-\cong B$.

We have another twist for augmented monoids $B^{\pm}$ in $\Sbim$ as follows.
Define an $\S$-bimodule map
\begin{gather*}
    \tau': B^-\ot_{\S} B^+\to B^+\ot_{\S} B^-
\end{gather*}
for $f\in \B^-(q,r)$, $g\in \B^+(p,q)$ by
$\tau'(f\otimes g)=0$ if there is a triple $(i,j,k)$ of distinct elements of $[q]$ such that
\begin{itemize}
    \item[$(\ast)$] $\{i^+,j^+\}\in f$ and $\{j^-,k^-\}\in g$
\end{itemize}
and otherwise, by $\tau'(f\otimes g)=f\hat{\circ} g$.

\begin{lemma}\label{Brauerrho}
The $\S$-bimodule map $\tau'$ is a twist for $B^\pm$.
\end{lemma}

\begin{proof}
We will prove the identity 
\begin{gather}\label{twist1}
    \tau' (\mu^-\otimes  B^+)= ( B^+\otimes \mu^-)(\tau'\otimes  B^-)( B^-\otimes \tau').
\end{gather}
For $f\in \B^-(r,s), g\in \B^-(q,r), h\in \B^+(p,q)$, we have
\begin{gather*}
    \tau' (\mu^-\otimes  B^+)(f\ot g\ot h)=\tau'(fg\ot h)
\end{gather*}
and
\begin{gather*}
\begin{split}
    ( B^+\otimes \mu^-)(\tau'\otimes  B^-)( B^-\otimes \tau')(f\ot g\ot h)
    &=( B^+\otimes \mu^-)(\tau'\otimes  B^-)(f\ot \tau'(g\ot h))\\
    &=( B^+\otimes \mu^-)(\sum \tau'(f\ot \hat{h}^+)\ot\hat{g}^-)\\
    &=\sum \tilde{h}^+\ot \hat{f}^- \hat{g}^-,
\end{split}
\end{gather*}
where $\tau'(g\ot h)=\sum \hat{h}^+\ot \hat{g}^-$ and $\tau'(f\ot \hat{h}^+)=\sum \tilde{h}^+\ot \hat{f}^-$.
It is easy to check that we have $(\ast)$ for $fg$ and $h$ if and only if we have $(\ast)$ for $g$ and $h$ or $(\ast)$ for $f$ and $\hat{h}^+$.
In this case, both sides of the identity \eqref{twist1} vanish.
Therefore, it suffices to consider the case where $\tau'=\hat{\circ}$. Since $\hat{\circ}$ is induced by the twist $\tau$ by letting $\delta=1\in \k$, the identity \eqref{twist1} follows from the fact that $\tau$ is a twist.

The identity $\tau'(\eta^-\otimes  B^+)= B^+\otimes \eta^-$ is trivial and the other identities for $\mu^+,\eta^+$ are similarly obtained.
\end{proof}

Let $\varepsilon^\pm:B^\pm\to \S$ denote the augmentation of the augmented monoids $B^\pm$ in $\Sbim$.

The associated action 
\begin{gather*}
    \partial:B^-\ot_{\S}B^+\to B^+
\end{gather*}
defined by $\partial=(B^+\ot \varepsilon^-)\tau'$
satisfies
\begin{gather*}
    \partial(f\ot g)=\begin{cases}
        h& \text{if }g=\overline{f}h \text{ for some }h\in B^+(p,r),\\
        0 & \text{otherwise}
    \end{cases}
\end{gather*}
for $f\in B^-(q,r), g\in B^+(p,q)$.

Let $b':=b_{\tau'}=b_{\partial}: B^-\ot_{\S} B^+\to \S$ be the pairing associated to the twist $\tau'$.

\begin{lemma}\label{lemmaBbrho}
For $f\in \B^{-,m}(q,r)$, $g\in \B^{+,m}(p,q)$, the pairing $b'$ satisfies
\begin{gather}\label{brho}
    b'(f\ot g)=\delta_{\overline f,g}1_p.
\end{gather}
Moreover, $b'$ satisfies
    \begin{gather*}
        b'(\eta^-\ot{B^+})=\varepsilon^+,\quad
    b'({B^-}\ot\eta^+)=\varepsilon^-.
    \end{gather*}
\end{lemma}

\begin{proof}
For $f\in \B^{-,m}(q,r)$, $g\in \B^{+,m}(p,q)$, we have
\begin{gather*}
\begin{split}
    b'(f\ot g)&= \varepsilon^+\partial(f\ot g)
    =\begin{cases}
        \varepsilon^+(h) & \text{if }g=\overline{f}h \text{ for some }h\in B^+(p,r)\\
        0 & \text{otherwise}
    \end{cases}\\
    &=\begin{cases}
        1_p & \text{if }g=\overline{f}\\
        0 & \text{otherwise}.
    \end{cases}\\
\end{split}
\end{gather*}

The last statement easily follows from equation \eqref{brho}.
\end{proof}

\begin{lemma}
    The isomorphism $\overline{(\ )}$ of $\k$-linear categories is an identity-on-objects anti-involution on $B$ satisfying 
    \begin{enumerate}
\item $\overline{(\ )}(B^\pm(p,q))=B^\mp(q,p)$ for any $p,q\in\Ob(B)$,
\item $\overline{b'(g\ot f)}=b'(\overline{f}\ot\overline{g})$ for $f\in B^+(p,q)$, $g\in B^-(q,p)$, $p,q\in\Ob(B)$.
\end{enumerate}
\end{lemma}

\begin{proof}
    The statement easily follows from the definition of the linear functor $\overline{(\ )}$ and Lemma~\ref{lemmaBbrho}.
\end{proof}

\begin{lemma}\label{drhoBrauer}
    The pairing $b'$ is perfect with the copairing $d':\S\to B^+\ot_{\S}B^-$ defined by 
    \begin{gather*}
        d'(1_p)=\sum_{r\preceq p}
        \sum_{f\in \B^{+,m}(r,p)}f\otimes \overline{f}\in \bigoplus_{r\ge 0} B^+(r,p)\ot_{\S_r} B^-(p,r).
    \end{gather*}
\end{lemma}

\begin{proof}
    It suffices to check that 
    \begin{gather*}
        ({B^+}\otimes b')(d'\otimes {B^+})(g)=g, \quad g\in \B^{+,m}(q,p)\\
        (b'\otimes {B^-})({B^-}\otimes d')(g)=g, \quad g\in \B^{-,m}(p,q).
    \end{gather*}
    We have 
    \begin{gather*}
    \begin{split}
        ({B^+}\otimes b')(d'\otimes {B^+})(g)
        &=\sum_{r\preceq p}\sum_{f\in \B^{+,m}(r,p)}f\otimes b'(\overline{f}\otimes g)\\
        &=\sum_{r\preceq p}\sum_{f\in \B^{+,m}(r,p)}f\otimes \delta_{f,g}1_{q}=g
    \end{split}
    \end{gather*}
    by Lemma~\ref{lemmaBbrho}.
    We can check the other identity similarly.
\end{proof}

We have obtained the following.

\begin{proposition}\label{Brauerqc}
Let $b'$ be the pairing associated to the twist $\tau'$, and let $d'$ be its copairing. Then $(B,b',d')$ is a quasi-cellular category over $\S$.
Moreover, it is involutive with anti-involution $\overline{(\ )}$.
\end{proposition}

By Propositions~\ref{stack-invertible-dtau} and~\ref{Brauerqc}, the copairing $d'$ of $b'$ is stack-invertible.
Here we explicitly describe the stack-inverse $(d')^-$ of $d'$.

\begin{proposition}
    We have 
    \begin{gather*}
         (d')^-(1_p)=\sum_{r\preceq p}(-1)^{\frac{p-r}{2}}\sum_{f\in \B^{+,m}(r,p)}f\otimes \overline{f}\in \bigoplus_{r\preceq p} B^+(r,p)\ot_{\S_r} B^-(p,r).
    \end{gather*}
\end{proposition}

\begin{proof}
We use the construction of the stack-inverse in the proof of Proposition~\ref{stack-invertibility-stratified}.
By Lemma~\ref{drhoBrauer}, we have
$$d'(1_p)=\sum_{r\le p}d'_{p,r},$$
where $d'_{p,r}=\sum_{f\in \B^{+,m}(r,p)}f\otimes \overline{f}$ if $r\preceq p$ and otherwise $0$.
Let $\widetilde{d'}=d'-\eta$.
Then we have $$\widetilde{d'}(1_p)=\sum_{r\prec p}d'_{p,r},$$
and thus 
$$(\widetilde{d'})^{\bullet k}(1_p)=\sum_{r_k\prec r_{k-1}\prec \dots\prec r_1\prec p}
d'_{r_{k-1},r_k}\bullet \dots\bullet d'_{r_1,r_2}\bullet d'_{p,r_1}.$$
Here, we have $(\widetilde{d'})^{\bullet k}(1_p)=0$ for $k> \lfloor p/2\rfloor$.
The stack-inverse is given by
$$(d')^-(1_p)=\sum_{k=0}^{\lfloor p/2 \rfloor}(-1)^k (\widetilde{d'})^{\bullet k}(1_p),$$
which is a linear combination of $f\ot \overline{f}$ for $f\in \B^{+,m}(r,p)$ with $r\prec p$.
The coefficient of $f\ot \overline{f}$ in $(d')^-(1_p)$ is given by
$$
\sum_{l=1}^{(p-r)/2} (-1)^l s((p-r)/2,l),
$$
where $s(k,l)$ denotes the number of surjections from $[k]$ to $[l]$.

It suffices to prove that for $k\ge 1$, we have
\begin{gather}\label{surjectionsignedsum}
   \sum_{l=1}^{k} (-1)^l s(k,l)=(-1)^k.
\end{gather}
Since $s(1,1)=1$, we have \eqref{surjectionsignedsum} for $k=1$.
Since we have
\begin{gather*}
    s(k,l)=l s(k-1,l)+l s(k-1,l-1),
\end{gather*}
it follows that
\begin{gather*}
\begin{split}
    \sum_{l=1}^{k} (-1)^l s(k,l)
    &=\sum_{l=1}^{k}(-1)^l(ls(k-1,l)+l s(k-1,l-1))\\
    &=\sum_{l=1}^{k-1}(-1)^l ls(k-1,l)+\sum_{l=1}^{k-1}(-1)^{l+1}  (l+1) s(k-1,l)\\
    &=(-1)\sum_{l=1}^{k-1}(-1)^{l}  s(k-1,l).
\end{split}
\end{gather*}
Therefore, by induction on $k$, we obtain \eqref{surjectionsignedsum}.
This completes the proof.
\end{proof}

\subsection{The functor $L$ and the variants $\qB$ and $\mB$ of the Brauer category}\label{subsectionqBandmB}

As we have observed, the Brauer category $B$ corresponds to the twisted tensor product $B_{\tau}$.

Let $\qB$ denote the $\k$-linear category over $\S$ corresponding to the monoid $B_{\tau'}$, which we call the \emph{quasi-Brauer category}.
Let $\mB$ denote the $\k$-linear category over $\S$ corresponding to the monoid $B_{b'}$, which we call the \emph{matrix Brauer category}.
Then by Theorem~\ref{isomofqC} and Proposition~\ref{Brauerqc}, we obtain the following.

\begin{corollary}\label{qBmBB}
    We have $\k$-linear functors
    \begin{gather*}
    B\xrightarrow{L=L_{\tau,b'}}\mB\underset{\cong}{\xleftarrow{F_{d'}}}\qB.
    \end{gather*}
\end{corollary}

\begin{example}
    Let us consider $L:B(1,3)\to \mB(1,3)$.
    Both $B(1,3)$ and $\mB(1,3)$ are the free $\k$-module with basis 
    $\{
    \overline{c_{1,2}},
    \overline{c_{1,3}},
    \overline{c_{2,3}}\}
    $.
    Then $L$ is expressed by the matrix
    \[
    L=\begin{pmatrix}
        \delta&1&1\\1&\delta&1\\1&1&\delta
    \end{pmatrix}
    \]
    which has a Smith normal form $(1)\oplus(\delta-1)\oplus((\delta-1)(\delta+2))$, and the determinant $(\delta-1)^2(\delta+2)$.
    Hence we have the following.
    \begin{enumerate}
        \item $L$ is surjective if and only if both $\delta-1$ and $\delta+2$ are invertible in $\k$.
        \item $L$ is injective if and only if both $\delta-1$ and $\delta+2$ are regular in $\k$.
    \end{enumerate}
    Note that if $L$ is surjective, then it is injective as well, and hence an isomorphism.
\end{example}

Here we describe the structure of the quasi-Brauer category $\qB$. We will explain the matrix Brauer category $\mB$ in the next subsection.

The objects of $\qB$ are non-negative integers and the hom-space $\qB(p,q)$ is the free $\k$-module with basis $\B(p,q)$ as is the case with the Brauer category.
The composition in $\qB$ is given as follows.
Let $g\in \B(p,q), f\in \B(q,r)$.
We set $f\circ g=0$ if there is a triple $(i,j,k)$ of distinct elements of $[q]$ such that we have $\{i^-,j^-\}\in g$ and $\{j^+,k^+\}\in f$.
Otherwise, we set $f\circ g:=f\hat{\circ} g\in \B(p,r)$.
(Note that $\delta$ does not appear in the category $\qB$).
For example, $\{\{1^+,2^+\}\}\circ\{\{1^-,2^-\}\}=1_0$ in $\qB$.
The identity morphism $\id_p^{\qB}\in \qB(p,p)$ is the same as that in the Brauer category, that is, $1_p\in \B^0(p,p)$. 

The quasi-Brauer category has a $\k$-linear PROP structure in a way similar to the Brauer category.
Let $\widetilde{\qB}$ denote the $\k$-linear PROP generated by one object $1$ and two morphisms
\begin{gather*}
    c:2\to0,\quad
    \overline{c}:0\to2
\end{gather*}
with relations
\begin{gather}
    (\id_1\otimes c)(\overline{c}\otimes \id_1)=0
    ,\quad
    cP_{1,1}=c,\quad P_{1,1}\overline{c}=\overline{c},\quad 
    c\overline{c}=\id_0.
\end{gather}
Note that these identities imply
$(c\otimes \id_1)(\id_1\otimes \overline{c})=0$.
We have a unique morphism of $\k$-linear PROPs $F:\widetilde{\qB}\to \qB$ which maps $c$ to $\{\{1^+,2^+\}\}$ and $\overline{c}$ to $\{\{1^-,2^-\}\}$.

\begin{proposition}
    The functor $F: \widetilde{\qB}\to \qB$ is an isomorphism.
\end{proposition}

\begin{proof}(Sketch)
We can check that $F$ is full since morphisms of $\qB$ are generated by $F(c)$ and $F(\overline{c})$. For each element $g\in \B(p,q)\subset \qB(p,q)$, we choose a lift $\tilde{g}\in \widetilde{\qB}(p,q)$ along $F$.
Set $\widetilde{\B}(p,q)=\{\tilde{g}\mid g\in\B(p,q)\}\subset\widetilde{\qB}(p,q)$. Then we can check that any morphism in $\widetilde{\qB}(p,q)$ can be written as a $\k$-linear combination of elements in $\widetilde{\B}(p,q)$, using the relations in $\widetilde{\qB}$. This means that $F$ is faithful.
Therefore, the functor $F: \widetilde{\qB} \to \qB$ is an isomorphism of categories. 
\end{proof}

\subsection{The matrix Brauer category $\mB$}\label{subsectionmB}

The objects of $\mB$ are non-negative integers.
The hom-space $\mB(p,q)$ is the free $\k$-module with basis $\B(p,q)$ as is the case with $B$ and $\qB$.
The composition in $\mB$ is the following \emph{matrix composition}:
$$
(f^+f^0f^-)(g^+g^0g^-) =
\begin{cases}
 \delta_{\overline{f^-},\: g^+}\;f^+(f^0 g^0)g^- & t=s,\\
 0 & t\neq s,
\end{cases}
$$
where $f^+\in \B^{+,m}(t,r),\; f^0\in \B^{0}(t,t),\; f^-\in \B^{-,m}(q,t),\; g^+\in \B^{+,m}(s,q),\; g^0\in \B^{0}(s,s),\; g^-\in \B^{-,m}(p,s)$.
The identity morphism $\id^{\mB}_p\in\mB(p,p)$ is the sum of the identity matrices:
$$
1_p^{\mB} := \sum_{r\preceq p}\sum_{f\in \B^{+,m}(r,p)} f \id_r \bar{f}.
$$

For $p\ge0$, we call the endomorphism algebra $\mB(p):=\mB(p,p)$ of the category $\mB$ the \emph{matrix Brauer algebra}.

\begin{proposition}\label{matrixBrauersemisimple}
    For each $p,q\ge0$, we have a canonical 
    $\k$-module isomorphism
    \begin{gather}\label{matrixBrauerdecomposition-pq}
        \mB(p,q)\cong\bigoplus_{r\preceq p,q} M_{d_{q,r},d_{p,r}}(\S_r),
    \end{gather}
     where $d_{p,r}=|\B^{+,m}(r,p)|=(\rank B^+(r,p))/r!$.
     If $p=q$, then we have an algebra isomorphism
    \begin{gather}\label{matrixBrauerdecomposition}
        \mB(p)\cong\bigoplus_{r\preceq p} M_{d_{p,r}}(\S_r).
    \end{gather}     
    In particular, if $\k$ is a field of characteristic $0$, then the algebra $\mB(p)$ is split semisimple.
\end{proposition}

\begin{proof}
The first part follows from the decomposition $$\mB(p,q)\cong \bigoplus_{r\preceq p,q} B^+(r,q)\ot_{\S_r}B^-(p,r)$$
and the fact that $B^+(r,q)$ is the free right $\S_r$-module with basis $\B^{+,m}(r,q)$, and that $B^-(p,r)$ is the free left $\S_r$-module with basis $\B^{-,m}(p,r)$.

For $p=q$, it follows from the definition of the category structure of $\mB$ that the map $\mB(p)\to\bigoplus_{r\preceq p}M_{d_{p,r}}(\S_r)$ is an algebra isomorphism.

If $\k$ is a field of characteristic $0$, then each $\S_r$ is split semisimple.
Therefore, each of the corresponding matrix algebras is split semisimple, and hence so is $\mB(p)$.
\end{proof}

The algebra on the right hand side of \eqref{matrixBrauerdecomposition} already appears in Brown \cite[Theorem 3.2A]{Brown1}. See also Corollary~\ref{isom-Br-mBr}.

We will prove that the matrix Brauer category $\mB$ can be embedded into the category $\S^{\op}\mmod$ of $\S^{\op}$-modules.
We have a functor
\begin{gather*}
    \mB^+: \S^{\op}\times \mB\to \k\mmod
\end{gather*}
such that $\mB^+(p,q)=B^+(p,q)$ and for $(\sigma, g)\in \S^{\op}(p,p)\times \mB(q,q')$, the $\k$-linear map 
$$\mB^+(\sigma,g):B^+(p,q)\to B^+(p,q')$$
is defined for $f\in B^+(p,q)$ by
$$\mB^+(\sigma,g)(f)=(\mu_{b'}(g\ot f))\sigma=g^+ b'(g^-\ot f)\sigma,$$
where we write $g=g^+\ot g^-\in B(q,q')\cong (B^+\ot_{\S}B^-)(q,q')$.

The functor $\mB^+$ induces a functor
\begin{gather*}
    \widehat{\mB^+}: \mB\to \S^{\op}\mmod
\end{gather*}
such that $\widehat{\mB^+}(q)(p)=\mB^+(p,q)$.
We may regard the functor $\widehat{\mB^+}$ as a left action of $\mB$ on $B^+$. 
Via the bijection appearing in Remark~\ref{actionendo}, the action $\widehat{\mB^+}$ of $\mB$ on $B^+$ corresponds to the identity morphism $\id_{\mB}$ of $\mB$.

We obtain the following. 

\begin{proposition}\label{mB+fullyfaithful}
The functor $\widehat{\mB^+}: \mB\to \S^\op\mmod$ is fully faithful.
\end{proposition}
\begin{proof}
    We have 
    \begin{gather*}
    \begin{split}
       \rank \Hom_{\S^{\op}\mmod}(B^+(-,q),B^+(-,q'))
       &=\sum_{r} \rank \Hom_{\S^{\op}_r\mmod}(B^+(r,q),B^+(r,q'))\\
       &=\sum_{r}  \rank B^+(r,q')\cdot |\B^{+,m}(r,q)|\\
       &=\sum_{r} |\B^{+,m}(r,q')|\cdot|\B^0(r,r)|\cdot |\B^{+,m}(r,q)| \\
       &=\sum_{r} |\B^{+,m}(r,q')|\cdot |\B^0(r,r)| \cdot |\B^{-,m}(q,r)|\\
       &=\rank\mB(q,q').
    \end{split}
    \end{gather*}
    Therefore, it suffices to prove that the map
    \begin{gather*}
        \widehat{\mB^+}: \mB(q,q')\to \Hom_{\S^{\op}\mmod}(B^+(-,q),B^+(-,q'))
    \end{gather*}
    is surjective.
    Actually, for any $\S^{\op}$-module map $\phi=\{\phi_s\}_{s}:B^+(-,q)\to B^+(-,q')$, we have
    \begin{gather*}
        \widehat{\mB^+}(\sum_{s\preceq q}\sum_{f^+\in \B^{+,m}(s,q)} \phi_s(f^+)\overline{f^+})=\phi
    \end{gather*}
    since we have for any $h=h^+h^0\in B^+(r,q)$ with $h^+\in \B^{+,m}(r,q), h^0\in \B^0(r,r)$,
    \begin{gather*}
    \begin{split}
        \widehat{\mB^+}(\sum_{s\preceq q}\sum_{f^+\in \B^{+,m}(s,q)} \phi_s(f^+)\overline{f^+})(h^+h^0)
        &=\sum_{s\preceq q}\sum_{f^+\in \B^{+,m}(s,q)} \delta_{\overline{f^+},\overline{h^+}} \phi_s(f^+)h^0\\
        &=\phi_r(h^+)h^0\\
        &=\phi_r(h^+h^0)\\
        &=\phi_r(h).
    \end{split}
    \end{gather*}
This completes the proof.
\end{proof}

\subsection{Center of $\mB(p)$}

In the following, we study the center $Z(\mB(p))$ of the matrix Brauer algebra $\mB(p)$ over
a field $\k$ of characteristic $0$, using well-known facts from the representation theory of symmetric groups, for which we refer the reader to \cite{Fulton--Harris}.

For a partition $\lambda$ of $p$, denoted by $\lambda\vdash p$, we have the primitive central idempotent $e_\lambda$ of $\S_p$, which can be defined by
\begin{gather}\label{elambda}
    e_{\lambda} = \frac{1}{H(\lambda)^2} \sum_{T: \text{ a tableau of }\lambda} c_T\in \S_p.
\end{gather}
Here $H(\lambda)\in\Z_{>0}$ is the product of all {hook lengths} for the {Young diagram} of $\lambda$, 
and $c_T\in\Z[\gpS_p]$ is the {Young symmetrizer} associated to the {tableau} $T$.
By Young's classical result, $\{e_\lambda\mid \lambda\vdash p\}$ is
the complete set of
primitive central idempotents of $\S_p$. 
(See \cite[Proposition 4.6]{DLS} and \cite[Corollary VI.3.7]{Simon}.)

It follows that the set $\{e_\lambda I_{d_{p,r}}\mid \lambda \vdash r\}$ is the complete set of primitive central idempotents in 
the matrix algebra $M_{d_{p,r}}(\S_r)$.
Via the isomorphism \eqref{matrixBrauerdecomposition}, the matrix $e_\lambda I_{d_{p,r}}$ corresponds to 
\begin{gather*}
    e^{(p)}_{\lambda}:=\sum_{f\in \B^{+,m}(r,p)}f e_{\lambda}\bar{f}.
\end{gather*}
Therefore, we obtain the following.

\begin{proposition}\label{pciformBp}
    The set 
    $$\{e^{(p)}_{\lambda}\mid \lambda \vdash r, \,r\preceq p\}$$
    forms the complete set of primitive central idempotents 
    of the matrix Brauer algebra $\mB(p)$.
\end{proposition} 

\section{Perfectness of the pairing $b$ for the Brauer category}\label{sectionpairingofB}
We study perfectness of the pairing $b$ for the Brauer category.
We show that the pairing $b_{p,q}:B^-(q,p)\ot_{\S_q} B^+(p,q)\to \S_p$ is perfect if and only if $\delta$ satisfies a specific condition. In this case, $\delta$ is referred to as \emph{non-singular}. We see that our definition of singularity is equivalent to the one given by K\"onig and Xi \cite{Koenig-Xi}.

\subsection{Perfectness of the pairing $b$}\label{sectionperfectnessofb}
We consider the Brauer category $B=B^{\Z[x],x}$ over $\Z[x]$ with parameter $x$.
For $p\preceq q $, let $n=|\B^{+,m}(p,q)|$.
Consider the matrix
    \begin{gather*}
        A_{p,q} = (a_{f,g})_{f,g\in\B^{+,m}(p,q)}\in M_n(\Z[x][\gpS_p])
    \end{gather*}
    with $a_{f,g}=b(\bar f\ot g)$.
    This matrix has been considered in K\"onig--Xi \cite[Section 4]{Koenig-Xi}; see Section~\ref{Rem:KX} below.
    
The left regular representation $L: \Z[\gpS_p]\hookrightarrow M_{p!}(\Z)$ induces an injective ring homomorphism $\iota: \Z[x][\gpS_p]\hookrightarrow M_{p!}(\Z[x])$, which induces an injective ring homomorphism
$$\iota:M_n(\Z[x][\gpS_p])\hookrightarrow M_{np!}(\Z[x]).$$
Let $f_{p,q}(x):=\det\iota(A_{p,q})\in\Z[x]$.

\begin{lemma}\label{lem:Apq-det-monic}
    In the above setting, $f_{p,q}(x)$ is a monic polynomial, and the matrix $A_{p,q}$ is invertible in $M_n(\Z[x,f_{p,q}(x)^{-1}][\gpS_p])$.
\end{lemma}

\begin{proof}
Let $k=\frac{q-p}{2}$.
For $f,g\in\B^{+,m}(p,q)$, we have
\[
    a_{f,f}=x^k1_p,
    \qquad
    \deg_x (a_{f,g})<k\quad(f\ne g).
\]
It follows from $\iota(a_{f,f})=x^k I_{p!}$ that we have
\[
    \iota(A_{p,q})=x^k I_{np!}+R,
\]
where every entry of $R$ has degree $<k$.
Therefore, the polynomial $f_{p,q}(x)=\det\iota(A_{p,q})$ is monic.

Let
\[
    g(t)=\det(tI-\iota(A_{p,q}))\in \Z[x][t]
\]
be the characteristic polynomial of $\iota(A_{p,q})$. 
By the Cayley--Hamilton theorem, we have
\[
  \iota(g(A_{p,q}))=g(\iota(A_{p,q}))=0.
\]
It follows from the injectivity of $\iota$ that $g(A_{p,q})=0$ in $M_n(\Z[x][\gpS_p])$. 
We have
\[
    g(0)=(-1)^{np!}f_{p,q}(x).
\]
Set 
$$h(t)=-g(0)^{-1}t^{-1}(g(t)-g(0))\in \Z[x,f_{p,q}(x)^{-1}][\gpS_p][t].
$$
Then $h(A_{p,q})\in M_{n}(\Z[x,f_{p,q}(x)^{-1}][\gpS_p])$ is an inverse of $A_{p,q}$.
\end{proof}

\begin{proposition}
\label{lem:copairing}
    If $p\preceq q$,
    then the pairing 
    \[b=b_{p,q}:B^-(q,p)\ot_{\S_q} B^+(p,q)\to \S_p\]
    over $\Z[x]$ with parameter $x$
    is perfect over $\Z[x,f_{p,q}(x)^{-1}]$, i.e., there is an $\S_q$-bimodule map
\[
d_{p,q}: \S_q\to B^+(p,q)\ot_{\S_p}B^-(q,p),
\]
defined over $\Z[x,f_{p,q}(x)^{-1}]$, 
that satisfies the zigzag identities with $b_{p,q}$.
\end{proposition}

\begin{proof}

By Lemma~\ref{lem:Apq-det-monic}, the matrix $A_{p,q}$ is invertible. Let 
$$(A_{p,q})^{-1}=(a^{g,f})_{g,f\in\B^{+,m}(p,q)}\in 
M_n(\Z[x,f_{p,q}(x)^{-1}][\gpS_p])$$
be the inverse of $A_{p,q}$.
Let 
$$d (1_q)_p:=\sum_{g,f\in\B^{+,m}(p,q)} g\; a^{g,f} \ot \bar{f}\in  B^+(p,q)\ot_{\S_p}B^-(q,p).$$

We need to prove the zigzag identities:
\begin{gather}
    \label{zigzag1pq}
    (\id_{B^+(p,q)}\ot b_{p,q})(d(1_q)_p \ot\id_{B^+(p,q)})=\id_{B^+(p,q)},\\ 
    \label{zigzag2pq}
    (b_{p,q}\ot\id_{B^-(q,p)})(\id_{B^-(q,p)}\ot d(1_q)_p)=\id_{B^-(q,p)}.
\end{gather}
We prove \eqref{zigzag1pq}.
For $h\in\B^{+,m}(p,q)$ and $\rho\in \gpS_p$, we have
\begin{gather*}
    \begin{split}
        (\id\ot b_{p,q})(d(1_q)_p \ot\id)(h\rho)
        &=(\id\ot b_{p,q})(d(1_q)_p\ot h\rho)\\
        &=\sum_{g,f\in\B^{+,m}(p,q)}
        g\; a^{g,f} b_{p,q}(\bar f\ot h)\rho\\
        &=\sum_{g,f\in\B^{+,m}(p,q)}
        g\; a^{g,f} a_{f,h}\rho\\
        &=\sum_{g\in\B^{+,m}(p,q)}g\; \delta^g_h\;\rho\\
        &=h\rho.
    \end{split}
\end{gather*}
The other identity \eqref{zigzag2pq} can be similarly proved.

Next, we show that  $d(1_q)_p$ commutes with any $\sigma\in\gpS_q$.
For this, by \eqref{zigzag1pq}, it suffices to show that for each $h\in\B^{+,m}(p,q)$ and $\sigma\in\gpS_q$ we have
\begin{gather}\label{commutative}
(b_{p,q}(\bar h\ot -)\ot\id_{B^-})(\sigma d(1_q)_p)=
(b_{p,q}(\bar h\ot -)\ot\id_{B^-})( d(1_q)_p\sigma).
\end{gather}
The left-hand side of \eqref{commutative} is 
\begin{gather*}
 \begin{split}
 (b_{p,q}(\bar h\sigma\ot -)\ot\id_{B^-})(d(1_q)_p)=(b_{p,q}\ot\id_{B^-(q,p)})(\id_{B^-(q,p)}\ot d(1_q)_p)(\bar h \sigma)=\bar h \sigma
 \end{split}
\end{gather*}
by \eqref{zigzag2pq}, which is equal to the right hand side of \eqref{commutative}.
Thus we have \eqref{commutative}.

Then we obtain an $\S_q$-bimodule map
$$d_{p,q} : \S_q\to B^+(p,q)\ot_{\S_p}B^-(q,p)$$
defined by
$d_{p,q} (\sigma) =\sigma d (1_q)_p$
for $\sigma\in \gpS_q$.

By the zigzag identities \eqref{zigzag1pq} and \eqref{zigzag2pq},
we have
\begin{gather*}
    (\id_{B^+}\ot b_{p,q})(d_{p,q} \ot\id_{B^+})=\id_{B^+},\quad 
    (b_{p,q}\ot\id_{B^-})(\id_{B^-}\ot d_{p,q})=\id_{B^-}. 
\end{gather*}
Therefore, the pairing $b_{p,q}$ is perfect over $\Z[x,f_{p,q}(x)^{-1}]$.
\end{proof}

Let $\k$ be any commutative ring and $\delta\in \k$.
Let $p\preceq q$.
We say that $\delta$ is \emph{non-singular} at rank $(p,q)$ with $p\preceq q$ if $f_{p,q}(\delta)$ is invertible in $\k$.
Then we obtain the following proposition.

\begin{proposition}\label{prop:perfectnonsingular}
Let $\k$ be any commutative ring and $\delta\in \k$.
Let $p\preceq q$.
The pairing $b_{p,q}$ for $B^{\k,\delta}$ is perfect if and only if $\delta$ is non-singular at rank $(p,q)$.
\end{proposition}

\begin{proof}
By Proposition~\ref{lem:copairing}, the $\S_q$-bimodule map $d_{p,q}$ is defined over $\Z[x,f_{p,q}(x)^{-1}]$.
If $\delta$ is non-singular at $(p,q)$, then we have the ring homomorphism $g:\Z[x,f_{p,q}(x)^{-1}]\to \k$ that maps $x$ to $\delta$. Therefore, $d_{p,q}$ gives rise to an $\S_q$-bimodule map defined over $\k$, which means that $b_{p,q}$ for $B^{\k,\delta}$ is perfect.

Suppose that $b_{p,q}$ for $B^{\k,\delta}$ is perfect with copairing $d_{p,q}$.
Set $$(A_{p,q})^{-}:=(a^{g,f})_{g,f\in \B^{+,m}(p,q)}\in M_n(\S_p),$$ where 
$$a^{g,f}=(b'\ot b')(B^-\ot d_{p,q}\ot B^+)(\overline{g}\ot f).$$
Then we have 
\begin{gather*}
\begin{split}
     (A_{p,q}(A_{p,q})^{-})_{f,h}
     &=\sum_{g\in \B^{+,m}(p,q)}b_{p,q}(\overline{f}\ot g)(b'\ot b')(B^-\ot d_{p,q}\ot B^+)(\overline{g}\ot h)\\
     &=(b_{p,q}\ot b'\ot b')(B^-\ot d'\ot d\ot B^+)(\overline{f}\ot h)\\
     &=b'(\overline{f}\ot h)\\
     &=\delta_{f,h}1_p.
\end{split}
\end{gather*}
In a similar way, we have $((A_{p,q})^{-}A_{p,q})_{f,h}=\delta_{f,h}1_p$, which implies that $A_{p,q}$ is invertible in $M_n(\S_p)$.
Therefore, the matrix $\iota(A_{p,q})\in M_{np!}(\k)$ is also invertible, and thus $f_{p,q}(\delta)=\det(\iota(A_{p,q}))\in \k$ is invertible.
This completes the proof.
\end{proof}

We say that $\delta$ is \emph{non-singular at ranks $\preceq r$} if $\delta$ is non-singular at ranks $(p,q)$ for any  $p\preceq q\preceq r$.
We also say that $\delta$ is (simply) \emph{non-singular} if 
$\delta$ is non-singular at ranks $\preceq r$ for all $r$.
We say that $\delta$ is \emph{singular} (at ranks $\preceq r$) if it is not non-singular (at ranks $\preceq r$).

\subsection{Comparison with K\"onig and Xi's notion of singularity}\label{Rem:KX}

K\"onig and Xi \cite{Koenig-Xi} introduced the notion of inflated algebras, and proved that the Brauer algebras are iterated inflations of the group algebras of symmetric groups over a field $\k$.
Here, we briefly recall their construction and compare our notion of singularity for $\delta$ with that of K\"onig and Xi \cite{Koenig-Xi}.

Let $\k$ be any field.
Let $C$ be a $\k$-algebra with an involution $i$, $V$ a $\k$-vector space with basis $\{v_1,\dots,v_d\}$, and $\varphi:V\ot V\to C$ a $\k$-bilinear form such that $i(\varphi(v,w))=\varphi(w,v)$ for $v,w\in V$.
Then the inflated algebra $A$ is a (possibly non-unital) associative algebra whose underlying space is $V\ot V\ot C$.

In the sense of K\"onig and Xi \cite[Section 3]{Koenig-Xi}, the bilinear form $\varphi$ is \emph{singular} if for some simple $C$-module $L$, we have
$$N(L):=\{\sum_{v\in V,l\in L} v\ot l\mid \sum_{v,l} \overline{\varphi}(w,v)l=0\;\text{ for all $w\in V$}\}\neq 0,$$
where $\overline{\varphi}$ is the composition of $\varphi$ with the projection $C\twoheadrightarrow C/\operatorname{rad}(C)$, where $\operatorname{rad}(C)$ is the Jacobson radical of $C$.
They also considered the matrix $\Phi=(\varphi(v_i,v_j))_{1\le i,j\le d}$.
It follows from \cite[Corollary 3.5 and Proposition 4.2]{Koenig-Xi} that the matrix $\Phi$ is invertible if and only if $\varphi$ is non-singular.

By \cite[Theorem 5.6]{Koenig-Xi}, 
each layer of the Brauer algebra is the inflation of $C:=\S_p$ along $V:=\k\B^{+,m}(p,q)$ with basis $\B^{+,m}(p,q)$, where $\varphi_{p,q}:V\ot V\to C$ is defined for $v_i,v_j\in \B^{+,m}(p,q)$ by $\varphi_{p,q}(v_i\ot v_j)=b_{p,q}(\overline{v_i}\ot v_j)$.
Note that the matrix $\Phi$ for the Brauer algebra coincides with the matrix $A_{p,q}$ that we considered in Section~\ref{sectionperfectnessofb}.
Since $A_{p,q}$ is invertible if and only if $\delta$ is non-singular at rank $(p,q)$, we obtain the following correspondence.

\begin{lemma}\label{KXsingular}
Let $\k$ be any field and $\delta\in \k$.
Let $p\preceq q$.
Then $\delta$ is non-singular at rank $(p,q)$ if and only if $\varphi_{p,q}$ is non-singular in the above-defined sense.
\end{lemma}

K\"onig and Xi \cite[Section 7]{Koenig-Xi} called $\delta\in \k$ for the Brauer algebra $B^{\k,\delta}(r)$ \emph{singular} if the associated bilinear form $\varphi_{p,q}$ is singular for some $p,q$ with $p\preceq q\preceq r$. Therefore, by Lemma~\ref{KXsingular}, the statement that $\delta$ is singular at ranks $\preceq r$ in our sense is equivalent to saying that $\delta$ is singular in the sense of K\"onig and Xi.
We have the following.

\begin{proposition}[K\"onig--Xi {\cite[Theorem 7.3]{Koenig-Xi}}]
\label{BMoritaequivalenttoS-field}
Let $\k$ be a field and let $r\ge 0$.
If $\delta\in\k$ is non-singular at ranks $\preceq r$,
then the Brauer algebra $B(r)$ is Morita equivalent to $\bigoplus_{q\preceq r}\S_q$.
\end{proposition}

\begin{remark}\label{remarkKoenig-Xi}
    If $\operatorname{char}\k=0$, then the converse of Proposition~\ref{BMoritaequivalenttoS-field} holds, which should be well known to experts. Indeed, if $\delta$ is singular, then $B(r)$ is not semisimple, but $\bigoplus_{q\preceq r}\S_q$ is semisimple.
\end{remark}

\subsection{The singular parameters}\label{subsec:singular-set}

Rui \cite{Rui} and Rui--Si \cite{Rui-Si-II}
determined the singular parameters for the Brauer algebra as follows.
(Here we need only the characteristic $0$ case.)

\begin{proposition}[{Rui--Si \cite[Corollary 2.5]{Rui-Si-II}}]\label{RuiSisingularset}
    Let $\k$ be a field of characteristic $0$, and let $r\ge 0$.
    Then $\delta\in \k$ is singular at ranks $\preceq r$ if and only if $\delta\in X_r$, where $X_r\subset \Z$ is defined by
\begin{gather*}
    X_r = X_r^0\cup\{i,-2i\mid 1\le i\le r-2,i\in\Z\}\cup\{j\in\Z\mid 4-r\le j\le -1\}
\end{gather*}
with
\begin{gather*}
    X_r^0=\begin{cases}
        \{0\}&(r\neq 0,1,3,5)\\
        \emptyset &(r=0,1,3,5).
    \end{cases}    
\end{gather*}
\end{proposition}

For $p\preceq q$, let $R_{p,q}\subset \mathbb{C}$ be the set of all roots of the monic polynomial $f_{p,q}(x)\in \Z[x]$. Then we have 
\begin{gather}\label{fpq}
    f_{p,q}(x)=\prod_{m\in R_{p,q}} (x-m)^{c_m}
\end{gather}
with $c_m\in \Z_{> 0}$.
Proposition~\ref{RuiSisingularset} implies the following.

\begin{corollary}\label{Xrroot}
    Let $r\ge 0$.
    Then we have
    \begin{gather*}
        X_r=\bigcup_{p\preceq q\preceq r} R_{p,q}.
    \end{gather*}
\end{corollary}

\begin{proof}
    By Proposition~\ref{RuiSisingularset}, for $\delta\in \mathbb{C}$, $\delta\in X_r$ if and only if $\delta$ is singular at ranks $\preceq r$.
    By the definition of singularity of $\delta$, this is equivalent to the condition that $f_{p,q}(\delta)=0$ for some $p\preceq q\preceq r$, which means that $\delta\in \bigcup_{p\preceq q\preceq r} R_{p,q}$.
\end{proof}

We can derive from Proposition~\ref{RuiSisingularset} the following characterization of non-singularity of $\delta$ for any commutative ring $\k$, which recovers the results of Rui \cite{Rui} and Rui--Si \cite{Rui-Si-II} for a field $\k$.

\begin{proposition}\label{Rui-Si-singularset2}
Let $\k$ be a commutative ring, and let $r\ge 0$. 
Then $\delta\in \k$ is non-singular at ranks $\preceq r$
if and only if 
for every $m\in X_r$, $\delta-i(m)\in \k$ is invertible, where $i:\Z\to\k$ is the unique ring homomorphism.
\end{proposition}

\begin{proof}
By definition, $\delta\in \k$ is non-singular at ranks $\preceq r$ if and only if for any $p\preceq q\preceq r$, $f_{p,q}(\delta)$ is invertible in $\k$.
By \eqref{fpq}, this holds if and only if $\delta-i(m)\in \k$ is invertible for any $m\in R_{p,q}$ with any $p\preceq q\preceq r$.
Therefore, the statement follows from Corollary~\ref{Xrroot}.
\end{proof}

\section{Perfectness of the Brauer category}
\label{sectionPerfectnessBrauer}
We prove that $B$ is perfect exactly when $\delta$ is non-singular, and then derive the isomorphisms with $\qB$ and $\mB$, Morita equivalence with $\S$, and applications to precontractions and central idempotents.

\subsection{Perfectness of the (truncated) Brauer category}

In this subsection, we consider the (truncated) Brauer category over a commutative ring $\k$ with $\delta\in\k$ unless otherwise mentioned.
For $r\ge 0$, let $B_{\preceq r}$ denote the full subcategory of $B$ such that $\Ob(B_{\preceq r})=\{q\in\Z_{\ge0}\mid q\preceq r\}$, which is called the truncated Brauer category.
Let $\S_{\preceq r}$ denote the full subcategory of $\S$ such that $\Ob(\S_{\preceq r})=\{q\in\Z_{\ge0}\mid q\preceq r\}$.
We obtain the following theorem about the perfectness of the truncated Brauer category.

\begin{theorem}\label{Brauerperfectnew}
Let $r\ge0$ be an integer. 
Then the truncated Brauer category $B_{\preceq r}$ is a perfect quasi-cellular category
if and only if $\delta$ is non-singular at ranks $\preceq r$.
\end{theorem}

\begin{proof}
In a way similar to Proposition~\ref{Brauerqc}, one can check that the truncated Brauer category $B_{\preceq r}$ and the truncated versions $b'_{\preceq r}$ of $b'$ and $d'_{\preceq r}$ of $d'$ form a quasi-cellular category.

Recall that $\delta$ is non-singular at ranks $\preceq r$ if and only if $\delta$ is non-singular at rank $(p,q)$ for any $p\preceq q\preceq r$.
By Proposition~\ref{prop:perfectnonsingular}, $\delta$ is non-singular at rank $(p,q)$ if and only if $b_{p,q}$ is perfect.
If the truncated Brauer category $B_{\preceq r}$ is perfect, then the copairing $d_{\preceq r}:\S_{\preceq r}\to B_{\preceq r}^+\ot_{\S_{\preceq r}}B_{\preceq r}^-$ of $b_{\preceq r}$ induces the copairing $d_{p,q}$ of $b_{p,q}$ for each $p\preceq q\preceq r$.
Therefore, it suffices to show that the copairings $d_{p,q}$ of $b_{p,q}$ with $p\preceq q\preceq r$ induce a copairing $d_{\preceq r}$ of $b_{\preceq r}$, which can be checked as follows.
Let $d_{\preceq r}: \S_{\preceq r}\to B_{\preceq r}^+\ot_{\S_{\preceq r}}B_{\preceq r}^-$ be the $\S_{\preceq r}$-bimodule map defined for $q\preceq r$ by
$$(d_{\preceq r})_q=\sum_{p\preceq q}d_{p,q}:\S_{\preceq r}(q,q)=\S_q\to(B_{\preceq r}^+\ot_{\S_{\preceq r}}B_{\preceq r}^-)(q,q)=\bigoplus_{p\preceq q}B_{\preceq r}^+(p,q)\ot_{\S_{p}}B_{\preceq r}^-(q,p).$$
Then by the zigzag identities for $d_{p,q}$ and $b_{p,q}$, we obtain that $d_{\preceq r}$ is the copairing of $b_{\preceq r}$ for the truncated Brauer category $B_{\preceq r}$.
This completes the proof.
\end{proof}

\begin{corollary}
Let $r\ge 0$ be an integer.
Then the truncated Brauer category $B_{\preceq r}$ is perfect if and only if for any ring homomorphism $g:\k\to F$ with $F$ a field, the truncated Brauer category $(B^{F,g(\delta)})_{\preceq r}$ is perfect.
\end{corollary}

\begin{proof}
By Theorem~\ref{Brauerperfectnew} and Proposition~\ref{Rui-Si-singularset2}, $(B^{\k,\delta})_{\preceq r}$ is perfect if and only if for every $m\in X_r$, $\delta-i(m)$ is invertible in $\k$.
Similarly, for any ring homomorphism $g:\k\to F$ with $F$ a field, $(B^{F,g(\delta)})_{\preceq r}$ is perfect if and only if for every $m\in X_r$, $g(\delta)-gi(m)$ is invertible in $F$.
Therefore, the statement follows from the following claim.

\begin{claim}
    An element $x\in \k$ is invertible if and only if for any ring homomorphism $g:\k\to F$ with $F$ a field, $g(x)\in F$ is invertible.
\end{claim}
We can prove Claim 1 as follows.
    The sufficiency is obvious.
    Suppose that for any ring homomorphism $g:\k\to F$ with $F$ a field, $g(x)\in F$ is invertible. Then for any maximal ideal $\mathfrak m\subset \k$, the projection from $\k$ to the field $\k/\mathfrak m$ maps $x$ to a non-zero element, which implies that $x\in \k\setminus \mathfrak m$. Therefore, we have $x\in \bigcap (\k\setminus \mathfrak m)=\k\setminus \bigcup \mathfrak m=\k^\times$, and thus $x\in\k$ is invertible.
\end{proof}

In a similar way, we obtain the following characterization of the perfectness of the Brauer category.

\begin{theorem}\label{Brauerperfect}
 The Brauer category $B$ is a perfect quasi-cellular category if and only if $\delta$ is non-singular.
 Moreover, the Brauer category $B$ is perfect if and only if for any ring homomorphism $g:\k\to F$ with $F$ a field, the Brauer category $B^{F,g(\delta)}$ is perfect.
\end{theorem}

By combining Theorem~\ref{isomofqC} and Theorem~\ref{Brauerperfect}, we obtain the following.

\begin{theorem}\label{isomBrauers}
We have isomorphisms of $\k$-linear categories
    \begin{gather*}
       B\xrightarrow[\cong]{F_{d}}
       \pmB\xrightarrow[\cong]{K_{b,b'}}
       \mB\xrightarrow[\cong]{F_{d'}^{-1}} 
       \qB,
    \end{gather*}
    if and only if $\delta$ is non-singular.
    Here $\pmB$ is the $\k$-linear category over $\S$ corresponding to the monoid $B_b$ in $\Sbim$.
\end{theorem}

We call $\pmB$ the \emph{pseudo-Brauer category}.
Let us write down the structure of $\pmB$, although we do not need it in the rest of this paper.
The objects of $\pmB$ are non-negative integers.
The hom-space $\pmB(p,q)$ is the free $\k$-module with basis $\B(p,q)$ as is the case with the categories $B$, $\qB$ and $\mB$.
The composition in $\pmB$ is given as follows.
For $f\in \B(q,r)$, let $f=f^+f^0f^-$ be the decomposition of the Brauer diagram with $f^+\in \B^{+,m}(t,r),\; f^0\in \B^{0}(t,t),\; f^-\in \B^{-,m}(q,t)$. Similarly, for $g\in \B(p,q)$, let $g=g^+g^0g^-$ with
$g^+\in \B^{+,m}(s,q),\; g^0\in \B^{0}(s,s),\; g^-\in \B^{-,m}(p,s)$.
If the Brauer diagram $f^-\hat{\circ}g^+\in \B(s,t)$ is level, then set $fg=f\circ g$, the composition in $B$;
otherwise, $f g=0$.
The identity morphism $\id_p^{\pmB}\in \pmB(p,p)$ is $d(1_p)$.

By combining Theorem~\ref{isomofqC} and Theorem~\ref{Brauerperfectnew}, we also have the truncated versions of the isomorphisms in Theorem~\ref{isomBrauers}.

\begin{theorem}\label{isomtruncatedBrauers}
Let $r\ge 0$.
The truncated version 
        $$L: B_{\preceq r}\to\mB_{\preceq r}$$
        of the linear functor $L:B\to\mB$
is an isomorphism
if and only if $\delta$ is non-singular at ranks $\preceq r$.
In this case,
we have isomorphisms of $\k$-linear categories
    \begin{gather*}
       B_{\preceq r}\xrightarrow[\cong]{F_{d}}
       \pmB_{\preceq r}\xrightarrow[\cong]{K_{b,b'}}
       \mB_{\preceq r}\xrightarrow[\cong]{F_{d'}^{-1}} 
       \qB_{\preceq r}.
    \end{gather*}
\end{theorem}

By combining this with Proposition~\ref{matrixBrauersemisimple}, we have the following.

\begin{corollary}\label{isom-Br-mBr}
Let $p\ge 0$.
If $\delta$ is non-singular at ranks $\preceq p$, then
    we have a canonical isomorphism
    \[
    L: B(p)\xrightarrow[\cong]{}\mB(p)\cong\bigoplus_{r\preceq p} M_{d_{p,r}}(\S_r),
    \]
    which gives a decomposition of the Brauer algebra $B(p)$ into a direct sum of matrix algebras over group algebras of symmetric groups.
\end{corollary}

Note that a non-canonical isomorphism $B(p)\cong\bigoplus_{r\preceq p} M_{d_{p,r}}(\S_r)$
has been given by Brown \cite[Theorem 3.2A]{Brown1} under the condition that $B(p)$ is semisimple.

\subsection{Precontractions for the Brauer categories}\label{subsectionpsiforB}

The definitions of precontractions in Sections~\ref{section1precontraction} and~\ref{subsectionqcpsitilde} specialize to the (truncated) Brauer category as follows.

Let $r\ge0$ be an integer.
By Proposition~\ref{precontractionforQC}, a precontraction for $B_{\preceq r}$ (resp. $B$) is a family of linear maps $\psi =\psi _{p,q}:B^+(p,q)\to B(p,q)$ for $p\preceq q\preceq r$ (resp. $p\preceq q$) satisfying the following conditions:
\begin{enumerate}
    \item ($\S$-bimodule map)
    \begin{gather*}\begin{split}
    \sigma \psi(f)&=\psi(\sigma f), \quad \text{for } f\in B^+(p,q), \sigma\in B^0(q,q)\text{ with } p\preceq q\\
    \psi(f)\sigma&=\psi(f \sigma), \quad \text{for } f\in B^+(p,q), \sigma\in B^0(p,p)\text{ with } p\preceq q,\\
    \end{split}\end{gather*}
    \item (Contraction properties)
\begin{gather*}
\begin{split}
     g\psi (f)&=\psi (\partial(g\ot f)), \quad \text{for }g\in B^-(q,s), f\in B^+(p,q) \text{ with } s\prec q, p\preceq q,\\
    \psi (f)g&=0, \quad \text{for }f\in B^+(p,q), g\in B^+(s,p) \text{ with }s\prec p \preceq q,
\end{split}
\end{gather*}
    \item (Decomposition of the identity morphisms)
\begin{gather*}
       \sum_{p\preceq q}
        \sum_{f\in \B^{+,m}(p,q)} \psi (f) \overline{f} =1_q.
\end{gather*}
\end{enumerate}
By the contraction properties and the equation 
$(\varepsilon^+\ot \varepsilon^-)\psi(1_p)=1_p$ corresponding to (P3), 
we obtain
\begin{gather}\label{psifpsi1}
    \psi(f)\psi(1_p)=\psi(f) \text{ for }f\in B^+(p,q).
\end{gather}

By \eqref{eq-psitilde}, we have the following.
\begin{proposition}\label{psiforB}
    If $\delta$ is non-singular at ranks $\preceq r$, then
    the linear maps $\psi _{p,q}:B^+(p,q)\to B(p,q)$ with $p\preceq q\preceq r$, which form the precontraction for $B_{\preceq r}$, exist and are given for $f\in B^+(p,q)$ by
\begin{gather}\label{psitildewithdtau}
    \psi_{p,q}(f)=(K^+_{b,b'})^{-1}_{p,q}(f)\,d^-(1_p)\in B(p,q).
\end{gather}
\end{proposition}

\subsection{Low-degree computation}\label{subsectionlowdegree}

For illustration, we give the values of $d$, $d^-$ and $\psi$ in the range $q\le 4$ under the assumption that $\delta$ is non-singular at ranks $\preceq q$.

Write $d=d^+$.
For $0\le q\le 3$, $d^\pm(1_q)$ is given as follows.
\begin{gather*}
        d^\pm(1_0)=1_0,\quad
        d^\pm(1_1)=1_1,\quad
        d^\pm(1_2)=1_2\pm\frac{1}{\delta}\overline{c_{1,2}}c_{1,2},\\
    \begin{split}
        d^\pm(1_3)=&1_3\pm\frac{\delta+1}{(\delta-1)(\delta+2)}(\overline{c_{1,2}}c_{1,2}+\overline{c_{1,3}}c_{1,3}+\overline{c_{2,3}}c_{2,3})\\
        &\mp\frac{1}{(\delta-1)(\delta+2)}(\overline{c_{1,2}}(c_{1,3}+c_{2,3})+ \overline{c_{1,3}}(c_{1,2}+c_{2,3})+ \overline{c_{2,3}}(c_{1,2}+c_{1,3})).
    \end{split}
\end{gather*}

We now consider $\psi(f)$ for $f\in B^+(p,q)$ with $p\preceq q\le 3$.
For $p=q$, since $\psi$ is an $\S$-bimodule map, $\psi $ is determined by the values $\psi (1_p)=d^-(1_p)$.
The precontraction $\psi :B^+(0,2)\to B(0,2)$ is given by
\begin{gather*}
    \psi (\overline{c_{1,2}})=\frac{1}{\delta}\overline{c_{1,2}},
\end{gather*}
and the precontraction $\psi :B^+(1,3)\to B(1,3)$ is given by
\begin{gather*}
    \psi (\overline{c_{i,j}})=\frac{\delta+1}{(\delta-1)(\delta+2)}\overline{c_{i,j}}-\frac{1}{(\delta-1)(\delta+2)}(\overline{c_{i,k}}+\overline{c_{j,k}})
\end{gather*}
for $\{i,j,k\}=\{1,2,3\}$.

In order to explain the case $q=4$, we make the following definition.
Let $J=B(4)\overline{c_{1,2}}c_{1,2}B(4)$.
The symmetric group $\gpS_4$ acts on $J$ by conjugation.
Let $J^{\gpS_4}$ denote the invariant space.
Define Brauer diagrams in $\B(4,4)$ by
\begin{gather*}
f_1=\ol{c_{1,2}}c_{1,2},\;\;
f_2=\ol{c_{1,2}}(12)c_{1,2},\;\;
f_3=\ol{c_{1,2}}c_{2,3},\;\;
f_4=\ol{c_{1,2}}(12)c_{2,3},\\    
f_5=\ol{c_{1,2}}c_{3,4},\;\;
f_6=\ol{c_{12,34}}c_{12,34},\;\;
f_7=\ol{c_{12,34}}c_{14,23},\;\;
\end{gather*}
where $c_{ij,kl}=\{\{i^+,j^+\},\{k^+,l^+\}\}\in\B^-(4,0)$.
For each $1\le i\le 7$, let $T_i\subset \B(4,4)$ denote the $\gpS_4$-orbit of $f_i$, and set $U_i=\sum_{g\in T_i}g$.
The sizes of $T_1,\dots,T_7$ are $6,6,24,24,12,3,6$, respectively.
One can check that $U_1,\dots,U_7$ form a basis of $J^{\gpS_4}$.

Then we have
\[\small
\begin{aligned}d(1_4)
={}&\,1_4
+\frac{\delta^{3}+4\delta^{2}-4}{\delta(\delta-2)(\delta+2)(\delta+4)}\,U_{1}
-\frac{4}{\delta(\delta-2)(\delta+2)(\delta+4)}\,U_{2}\\[1mm]
&-\frac{\delta+3}{(\delta-2)(\delta+2)(\delta+4)}\,U_{3}
+\frac{1}{(\delta-2)(\delta+2)(\delta+4)}\,U_{4}
+\frac{2}{\delta(\delta-2)(\delta+4)}\,U_{5}\\[1mm]
&+\frac{\delta+1}{\delta(\delta-1)(\delta+2)}\,U_{6}
-\frac{1}{\delta(\delta-1)(\delta+2)}\,U_{7}.
\end{aligned}
\]
\[\small
\begin{aligned}d^-(1_4)
={}&\,1_4
-\frac{\delta^{3}+4\delta^{2}-4}{\delta(\delta-2)(\delta+2)(\delta+4)}\,U_{1}
+\frac{4}{\delta(\delta-2)(\delta+2)(\delta+4)}\,U_{2}\\[1mm]
&+\frac{\delta+3}{(\delta-2)(\delta+2)(\delta+4)}\,U_{3}
-\frac{1}{(\delta-2)(\delta+2)(\delta+4)}\,U_{4}
-\frac{2}{\delta(\delta-2)(\delta+4)}\,U_{5}\\[1mm]
&+\frac{\delta^2+3\delta+6}{(\delta-1)(\delta-2)(\delta+2)(\delta+4)}\,U_{6}
-\frac{3\delta+2}{(\delta-1)(\delta-2)(\delta+2)(\delta+4)}\,U_{7}.
\end{aligned}
\]
The elements $\psi(f)$ for $f\in\B^+(p,4)$, $p=0,2$, can be derived from the identity for $d^-(1_4)$ above.

\subsection{Stable isomorphism $L:B\to \mB$ over a field of characteristic $0$}

Let $\k$ be a field of characteristic $0$.
Then Theorems~\ref{isomBrauers} or~\ref{isomtruncatedBrauers} imply that the family of functors 
$$L=\{L^{\delta}=K^{\delta}_{b,b'} F^{\delta}_{d}:B^{\k,\delta}\to \mB^{\k,\delta}\}_{\delta\in \k}$$
can be regarded as a ``stable isomorphism''. In the following theorem, we give a stable range in terms of $\psi $.

\begin{theorem}\label{stableisomorphism}
Let $\k$ be a field of characteristic $0$.
Then the family of functors $L=\{L^\delta:B^{\k,\delta}\to \mB^{\k,\delta}\}_{\delta\in \k}$ is a ``stable isomorphism'' in the following sense.
\begin{enumerate}
    \item If $\delta\in\k\setminus\Z$, then $L^\delta$ is an isomorphism.
    \item If $\delta\in\Z$, then the $\k$-linear map $L^\delta:B^{\k,\delta}(p,q)\to \mB^{\k,\delta}(p,q)$ is an isomorphism for all pairs $(p,q)$ with $q<\max\{\frac{-\delta+4}{2},\delta+2\}$.
\end{enumerate}
\end{theorem}

\begin{proof}
(1) If $\delta\in\k\setminus\Z$, then $\delta$ is non-singular by Proposition~\ref{RuiSisingularset}. Then by Theorem~\ref{isomBrauers}, $L^{\delta}$ is an isomorphism.

(2)    By Proposition~\ref{RuiSisingularset}, if $q<\max\{\frac{-\delta+4}{2},\delta+2\}$, then $\delta$ is non-singular at ranks $\preceq q$.
    Then by Theorem~\ref{Brauerperfectnew} and Proposition~\ref{perfectstack-invertible}, $d_{\preceq q}$ and $(d^-)_{\preceq q}$ are defined.
    Therefore, $\psi (f)$ is defined by equation \eqref{psitildewithdtau} for any $f\in B^+(r,q)$ with $r\preceq q$.
    Thus, we have $L^{-1}:\mB(p,q)\to B(p,q)$ defined for $f\in B^+(r,q), g\in B^-(p,r)$ by $L^{-1}(f g)=\psi (f)g$, which gives the inverse of $L^\delta$ by \eqref{eq-Linverse}.
\end{proof}

\begin{remark}
Let $\k$ be a field of characteristic $0$.
We also have the family of functors $L'=\{(L')^{\delta}:B^{\k,\delta}\to \mB^{\k,\delta}\}$ given in Remark~\ref{Lprime}, which is a stable isomorphism, where if $\delta\in\k\setminus\Z$, then $(L')^\delta$ is an isomorphism, and if $\delta\in \Z$, then the $\k$-linear map 
$(L')^{\delta}:B^{\k,\delta}(p,q)\to\mB^{\k,\delta}(p,q)$
is an isomorphism for all pairs $(p,q)$ with $p< \max\{\frac{-\delta+4}{2},\delta+2\}$.
This can be derived from Theorem~\ref{stableisomorphism} by reflection.
\end{remark}

\begin{remark}\label{L'inverseL}
By Remark~\ref{L'inverseL1}, if $\delta$ is non-singular, then we have a canonical automorphism $$(L')^{-1}L: B\to B$$ of linear categories over $\S$.
This is non-trivial since we have for $\overline{c_{1,2}}\in B^+(1,3)$,
    \begin{gather*}
         (L')^{-1}L(\overline{c_{1,2}})
         =(L')^{-1}(\delta \overline{c_{1,2}}+\overline{c_{1,3}}+\overline{c_{2,3}})=\delta \overline{c_{1,2}}+\overline{c_{1,3}}+\overline{c_{2,3}}.
    \end{gather*}
\end{remark}

\subsection{Morita equivalence between $B$ and $\S$}

Let $\k$ be any commutative ring and $\delta\in\k$.
Here we apply the results of Section~\ref{sectionMoritaC} to the (truncated) Brauer category.

Theorems~\ref{MoritaC},~\ref{Brauerperfectnew} and~\ref{Brauerperfect} imply the following Morita equivalences for the truncated Brauer category and the Brauer category.

\begin{theorem}\label{MoritaequivalencetruncatedBrauer}
    Let $r\ge 0$.
    The pair $(B^+_{\preceq r},B^-_{\preceq r})$ gives a Morita equivalence between $B_{\preceq r}$ and $\S_{\preceq r}$ if and only if $\delta$ is non-singular at ranks $\preceq r$.
\end{theorem}

\begin{theorem}\label{BMoritaequivalenttoS}
The pair $(B^+,B^-)$ gives a Morita equivalence between $B$ and $\S$ if and only if $\delta$ is non-singular.
\end{theorem}

We regard the Brauer algebra $B(r)$ as the full subcategory of the truncated Brauer category $B_{\preceq r}$ with a single object $r$.
Then we obtain the following.

\begin{lemma}\label{BraueralgebraandtruncatedBrauer}
If $r=0$, if $r$ is odd, or if $\delta\in\k^\times$, then $B(r)$ is Morita equivalent to $B_{\preceq r}$.
\end{lemma}

\begin{proof}
The case $r=0$ clearly holds; assume $r>0$.
For each $p\preceq r$, $p\neq 0$, define $f^+_p\in B^+(p,r)$ and $f^-_p\in B^-(r,p)$ by
\begin{gather*}
    f^+_p =\overline{c_{r-1,r}}\cdots \overline{c_{p+3,p+4}}\,\overline{c_{p+1,p+2}},\\
    f^-_p =c_{p,p+1}c_{p+2,p+3}\cdots c_{r-2,r-1}.
\end{gather*}
Then we have  $f^-_pf^+_p=1_p$.

For even $r>0$, define $f^+_0\in B^+(0,r)$ and $f^-_0\in B^-(r,0)$ by
\begin{gather*}
    f^+_0 =\overline{c_{r-1,r}}\cdots \overline{c_{3,4}}\,\overline{c_{1,2}},\\
    f^-_0 =\delta^{-r/2} c_{1,2}c_{3,4}\cdots c_{r-1,r},
\end{gather*}
where the latter is well defined since $\delta$ is invertible.
Then we have  $f^-_0f^+_0=1_0$.

Thus, in each case, each object $p$ of $B_{\preceq r}$ is a retract of the object~$r$.
Hence $B_{\preceq r}$ and $B(r)$ are Morita equivalent.
\end{proof}

We obtain the following Morita equivalence for the Brauer algebra. 

\begin{theorem}[Cf. K\"onig--Xi \cite{Koenig-Xi}]
    Let $\k$ be a commutative ring. Let $r\ge 0$. 
    Suppose that $\delta$ is non-singular at ranks $\preceq r$.
    Then the Brauer algebra $B(r)$ is Morita equivalent to the algebra $\bigoplus_{q\preceq r}\S_q$.
\end{theorem}

\begin{proof}
    For even $r>0$, if $\delta$ is non-singular at ranks $\preceq r$, then $\delta$ is invertible by Proposition~\ref{Rui-Si-singularset2}.
    Therefore, by Lemma~\ref{BraueralgebraandtruncatedBrauer}, the Brauer algebra $B(r)$ is Morita equivalent to the truncated Brauer category $B_{\preceq r}$, which is Morita equivalent to the category $\S_{\preceq r}$ by Theorem~\ref{MoritaequivalencetruncatedBrauer}.
    Since $\S_{\preceq r}$ is Morita equivalent to the algebra $\bigoplus_{q\preceq r}\S_q$, the statement follows.
\end{proof}

\begin{remark}
  The Morita equivalence for the Brauer algebra $B(r)$ over a field has been proved by K\"onig and Xi \cite[Theorem 7.3]{Koenig-Xi}. (See Proposition~\ref{BMoritaequivalenttoS-field}.)
\end{remark}

If $\k$ is a field of characteristic $0$, then we have the following.

\begin{proposition}
\label{BMoritaequivalenttoS-fieldofchar0}
Let $\k$ be a field of characteristic $0$, and let $r\ge 0$.
Then $B_{\preceq r}$ is Morita equivalent to $\S_{\preceq r}$ if and only if $\delta$ is non-singular at ranks $\preceq r$.
\end{proposition}

\begin{proof}
The ``if'' part follows from Theorem~\ref{MoritaequivalencetruncatedBrauer}.
Let us prove the ``only if'' part.
By the Yoneda lemma, for each $q\preceq r$, we have $\End_{B\mmod}(B(q,-))\cong B(q)^{\op}$, which is finite-dimensional.
For an $\S$-module $M$, if $\End_{\S\mmod}(M)$ is finite-dimensional, then $\End_{\S\mmod}(M)$ is semisimple.
Therefore, since $\S\mmod$ and $B\mmod$ are equivalent, $B(q)$ is a semisimple algebra.
Hence, by Remark~\ref{remarkKoenig-Xi}, $\delta$ is non-singular at ranks $\preceq r$.
\end{proof}

\subsection{Center $Z(B)$ of $B$}

We now apply the result of Section~\ref{sectioncenterC} to the Brauer category.
By Theorems~\ref{centerofC} and~\ref{Brauerperfect}, we obtain the following.

\begin{theorem}\label{centerofB}
 Let $\k$ be any commutative ring and $\delta\in\k$.
 If $\delta$ is non-singular, then we have a $\k$-algebra isomorphism
  \begin{gather*}
      \zeta_{\tau}:Z(\S)\to Z(B)
  \end{gather*}
  defined for $f=(f_q)_q\in Z(\S)$, $p\in \Ob(B)$ by 
  \begin{gather*}
      (\zeta_{\tau}(f))_p=\sum_{q\preceq p}d^+_{p,q}\, f_q\, d^-(1_q)\, d^-_{p,q}\in B(p)
  \end{gather*}
  where we write $d(1_p)=\sum_{q\preceq p}d^+_{p,q}\ot d^-_{p,q}$.
\end{theorem}

The $\k$-algebra isomorphism $\zeta_{\tau}:Z(\S)\to Z(B)$ can be written as
\begin{gather}\label{zetataufpsi}
    (\zeta_{\tau}(f))_p
    =\sum_{q\preceq p}\sum_{g\in \B^{+,m}(q,p)}\psi (g)f_q\bar{g}
\end{gather}
for $f=(f_q)_q\in Z(\S)$, $p\in \Ob(B)$
since we have
\begin{gather*}
\begin{split}
(\zeta_{\tau}(f))_p
&=\sum_{q\preceq p}d^+_{p,q}\, f_q\, d^-(1_q)\, d^-_{p,q}\\
&=F_d^{-1}(B^+\ot f\ot B^-)d(1_p)\\
&=L^{-1}LF_d^{-1}(B^+\ot f\ot B^-)d(1_p)\\
&=L^{-1}K_{b,b'}(B^+\ot f\ot B^-)d(1_p)\\
&=L^{-1}(B^+\ot f\ot B^-)d'(1_p)\\
&=(B^+\ot \mu^-)(\psi \ot B^-)(B^+\ot f\ot B^-)d'(1_p) \quad (\text{by } \eqref{eq-Linverse})\\
&=\sum_{q\preceq p}\sum_{g\in \B^{+,m}(q,p)}\psi (g)f_q\bar{g}.
\end{split}
\end{gather*}

In a way similar to Theorem~\ref{centerofB}, by Theorems~\ref{centerofC} and~\ref{Brauerperfectnew}, if $\delta$ is non-singular at ranks $\preceq p$, then
we obtain the composition $\zeta_{\tau,\preceq p}$ of the following $\k$-algebra isomorphisms 
\begin{gather}\label{zetataup}
    \zeta_{\tau,\preceq p}:
    \bigoplus_{q\preceq p}Z(\S_q)=Z(\bigoplus_{q\preceq p}\S_{q})\cong Z(\S_{\preceq p})\xrightarrow[\cong]{\zeta_{\tau}|_{\preceq p}} Z(B_{\preceq p})\cong Z(B(p)).
\end{gather}

\subsection{Canonical central idempotent $\varepsilon^{(p)}$ of the Brauer algebra $B(p)$}\label{secepsilonpBp}

In this subsection, we study the canonical central idempotent of the Brauer algebra and the family of canonical central idempotents of the (truncated) Brauer category over any commutative ring $\k$ with $\delta\in\k$. 

For the Brauer algebra $B(p)$, the notion of a canonical central idempotent is the same as the splitting idempotent in the sense of King, Martin and Parker \cite{King-Martin-Parker}, which is an element $\varepsilon^{(p)}\in B(p)$ such that
\begin{enumerate}
    \item $\varepsilon^{(p)}\in 1_p+ J_p$,
    \item $\varepsilon^{(p)}J_p=0=J_p\varepsilon^{(p)}$.
\end{enumerate}
Here $J_p=\bigoplus_{q\prec p}B^+(q,p)B^-(p,q)\subset B(p)$.
A canonical central idempotent is unique if it exists.

A family of canonical central idempotents for the truncated Brauer category $B_{\preceq r}$, or the Brauer category $B$, is simply a family $\{\varepsilon^{(p)}\}_p$ of splitting idempotents $\varepsilon^{(p)}$ of $B(p)$ for objects $p$.

By Theorems~\ref{perfectcci},~\ref{Brauerperfectnew} and~\ref{Brauerperfect}, we have the following.

\begin{proposition} 
Let $r\ge0$. Then the truncated Brauer category $B_{\preceq r}$ admits a family of canonical central idempotents if and only if $\delta$ is non-singular at ranks $\preceq r$.
Similarly, the Brauer category $B$ admits a family of canonical central idempotents if and only if $\delta$ is non-singular.
In both cases, if there is such a family of canonical central idempotents, then it is unique.
\end{proposition}

Since the family $\{\varepsilon^{(p)}\}_{p\ge0}$ of canonical central idempotents $\varepsilon^{(p)}$ is the stack-inverse of the copairing $d$ of $b$, we obtain the following formula for the canonical central idempotents.

\begin{proposition}\label{varepsilonpB}
We have
\begin{gather*}
    \varepsilon^{(p)}=d^-(1_p)=\zeta_\tau(1_p)_p= \psi (1_p).
\end{gather*}    
\end{proposition}

\begin{remark}
King, Martin and Parker \cite{King-Martin-Parker} constructed the canonical central idempotent $\varepsilon^{(p)}$ for the Brauer algebra by studying the subalgebra of the Brauer algebra over the rational function field $\Q(\delta)$ that is the $K$-algebra generated by $(i,i+1)$ and $\frac{1}{\delta}c_{i,i+1}\overline{c_{i,i+1}}$ with $1\le i\le p-1$, where $K=\{f/g\mid f,g\in \Z[\delta], g \text{ monic }, \deg(f)\le \deg(g)\}$.
\end{remark}

\subsection{Primitive central idempotents of the Brauer algebra $B(p)$ over a field of characteristic $0$}\label{sectionpci}

In this subsection, we work over a field $\k$ of characteristic $0$. 
Let $p\ge 0$.
Let $\delta\in\k$ be non-singular at ranks $\preceq p$.
For each $q\preceq p$, the set $\{e_\lambda\mid \lambda \vdash q\}$ forms the complete set of primitive central idempotents of $\S_q=\k[\gpS_q]$ (see Section~\ref{subsectionmB}). 
For each $\lambda\vdash q$, set 
$$\varepsilon^{(p)}_\lambda:=\zeta_{\tau,\preceq p}(e_\lambda)=(\zeta_{\tau}(e_{\lambda}))_p=d^+_{p,q} e_\lambda d^-(1_q) d^-_{p,q}=L^{-1}(e^{(p)}_\lambda)\in Z(B(p)),$$
where $\zeta_{\tau,\preceq p}$ is the $\k$-algebra isomorphism defined in \eqref{zetataup}.
Then we obtain the following.

\begin{proposition} \label{pciforBp}
Let $\k$ be a field of characteristic $0$, and let $\delta\in\k$ be non-singular at ranks $\preceq p$.
Then the set
\begin{gather*}
     \{\varepsilon^{(p)}_\lambda\mid  \lambda\vdash q, q\preceq p\}
\end{gather*}
forms the complete set of primitive central idempotents of the Brauer algebra $B(p)$ over $\k$.
\end{proposition}

\begin{remark}
Doty, Lauve and Seelinger \cite{DLS} worked over a field $\k$ of characteristic $0$ and $\delta\in \k\setminus \Z$, so that the Brauer algebra is split semisimple. They provided a recursive description of the complete set $\{\varepsilon^{(p)}_\lambda\}$ of primitive central idempotents by showing that Brauer algebras form a multiplicity-free family and identifying a Jucys--Murphy sequence for this family.
\end{remark}

We give several useful formulas for the primitive central idempotents $\varepsilon_\lambda^{(p)}$ of the Brauer algebra $B(p)$.
It follows from \eqref{zetataufpsi} that for $\lambda\vdash q, q\preceq p$, we have
\begin{gather}\label{pcipsitilde}
\varepsilon^{(p)}_\lambda
=\sum_{f\in \B^{+,m}(q,p)}\psi (f) e_{\lambda}\bar{f}.
\end{gather}
In particular, for $\lambda\vdash p$, we have
\begin{gather}\label{centralidempotentdecomp}
     \varepsilon^{(p)}_{\lambda}= \psi (1_p) e_{\lambda}=\varepsilon^{(p)}e_{\lambda}
\end{gather}
since $\B^{+,m}(p,p)=\{1_p\}$ and by Proposition~\ref{varepsilonpB}.
This implies that we have
\begin{gather*}
   \varepsilon^{(p)}=\sum_{\lambda\vdash p}\varepsilon^{(p)}_\lambda.
\end{gather*}
For $\lambda\vdash q$, $q\preceq p$, we have
\begin{gather}\label{epsilonplambdaJk}
    \varepsilon^{(p)}_{\lambda}\in B^+(q,p)\varepsilon^{(q)}B^-(p,q)\subset \bigoplus_{r\preceq q}B^+(r,p)B^-(p,r)
\end{gather}
since we have
\begin{gather*}
\begin{split}
    \varepsilon^{(p)}_{\lambda}
    &=\sum_{f\in \B^{+,m}(q,p)}\psi (f) e_{\lambda}\bar{f} \quad  (\text{by }  \eqref{pcipsitilde})\\
    &=\sum_{f\in \B^{+,m}(q,p)}\psi (f)\psi (1_q) e_{\lambda}\bar{f} \quad (\text{by }  \eqref{psifpsi1})\\
    &=\sum_{f\in \B^{+,m}(q,p)}\psi (f)\varepsilon^{(q)} e_{\lambda}\bar{f} \quad (\text{by Proposition }~\ref{varepsilonpB})
\end{split}
\end{gather*}
and $\psi (f)\varepsilon^{(q)}\in B^+(q,p)\varepsilon^{(q)}$ by the property $J_q \varepsilon^{(q)}=0$.

We describe the primitive central idempotents $\varepsilon^{(p)}_\lambda$ of $B(p)$ for $p\le 3$:
\begin{gather*}
    \begin{split}
        \varepsilon^{(0)}_\emptyset&=1_0,\quad \varepsilon^{(1)}_{1}=1_1,\\
        \varepsilon^{(2)}_2&=e_2-\frac{1}{\delta}\overline{c_{1,2}}c_{1,2},\quad 
        \varepsilon^{(2)}_{1^2}=e_{1^2},\quad 
        \varepsilon^{(2)}_\emptyset=\frac{1}{\delta}\overline{c_{1,2}}c_{1,2},\\
        \varepsilon^{(3)}_3&=e_{3}-\frac{1}{3(\delta+2)}(\overline{c_{1,2}}+\overline{c_{1,3}}+\overline{c_{2,3}})(c_{1,2}+c_{1,3}+c_{2,3}),\\
        \varepsilon^{(3)}_{21}&=e_{21}-\frac{2}{3(\delta-1)}(\overline{c_{1,2}}c_{1,2}+\overline{c_{1,3}}c_{1,3}+\overline{c_{2,3}}c_{2,3})\\
        &\quad +\frac{1}{3(\delta-1)}(\overline{c_{1,2}}(c_{1,3}+c_{2,3})+ \overline{c_{1,3}}(c_{1,2}+c_{2,3})+ \overline{c_{2,3}}(c_{1,2}+c_{1,3})),\\
        \varepsilon^{(3)}_{1^3}&=e_{1^3},\\
        \varepsilon^{(3)}_1&=\frac{\delta+1}{(\delta-1)(\delta+2)}(\overline{c_{1,2}}c_{1,2}+\overline{c_{1,3}}c_{1,3}+\overline{c_{2,3}}c_{2,3})\\
        &\quad -\frac{1}{(\delta-1)(\delta+2)}(\overline{c_{1,2}}(c_{1,3}+c_{2,3})+ \overline{c_{1,3}}(c_{1,2}+c_{2,3})+ \overline{c_{2,3}}(c_{1,2}+c_{1,3})).
    \end{split}
\end{gather*}
The computations of the elements $\varepsilon^{(p)}_\lambda$ for $p\le 3$ above have been given in \cite{DLS}, using Jucys--Murphy elements.

\section{Applications to tensor spaces}
\label{sec:tensor-spaces}
We apply the results in the preceding sections to the action of the Brauer category on the tensor spaces $V^{\ot q}$.
\subsection{Vector space with non-degenerate bilinear form}

Let $\k$ be a field of characteristic $0$.
Let $V$ be a finite-dimensional vector space equipped with a non-degenerate bilinear form
\[\omega: V\times V\to \k\]
that is either symmetric or skew-symmetric.
In the skew-symmetric case, the dimension of $V$ is even.
Set
\[
e = 
\begin{cases}
    +1&\text{if $\omega$ is symmetric},\\
    -1&\text{if $\omega$ is skew-symmetric}.
\end{cases}
\]

Let $G=G(V)\subset \GL(V)$ denote the subgroup of $\GL(V)$ preserving $\omega$.
Thus, we have $G=\O(V)$ if $e=1$, and $G=\Sp(V)$ if $e=-1$.

Let $\omega^*\in V\ot V$ be the unique tensor such that
\[
(\id_V\ot \omega)(\omega^*\ot u)=u
\]
for each $u\in V$.

In the following, we consider the Brauer category $B=B^{\k,\delta}$, where
\begin{gather*}
    \delta = \omega\circ \omega^* = e\dim V.
\end{gather*}

\subsection{Action of $B$ on tensor spaces}

As is well known \cite{Lehrer-Zhang}, the Brauer category $B=B^{\k,\delta}$ acts on the tensor powers $V^{\ot r}$, $r\ge0$.
We have a $\k$-linear map
\begin{gather*}
    \alpha=\alpha_{p,q}  :  B(p,q)\otimes V^{\ot p}\to V^{\ot q}
\end{gather*}
defined as follows.
Let $f\in \B(p,q)$ be a Brauer diagram.
Recall that $f$ is a partition of $[p]^+\sqcup[q]^-=\{1^+,\dots,p^+,1^-,\dots,q^-\}$ into unordered pairs.
For
\[
f=\{\{i_1^+,j_1^+\},\dots,\{i_s^+,j_s^+\},\{k_1^+,l_1^-\},\dots,\{k_t^+,l_t^-\},\{m_1^-,n_1^-\},\dots,\{m_u^-,n_u^-\}\},
\]
with $2s+t=p, 2u+t=q$,
    $i_a<j_a\ (1\le a\le s)$ and $
    m_c<n_c\ (1\le c\le u)$ and $v=v_1\otimes\cdots\otimes v_p\in V^{\otimes p}$, we set
\[
\alpha_{p,q}(f\otimes v)
=e^{C(f)}
\Bigl(\prod_{a=1}^s \omega(v_{i_a},v_{j_a})\Bigr)
\sum\; w_1\otimes\cdots\otimes w_q\in V^{\otimes q},
\]
where $w_{l_b}=v_{k_b}$ for $1\le b\le t$, and $\sum w_{m_c}\ot w_{n_c}=\omega^*$ for $1\le c\le u$.
Here $C(f)$ is the number of crossings in the standard drawing of $f$; only its parity matters in
$e^{C(f)}$.
For example, for $f=\{\{1^+,2^+\},\{3^+,2^-\},\{1^-,3^-\}\}\in\B(3,3)$, we have
\[
\alpha_{3,3}(f\ot(v_1\ot v_2\ot v_3)) = e^1 \omega(v_1,v_2)\sum \omega'\ot v_3\ot \omega'',
\]
where we write $\sum \omega'\ot\omega''=\omega^*$.

Let $G(V)\mmod$ be the category of $G(V)$-modules.
The linear maps $\alpha_{p,q}$ give rise to a $\k$-linear functor $$F^V:B \to G(V)\mmod$$
as follows. For an object $p\in \Ob(B)$, set $F^V(p)=V^{\ot p}$, and for a morphism $f\in B(p,q)$, set $F^V(f)(v)=\alpha_{p,q}(f\ot v)$.

We also write $f\cdot v=F^V(f)(v)$.

\subsection{Traceless tensors}
Let $p\ge0$.
The \emph{traceless part} $V^{\lara p}$ of $V^{\ot p}$ is defined by
\[
V^{\lara p}=\bigcap_{1\le i<j\le p}\ker (F^V(c_{i,j}):V^{\ot p}\to V^{\ot p-2}),
\]
which is a $G(V)$-submodule of $V^{\ot p}$.
There exists a unique $G(V)$-equivariant projection $$\pi^{\lara p}:V^{\ot p}\to V^{\lara p}$$
via Weyl's decomposition of $V^{\ot p}$ into irreducible $G(V)$-modules.
For this we use the well-known fact that 
$V^{\ot p}$ is a completely reducible $G(V)$-module with direct sum decomposition $V^{\ot p}=W\oplus V^{\lara p}$ where $W$ and $V^{\lara p}$ have no common irreducible component: the irreducible components of $V^{\lara p}$ are
indexed by partitions $\lambda\vdash p$
while the irreducible components of $W$ are indexed by partitions $\lambda\vdash p'$ with $p'\prec p$; see \cite{Weyl}.

In the stable range, the projection $\pi^{\lara p}$ can be constructed using the canonical central idempotent $\varepsilon^{(p)}$ as follows.

\begin{proposition}\label{projection-traceless}
Let $p\ge 0$.
Suppose that $\delta=e\dim V$ is non-singular at ranks $\preceq p$.
Then we have
\[F^V(\varepsilon^{(p)})=i\pi^{\lara p}:V^{\ot p}\to V^{\ot p}\]
where $i:V^{\lara p}\to V^{\ot p}$ is the inclusion.
\end{proposition}

\begin{proof}
    Since $\varepsilon^{(p)}\in B(p)$ is an idempotent, $F^V(\varepsilon^{(p)})$ is a projection of $V^{\ot p}$ onto its $\k[G(V)]$-submodule.
    Since $\varepsilon^{(p)}=\psi(1_p)$ satisfies the contraction property $c_{i,j}\varepsilon^{(p)}=0$ for any $1\le i<j\le p$, it follows that $F^V(\varepsilon^{(p)})(V^{\ot p})\subset V^{\lara p}$.

    To prove that $V^{\lara p}\subset F^V(\varepsilon^{(p)})(V^{\ot p})$, let $v\in V^{\lara p}$. We will prove $\varepsilon^{(p)}\cdot v=v$.
    Recall $1_p=\sum_{r\preceq p}\sum_{f\in\B^{+,m}(r,p)}\psi(f)\overline{f}$.
    The term with $r\prec p$ maps $v$ to $0$ since $v$ is traceless. The term with $r=p$ is $\psi(1_p)=\varepsilon^{(p)}$.
    Therefore, $v=1_p\cdot v=\varepsilon^{(p)}\cdot v$, which completes the proof.
    \end{proof}

\begin{remark}
    Proposition~\ref{projection-traceless} is not new. For instance, Bulgakova, Goncharov and Helpin \cite{BGH} study the traceless part as the image of the endomorphism on $V^{\otimes p}$ given by the splitting idempotent $\varepsilon^{(p)}$ for the Brauer algebra $B(p)$.
\end{remark}

\subsection{The functor $\widetilde{F}^V$}

Using the functor $\widehat{\mB^+}: \mB\to \S^{\op}\mmod$ defined in Section~\ref{subsectionmB},
we obtain a functor
\[
{}_BB^+=\widehat{\mB^+}\circ L:B\to \S^\op\mmod,
\]
which we may regard as a left action of $B$ on $B^+$.
Here, we have ${}_BB^+(q)=B^+(-,q)$ for $q\ge0$ and for a morphism $g\in B(q,q')$, we have
\[
{}_BB^+(g):B^+(-,q)\to B^+(-,q')
\]
defined for $p\ge 0$ and $f\in B^+(p,q)$ by
\[
{}_BB^+(g)(f)=(B^+\ot b')(L_{\tau,b'}(g)\ot f)
        =\sum_{h\in \B^{+,m}(p,q')}h\cdot b'(\pi^-(\overline{h}g)\otimes f),
\]
where $\pi^-= \varepsilon^+\ot\id_{B^-}:B(q,p)\to B^-(q,p)$ is the projection.

We also write $g\,\sharp\, f={}_BB^+(g)(f)$.

Define a $\k$-linear functor
$$\widetilde{F}^V:B\to G(V)\mmod$$
as follows: for an object $q\in \Ob(B)$, set 
$$\widetilde{F}^V(q)=\bigoplus_{p\preceq q}B^+(p,q)\otimes_{\S_p}V^{\lara p},$$
and for a morphism $g\in B(q,q')$, set 
$$\widetilde{F}^V(g)(\sum f\otimes v)=\sum (g\,\sharp\, f)\ot v$$
for $\sum f\ot v\in \bigoplus_{p\preceq q}B^+(p,q)\otimes_{\S_p}V^{\lara p}$.

\subsection{The maps $\Phi^V_q$ and $\Psi^V_q$}

For $q\ge0$, define a $\k[G(V)]$-module map
\begin{gather}
\label{eq:PhiVq}
    \Phi^V_q:V^{\ot q} \to \bigoplus_{p\preceq q}B^+(p,q)\otimes_{\S_p}V^{\lara p}
\end{gather}
    for $v\in V^{\ot q}$ by
    \begin{gather*}
        \Phi^V_q(v)=
        \sum_{p\preceq q}\sum_{f\in \B^{+,m}(p,q)} f\ot \pi^{\lara p}(\overline{f}\cdot v).
    \end{gather*}

\begin{lemma}
    The map $\Phi^V_q$ is injective.    
\end{lemma}

\begin{proof}
Let $\langle\ ,\ \rangle:V^{\ot q}\times V^{\ot q}\to\k$ be the non-degenerate bilinear form defined by
\[\langle u_1\ot\dots\ot u_q,v_1\ot\dots\ot v_q\rangle
=\omega(u_1,v_1)\cdots\omega(u_q,v_q).\]
For $f\in\B^+(p,q)$, $u\in V^{\ot p}$ and $v\in V^{\ot q}$, we have
$\langle f\cdot u,v\rangle=\pm\langle u,\overline{f}\cdot v\rangle$.
In particular, $V^{\lara p}$ is orthogonal to $g\cdot V^{\ot(p-2)}$ for every $g\in\B^+(p-2,p)$, so that $\langle t,w\rangle=\langle t,\pi^{\lara p}(w)\rangle$ for $t\in V^{\lara p}$ and $w\in V^{\ot p}$.
Hence, for $t\in V^{\lara p}$ and $v\in V^{\ot q}$, we have
\[
\langle f\cdot t,v\rangle=\pm\langle t,\pi^{\lara p}(\overline f\cdot v)\rangle .
\]
If $\Phi^V_q(v)=0$, then $\pi^{\lara p}(\overline f\cdot v)=0$ for all $p\preceq q$ and $f\in\B^{+,m}(p,q)$, and hence $v$ is orthogonal to $\sum_{p\preceq q}\sum_{f\in\B^{+,m}(p,q)}f\cdot V^{\lara p}=V^{\ot q}$ (see e.g. \cite[Lemma 17.15]{Fulton--Harris}).
Therefore $v=0$.
\end{proof}

\begin{lemma}
        The maps $\Phi^V_q$, $q\ge0$, form a natural transformation
        \[
        \Phi^V : F^V \to \widetilde{F}^V: B\to G(V)\mmod.
        \]
\end{lemma}

\begin{proof}
        We need to check that 
    \begin{gather}\label{PhiqBq-modulemap}
        \Phi^V_{q'}(g\cdot v)=\widetilde{F}^V(g)\cdot \Phi^V_q(v)
    \end{gather}
    for any $g\in B(q,q')$ and $v\in V^{\ot q}$.
    The left-hand side is
    \begin{gather*}
        \sum_{p\preceq q'}\sum_{f\in \B^{+,m}(p,q')} f\ot \pi^{\lara p}(\overline{f}g\cdot v)=\sum_{p\preceq q'}\sum_{f\in \B^{+,m}(p,q')} f\ot \pi^{\lara p}(\pi^-(\overline{f}g)\cdot v),
    \end{gather*}
    and the right-hand side is 
    \begin{gather*}
        \sum_{p\preceq q}\sum_{f\in \B^{+,m}(p,q)} (g\,\sharp\, f)\ot \pi^{\lara p}(\overline{f}\cdot v).
    \end{gather*}
    Since we have 
    \begin{gather*}
        \sum_{p\preceq q}\sum_{f\in \B^{+,m}(p,q)}(g\,\sharp\, f)\ot \overline{f}
        =L_{\tau,b'}(g)=\sum_{p\preceq q'}\sum_{f\in \B^{+,m}(p,q')} f\ot\pi^-(\overline{f}g),
    \end{gather*}
    the left-hand side of \eqref{PhiqBq-modulemap} coincides with the right-hand side of \eqref{PhiqBq-modulemap}.
\end{proof}

The following is an explicit functorial form of Weyl's 
decomposition of the tensor space $V^{\otimes q}$ \cite{Weyl}.

\begin{theorem}\label{Phi-Psi}
If $\delta$ is non-singular at ranks $\preceq q$,
then the map $\Phi^V _q$ is an isomorphism
with its inverse
\[
\Psi^V_q : \bigoplus_{p\preceq q}B^+(p,q)\ot_{\S_p}V^{\lara p}\to V^{\ot q}
\]
given by
\[
\Psi^V_q(f\ot v)=\psi(f)\cdot v
\]
for $f\in B^+(p,q)$ and $v\in V^{\lara p}$.
Consequently, for $r\ge 0$, if $\delta$ is non-singular at ranks $\preceq r$, then the restriction $$\Phi^V_{\preceq r}:F^V|_{B_{\preceq r}}\to \widetilde{F}^V|_{B_{\preceq r}}$$
of $\Phi^V$ to $B_{\preceq r}$ is a natural isomorphism.
\end{theorem}

\begin{proof}
It suffices to check $\Psi^V_q\Phi^V_q=\id$ and $\Phi^V_q\Psi^V_q=\id$.
    For $v\in V^{\ot q}$, we have
    \begin{gather*}
        \begin{split}
            \Psi^V_q\Phi^V_q(v)
            &=\Psi^V_q(\sum_{p\preceq q}\sum_{f\in \B^{+,m}(p,q)} f\ot \pi^{\lara p}(\overline{f}\cdot v))\\
            &=\Psi^V_q(\sum_{p\preceq q}\sum_{f\in \B^{+,m}(p,q)} f\ot (\varepsilon^{(p)}\overline{f}\cdot v)) \quad (\text{by Prop.~\ref{projection-traceless}})\\
            &=\sum_{p\preceq q}\sum_{f\in \B^{+,m}(p,q)} \psi(f)\varepsilon^{(p)}\overline{f}\cdot v\\
            &=\sum_{p\preceq q}\sum_{f\in \B^{+,m}(p,q)} \psi(f) \overline{f}\cdot v \quad (\text{by \eqref{psifpsi1} and Prop.~\ref{varepsilonpB}})\\
            &=v \quad (\text{by the decomposition of the identity morphisms}).
        \end{split}
    \end{gather*}
    Hence we have $\Psi^V_q\Phi^V_q=\id$.
    We also have for $s\preceq q$, $g\in \B^+(s,q)$, $v\in V^{\lara s}$,
    \begin{gather*}
        \begin{split}
           \Phi^V_q\Psi^V_q(g\ot v)
           &= \Phi^V_q( \psi(g)\cdot v)\\
           &=\sum_{p\preceq q}\sum_{f\in \B^{+,m}(p,q)} f\ot \pi^{\lara p}(\overline{f}\psi(g)\cdot v)\\
           &=\sum_{p\preceq q}\sum_{f\in \B^{+,m}(p,q)} f\ot (\psi(1_p)\overline{f}\psi(g)\cdot v).
        \end{split}
    \end{gather*}
    By the contraction properties of $\psi(1_p)$ and $\psi(g)$, we have $\psi(1_p)\overline{f}\psi(g)=0$ unless $p=s,f=g$, and $\psi(1_s)\overline{g}\psi(g)=\psi(1_s)=\varepsilon^{(s)}$.
    Therefore, we obtain 
    \begin{gather*}
        \Phi^V_q\Psi^V_q(g\ot v)
        =g\ot \varepsilon^{(s)}\cdot v
        =g\ot \pi^{\lara s} (v)
        =g\ot v
    \end{gather*}
    since $v$ is traceless.
    Thus, we have $\Phi^V_q\Psi^V_q=\id$.
\end{proof}

\subsection{Functorial Schur--Weyl duality}

Here we give a functorial reformulation of the decomposition of the tensor space $V^{\ot q}$ into irreducibles.

For a partition $\lambda\vdash p$, let $S^\lambda$ denote the Specht module, which is the irreducible $\S_p$-module corresponding to $\lambda$.
We have the Wedderburn decomposition
\begin{gather}\label{decomposition-Sp}
\S_p
\cong \bigoplus_{\lambda\vdash p} \End_\k(S^\lambda)
\cong \bigoplus_{\lambda\vdash p} S^\lambda\ot_\k (S^\lambda)^*,
\end{gather}
which is a canonical isomorphism of $\S_p$-bimodules.

For $e=+1,-1$, define $\k_e$ by $\k_{+1}=\k$, $\k_{-1}=\sgn$, and set
$S_e^\lambda=S^\lambda\ot\k_e$. 
Note that if $e=-1$, then $S_{-1}^\lambda\cong S^{\lambda'}$, where $\lambda'$ denotes the transpose of the partition $\lambda$.
The decomposition \eqref{decomposition-Sp} induces a decomposition
\begin{gather*}
\S_p
\cong \bigoplus_{\lambda\vdash p} S_e^\lambda\ot_\k (S_e^\lambda)^*.
\end{gather*}
We set
\[
V^{\lara{\lambda}}=(S_e^\lambda)^* \ot_{\S_p}V^{\lara p},
\]
which is the irreducible representation of $G(V)$ corresponding to the partition $\lambda$ if it is non-zero.

Set
\begin{gather*}
{}_B B^+(\lambda,-)={}_BB^+(p,-)\otimes_{\S_p} S_e^\lambda.
\end{gather*}
Each nonzero component ${}_BB^+(\lambda,q)$ of the $B$-module ${}_BB^+(\lambda,-)$ is a simple $B(q)$-module if $\delta$ is non-singular at ranks $\preceq q$.

Define a $\k$-linear functor
\begin{gather*}
    \widehat{F}^V: B\to G(V)\mmod
\end{gather*}
as follows: for an object $q\in \Ob(B)$, set
\begin{gather*}
     \widehat{F}^V(q)=\bigoplus_{p\preceq q}\bigoplus_{\lambda\vdash p}{}_BB^+(\lambda,q)\ot_\k  V^{\lara \lambda},
\end{gather*}
and for a morphism $g\in B(q,q')$, set 
\begin{gather*}
    \widehat{F}^V(g)(\sum f\otimes v)=\widetilde{F}^V(g)(\sum f\otimes v)=\sum (g\,\sharp\, f)\ot v
\end{gather*}
for $\sum f\ot v\in \bigoplus_{p\preceq q}\bigoplus_{\lambda\vdash p}{}_B B^+(\lambda,q)\otimes_{\k}V^{\lara \lambda}$.
Then we have a chain of canonical natural isomorphisms
\begin{gather*}
    \begin{split}
    \widetilde{F}^V(q)
        &=\bigoplus_{p\preceq q}{}_BB^+(p,q)\ot_{\S_p} V^{\lara p}\\
        &\cong\bigoplus_{p\preceq q}{}_BB^+(p,q)\ot_{\S_p}{\S_p}\ot_{\S_p} V^{\lara p}\\
        &\cong\bigoplus_{p\preceq q}\bigoplus_{\lambda\vdash p}{}_BB^+(p,q)\ot_{\S_p}S_e^\lambda\ot_\k (S_e^{\lambda})^*\ot_{\S_p} V^{\lara p}\\
        &\cong\bigoplus_{p\preceq q}\bigoplus_{\lambda\vdash p}{}_BB^+(\lambda,q)\ot_\k  V^{\lara \lambda}\\
        &=\widehat{F}^V(q).
    \end{split}
\end{gather*}
The above argument and Theorem~\ref{Phi-Psi} imply the following.
\begin{proposition}[Categorical Schur--Weyl decomposition of $V^{\ot q}$]\label{prop:SWB}    
 We have a canonical natural transformation
 \begin{gather*}
       \widehat{\Phi}^V: F^V \to \widehat{F}^V: B\to G(V)\mmod,
    \end{gather*}
where for $q$ such that $\delta$ is non-singular at ranks $\preceq q$, the $q$-component of $\widehat{\Phi}^V$
\begin{gather*}
       \widehat{\Phi}^V_q: V^{\ot q} \to
        \bigoplus_{p\preceq q}
        \bigoplus_{\lambda\vdash p}
        {}_BB^+(\lambda,q)\ot_\k V^{\lara{\lambda}}
\end{gather*}
is a $B(q)\ot\k[G(V)]$-module isomorphism.
\end{proposition}

\section{Base replacements of quasi-cellular categories}
\label{sectionBasereplacement}

In this section, we first consider the Karoubi envelope of a quasi-cellular category, and then consider the special case of the Brauer categories.

\subsection{Base replacements of quasi-cellular categories}
The \emph{Karoubi envelope} or the \emph{idempotent completion} $\Kar(\C)$ of a linear category $\C$ is defined as follows.
The objects of $\Kar(\C)$ are pairs $(x,e)$ with $x$ an object in $\C$ and $e:x\to x$ an idempotent.
The hom-spaces are 
defined by $\Kar(\C)((x,e),(x',e'))=e'\C(x,x')e \subset \C(x,x')$.
The Karoubi envelope $\Kar(\C)$ is an idempotent-complete linear category, and it has $\C$ as a linear full subcategory.

A stratified structure, and likewise a quasi-cellular structure, on a linear category $\C$ does not necessarily extend to its Karoubi envelope $\Kar(\C)$.  Nevertheless, for a suitable collection of idempotents in $\C$, the associated full subcategory of $\Kar(\C)$ retains the corresponding structure.

\begin{proposition}\label{Karoubi-q-c}
Let $\k$ be a commutative ring.
    Let $\C$ be a stratified $\k$-linear category over a base category $\C^0$.
    For each $x\in \Ob(\C)$, let $E_x$ be a finite set of (necessarily orthogonal) idempotents in $\C^0(x,x)$ such that the algebra $\C^0(x,x)$ has a direct sum decomposition
    \begin{gather}
        \label{C0-dir-sum}
\C^0(x,x)=\bigoplus_{e\in E_x}\C^0(x,x)\;e\;\C^0(x,x).
    \end{gather}
    Let $\hC$ denote the full subcategory of $\Kar(\C)$ with $\Ob(\hC)=\{(x,e)\mid x\in \Ob(\C),e\in E_x\}$.
    Let $\hC^0$ denote the wide subcategory of $\hC$ with hom-sets $\hC^0((x,e),(x',e'))=e'\C^0(x,x')e$. 
    Then we have the following.
    \begin{enumerate}
\item $\hC^0$ is a base category.
\item $\hC$ has the structure of a stratified linear category
    over $\hC^0$.
\item For $(x,e),(y,e')\in\Ob(\hC)$, we have a direct sum decomposition of $\hC((x,e),(y,e'))$:
\begin{gather}
    \label{decomposition-hC}
    \hC((x,e),(y,e'))=\bigoplus_{z\le x,y} \bigoplus_{e''\in E_z} e'\;\C^+(z,y)\;e''\;\C^-(x,z)\;e.
\end{gather}
\item If, moreover, $\C$ is quasi-cellular, then so is $\hC$.
\item $\hC$ is Morita equivalent to $\C$.
    \end{enumerate}
\end{proposition}

We call the stratified linear category (resp. quasi-cellular category) $\hC$ the \emph{base replacement} of the stratified linear category (resp. quasi-cellular category) $\C$ via the decompositions \eqref{C0-dir-sum}.

\begin{proof}
(1) 
Define a partial order $\le$ on $\Ob(\hC^0)$ such that we have $(x,e)\le (x',e')$ if and only if either $x<x'$ or 
$(x,e)=(x',e')$.
Then $\le$ on $\Ob(\hC^0)$ satisfies the descending chain condition.
Let $(x,e),(x',e')\in \Ob(\hC^0)$ with $(x,e)\neq (x',e')$.
If $x\neq x'$, then $\C^0(x,x')=0$ implies that we have $\hC^0((x,e),(x',e'))=0$.
If $x=x'$ and $e\neq e'$, then we have 
\[
\hC^0((x,e),(x',e'))=e'\C^0(x,x)e\subset
\C^0(x,x)e'\C^0(x,x)\cap\C^0(x,x)e\C^0(x,x)=0.
\]
Therefore, $\hC^0$ is a base category.

(2) For $(x,e),(x',e')\in \Ob(\hC)$, we set $$\hC^{\pm}((x,e),(x',e'))=e'\C^{\pm}(x,x')e\subset\C^{\pm}(x,x').$$
It is easy to check that $\hC^+$ (resp.  $\hC^-$) forms a
wide subcategory of $\hC$ which is 
strongly upward (resp. downward).
By \eqref{C0-dir-sum}, the map
\begin{gather}\label{iso-C0}
    \bigoplus_{e''\in E_z}\C^0(z,z)e''\ot_{e''\C^0(z,z)e''}e''\C^0(z,z)
    \xrightarrow[\cong]{\ \circ\ } \C^0(z,z)
\end{gather}
is an isomorphism for $z\in\Ob(\C)$.
Then, for 
$(x,e),(y,e')\in\Ob(\hC)$,
we have a chain of isomorphisms
\begin{gather}\label{karoubistr}
\begin{split}
   &\bigoplus_{z\in\Ob(\C)}\bigoplus_{e''\in E_z}\hC^+((z,e''),(y,e'))\ot_{\hC^0((z,e''),(z,e''))}\hC^-((x,e),(z,e''))\\
    &=\bigoplus_{z\in\Ob(\C)}\bigoplus_{e''\in E_z}e'\C^+(z,y)e''\ot_{e''\C^0(z,z)e''}e''\C^-(x,z)e \\
    &=\bigoplus_{z\in\Ob(\C)}\bigoplus_{e''\in E_z}e'\C^+(z,y)\C^0(z,z)e''\ot_{e''\C^0(z,z)e''}e''\C^0(z,z)\C^-(x,z)e\\
    &\xrightarrow[\cong]{\eqref{iso-C0}}\bigoplus_{z\in\Ob(\C)} e'\C^+(z,y)\C^0(z,z)\C^-(x,z)e=e'\C(x,y)e=\hC((x,e),(y,e')).
\end{split}
\end{gather}
Thus, $\hC$ has the structure of a stratified linear category over $\hC^0$.

(3) This follows from the above arguments.

(4) Let $\C=(\C,b',d')$ be a quasi-cellular category.
The pairing $$b': \bigoplus_{x\in \Ob(\C)}\C^-(x,y)\ot_{\C^0(x,x)}\C^+(y,x)\to \C^0(y,y)$$
induces a $\hC^0((y,e'),(y,e'))$-bimodule map
$$\hat b':\bigoplus_{(x,e)\in \Ob(\hC)}\hC^-((x,e),(y,e'))\ot_{\hC^0((x,e),(x,e))}\hC^+((y,e'),(x,e))\to \hC^0((y,e'),(y,e'))$$
defined by $\hat b' (e' g e\ot e f e')=e'b'(g e\ot e f)e'$ for $g\in \C^-(x,y), f\in \C^+(y,x)$.
The copairing $d':\C^0(x,x)\to \bigoplus_{y\in \Ob(\C)}\C^+(y,x)\ot_{\C^0(y,y)}\C^-(x,y)\cong \C(x,x)$ of $b'$
induces a $\hC^0((x,e),(x,e))$-bimodule map
$$\hat d':\hC^0((x,e),(x,e))\to \bigoplus_{(y,e')\in \Ob(\hC)}\hC^+((y,e'),(x,e))\ot_{\hC^0((y,e'),(y,e'))}\hC^-((x,e),(y,e'))$$
defined for $f\in \C^0(x,x)$ by $\hat d'(efe)=ed'(f)e\in e\C(x,x)e$, which can be seen as an element of $(\hC^+\ot_{\hC^0}\hC^-)((x,e),(x,e))$ via the isomorphism \eqref{karoubistr}.
We show that $({\hC^+}\ot \hat b')(\hat d'\ot {\hC^+})=\id_{\hC^+}$;
the other identity can be checked similarly.
For $efe'\in \hC^+((y,e'),(x,e))$ with $f\in \C^+(y,x)$, we have
\begin{gather*}
    \begin{split}
    ({\hC^+}\ot \hat b')(\hat d'\ot {\hC^+})(efe')
    &=({\hC^+}\ot \hat b')(ed'(1_x)e\ot efe')\\
    &=\sum_{e''}({\hC^+}\ot \hat b')(e (d')^+_{e''} e''\ot e'' (d')^-_{e''} e \ot efe')\\
    &=\sum_{e''}e (d')^+_{e''} e'' b'((d')^-_{e''} e\ot ef)e'\\
    &=\sum_{e''}e (d')^+_{e''} e''b'(e'' (d')^-_{e''} e\ot efe')\\
    &=efe',
    \end{split}
\end{gather*}
where we write $ed'(1_x)e=\sum_{e''}e (d')^+_{e''} e''\ot e'' (d')^-_{e''} e\in (\hC^+\ot_{\hC^0}\hC^-)((x,e),(x,e))$.
One can check easily that \eqref{qc-eta-epsilon} follows for the present setting.

Therefore, $\hC=(\hC,\hat b',\hat d')$ forms a quasi-cellular category, which completes the proof.

(5)
Let $\add(\Kar(\C))$ be the additive envelope of $\Kar(\C)$.
The inclusion functor $F:\hC\to\add(\Kar(\C))$ is fully faithful.

We show that $F$ is Cauchy dense in the sense given in Lemma~\ref{lem-Morita} below.
Every object of $\add(\Kar(\C))$ is a finite direct sum of objects $(x,p)\in\Kar(\C)$,
and each $(x,p)$ is a direct summand of $(x,1_x)$.
Put $A=\C^0(x,x)$.  
For each $e\in E_x$, let $z_e$ denote the unit of $AeA$.
From \eqref{C0-dir-sum}, it follows that
$1_x=\sum_{e\in E_x} z_e$ with orthogonal central idempotents $z_e\in AeA$; hence
$(x,1_x)\cong \bigoplus_{e\in E_x}(x,z_e)$ in $\add(\Kar(\C))$.

We will show that $(x,z_e)$ is a direct summand of $(x,e)^{\oplus n}$ for some $n$.
Choose a decomposition
$z_e=\sum_{j=1}^n a_j e b_j$ with $a_j,b_j\in A$.
Define morphisms in $\add(\Kar(\C))$
\[
i=(e b_j z_e)_j:(x,z_e)\to (x,e)^{\oplus n},
\qquad
r=(z_ea_je)_j:(x,e)^{\oplus n}\to (x,z_e)
\]
Then
\[
r i
=\sum_{j=1}^n (z_e a_j e)(e b_j z_e)
=\sum_{j=1}^n z_e a_j e b_j z_e
= z_e\Big(\sum_{j=1}^n a_j e b_j\Big)z_e
= z_e^3
= z_e,
\]
so $(x,z_e)$ is a direct summand of $(x,e)^{\oplus n}$ in $\add(\Kar(\C))$.

Therefore, every object of $\add(\Kar(\C))$ is a direct summand
of a finite direct sum of objects in the image of $F$, so $F$ is Cauchy dense.

Thus, by Lemma~\ref{lem-Morita} below, $\hC$ is Morita equivalent to $\add(\Kar(\C))$, and hence to $\C$.
\end{proof}

The following is well known.

\begin{lemma}[See e.g. {\cite[Chap. 5]{Kelly}}]
\label{lem-Morita}
Let $F:\calA\to\calB$ be a fully faithful linear functor between small linear categories such that $\calB$ is additive. 
Suppose that the functor $F$ is \emph{Cauchy dense} in the sense that each object $b$ in $\calB$ is a direct summand of a finite direct sum of objects in the image of $F$.
    Then $F$ induces a Morita equivalence between $\calA$ and $\calB$.
\end{lemma}

\subsection{Deligne's category}
We work over a field $\k$ of characteristic $0$.
Let $B=B^{\k,\delta}$ be the Brauer category with parameter $\delta\in \k$.
Deligne's category $\bRepOd$ may be defined as the Karoubi envelope $\Kar(\add(B))$ of the additive envelope $\add(B)$ of the Brauer category $B$ \cite{Deligne, Deligne-Lie}.

For each partition $\lambda\vdash r$,
we choose a primitive idempotent $f_\lambda\in\S_r$ such that $\S_r f_\lambda\S_r=\S_r e_\lambda\S_r$.
We have a direct sum decomposition into two-sided ideals
\begin{gather}\label{decomposition-Sr}
\S_r = \bigoplus_{\lambda\vdash r}\S_r f_\lambda\S_r .
\end{gather}

Let $\hB$ denote the full subcategory of $\bRepOd=\Kar(\add(B))$ whose
objects are $(r,f_\lambda)$ with $r\ge0$, $\lambda\vdash r$.  We write the object $(r,f_\lambda)$ simply as $\lambda$.
Then,
 for $\lambda\vdash p$ and $\mu\vdash q$,
\[
    \hB(\lambda,\mu)=\hB((p,f_\lambda),(q,f_\mu))=f_\mu B(p,q)f_\lambda .
\]
Let $\hB^0$ denote the wide subcategory of $\hB$ with hom-spaces
\[
\hB^0(\lambda,\mu)=f_\mu B^0(p,q)f_\lambda.
\]

The following proposition follows easily from Proposition~\ref{Karoubi-q-c}.

\begin{proposition}
\label{Deligne-qc}
The category $\hB$ is the base replacement of the quasi-cellular category
$B$ associated with the decompositions \eqref{decomposition-Sr}.  Hence, $\hB$ is a quasi-cellular category.
More precisely, $\hB$ is stratified over the base category $\hB^0$ defined by
\[
    \hB^0(\lambda,\mu)
    =
    f_\mu B^0(p,q)f_\lambda
    =
    f_\mu \S(p,q)f_\lambda .
\]
For $\lambda\vdash p$ and $\mu\vdash q$, the hom-space $\hB(\lambda,\mu)$ has the direct
decomposition
\begin{gather}
\label{hB-decomposition}
\begin{split}
    \hB(\lambda,\mu)
    &=
    \bigoplus_{s\preceq p,q}
    \bigoplus_{\nu\vdash s}
    f_\mu B^+(s,q)f_\nu\, B^-(p,s)f_\lambda\\
    &\cong
    \bigoplus_{s\preceq p,q}
    \bigoplus_{\nu\vdash s}
    f_\mu B^+(s,q)f_\nu\ot_{f_\nu\S_s f_\nu}f_\nu B^-(p,s)f_\lambda.
\end{split}
\end{gather}
Moreover, $\hB$ is Morita equivalent to the Brauer category $B$, and hence to Deligne's category $\Kar(\add(B))=\bRepOd$.
\end{proposition}

\begin{proof}
Apply Proposition~\ref{Karoubi-q-c} to the quasi-cellular category $B$
over the base category $B^0=\S$, with
$E_r=\{f_\lambda\mid \lambda\vdash r\}$.
The decomposition \eqref{decomposition-Sr} is precisely the hypothesis
\eqref{C0-dir-sum}.  The stated decomposition of hom-spaces and the Morita
equivalence follow directly from Proposition~\ref{Karoubi-q-c}.
\end{proof}

\bibliographystyle{plain}
\bibliography{reference}

\end{document}